\documentclass[preprint, 3p, times]{elsarticle}

\usepackage{amsmath}
\usepackage[utf8]{inputenc} 
\usepackage[english]{babel}

\usepackage{amsmath}
\usepackage{amsthm}
\usepackage{amssymb}
\usepackage{bm}
\usepackage[utf8]{inputenc}
\usepackage[english]{babel} 
\usepackage{mathrsfs}
\usepackage{graphicx}
\usepackage{tikz}
\usetikzlibrary{arrows}
\usetikzlibrary{positioning}
\newcommand{\leftiearrow}{\tikz \draw[-triangle 90] (.1,0)--(0,0);}
\newcommand{\upparrow}{\tikz \draw[-triangle 90] (0,0)--(0,.1);}
\newcommand{\downnarrow}{\tikz \draw[-triangle 90] (0,.1)--(0,0);}
\newcommand{\rightiearrow}{\tikz \draw[-triangle 90] (0,0)--(.1,0);}

\newtheorem{thm}{Theorem}[section]

\newtheorem{lem}[thm]{Lemma}

\theoremstyle{definition}

\newtheorem{rmk}[thm]{Remark}
\newtheorem{defin}[thm]{Definition}

\DeclareMathOperator{\ran}{ran}
\DeclareMathOperator{\codim}{codim}
\DeclareMathOperator{\ind}{index}
\DeclareMathOperator{\ess}{ess}
\DeclareMathOperator{\ue}{ue}
\DeclareMathOperator{\smbl}{Smbl}
\DeclareMathOperator{\Tr}{Tr}
\DeclareMathOperator{\inte}{int}
\DeclareMathOperator{\exte}{ext}
\DeclareMathOperator{\ea}{ea}

\journal{Journal de Math\'ematiques Pures et Appliqu\'ees}
\begin{document} 
	\begin{frontmatter}
\title{The spectra of harmonic layer potential operators on domains with rotationally symmetric conical points\tnoteref{fund}}
\tnotetext[fund]{This work was supported by the Swedish Research Council under contract 621-2014-5159.}
\author{Johan Helsing}
\ead{helsing@maths.lth.se}
\address{Lund University, Box 118, SE-221 00 Lund, Sweden}

\author{Karl-Mikael Perfekt\corref{cor1}}
\ead{k.perfekt@reading.ac.uk}
\address{Department of Mathematics and Statistics, University of Reading, Reading RG6 6AX, United Kingdom}

\cortext[cor1]{Corresponding author}

\begin{abstract}
  We study the adjoint of the double layer potential associated with
  the Laplacian (the adjoint of the Neumann--Poincar\'e operator), as
  a map on the boundary surface $\Gamma$ of a domain in $\mathbb{R}^3$
  with conical points. The spectrum of this operator directly reflects
  the well-posedness of related transmission problems across $\Gamma$.
  In particular, if the domain is understood as an inclusion with
  complex permittivity $\epsilon$, embedded in a background medium
  with unit permittivity, then the polarizability tensor of the domain
  is well-defined when $(\epsilon+1)/(\epsilon-1)$ belongs to the resolvent set in
  energy norm. We study surfaces $\Gamma$ that have a finite number of
  conical points featuring rotational symmetry. On the energy space,
  we show that the essential spectrum consists of an interval.  On
  $L^2(\Gamma)$, i.e.  for square-integrable boundary data, we show
  that the essential spectrum consists of a countable union of curves,
  outside of which the Fredholm index can be computed as a winding
  number with respect to the essential spectrum. We provide explicit
  formulas, depending on the opening angles of the conical points. We
  reinforce our study with very precise numerical experiments,
  computing the energy space spectrum and the spectral measures of the
  polarizability tensor in two different examples. Our results
  indicate that the densities of the spectral measures may approach
  zero extremely rapidly in the continuous part of the energy space
  spectrum. \bigskip
  
  \noindent {\bf R\'esum\'e} \newline \newline
Nous \'etudions l'adjoint du potentiel de double couche associ\'e \`a
l'op\'erateur de Laplace (l'adjoint de l'op\'erateur de Neumann-Poincar\'e)
d\'efini sur la fronti\`ere $\Gamma$ d'un domaine de $\mathbb{R}^3$ contenant
des points coniques. Le spectre de cet op\'erateur est intimement li\'e
\`a la r\'esolution de probl\`emes de transmission \`a travers $\Gamma$.
En particulier, dans le contexte de la propagation des ondes \'electromagn\'etiques,
si le domaine d\'elimit\'e par $\Gamma$ repr\'esente une inclusion contenant un mat\'eriau de permittivit\'e complexe $\epsilon$, immerg\'e dans un milieu infini de permittivit\'e
\'egale \`a 1, on peut d\'efinir le tenseur de polarisabilit\'e d\`es que le rapport $(\epsilon+1)/(\epsilon-1)$ appartient \`a l'ensemble r\'esolvent de l'op\'erateur au sens de la norme d'\'energie.
Nous \'etudions des surfaces $\Gamma$ qui poss\`edent un nombre fini de points
coniques \`a sym\'etrie de rotation. Lorsque l'op\'erateur est d\'efini sur l'espace d'\'energie, nous montrons
que son spectre essentiel est un intervalle. Lorsqu'il est d\'efini dans l'espace $L^2(\Gamma)$,
i.e. pour des fonctions de carr\'e int\'egrable sur $\Gamma$, nous montrons que son
spectre est constitu\'e d'une union de courbes, en dehors desquelles on peut calculer
l'indice de Fredholm de l'op\'erateur, comme l’indice par rapport \`a
ces courbes. Nous donnons des formules explicites, en fonction de l'angle
d'ouverture des points coniques. Nous compl\'etons notre \'etude par des exp\'eriences num\'eriques tr\`es pr\'ecises,
o\`u, pour deux exemples, nous calculons le spectre de l'op\'erateur au sens de l'espace d'\'energie et les mesures spectrales du tenseur de polarisabilit\'e.
Nos r\'esultats sugg\`erent que les densit\'es des mesures spectrales approchent z\'ero
extr\^emement rapidement dans la partie continue du spectre au sens de l'espace d'\'energie.
\end{abstract}
\end{frontmatter}

\section{Introduction}
Let $\Gamma \subset \mathbb{R}^3$ be a connected Lipschitz surface, enclosing a bounded open domain $\inte(\Gamma)$ and with surface measure $d\sigma$. We are interested in the spectrum of the layer potential operator
\begin{equation} \label{eq:kdef}
K^\Gamma f(\bm{r}) = \frac{1}{2\pi}\int_{\Gamma} K^\Gamma(\bm{r}, \bm{r}') f(\bm{r}') \, d\sigma(\bm{r}'), \quad \bm{r} \in \Gamma,
\end{equation}
based on the normal derivative of
the Newtonian kernel
\begin{equation} \label{eq:kdefker}
K^\Gamma(\bm{r}, \bm{r}') = \frac{\langle \bm{r} - \bm{r}', \bm{\nu}_{\bm{r}} \rangle}{|\bm{r}' - \bm{r}|^3},
\end{equation}
where $\bm{\nu}_{\bm{r}}$ denotes the outward unit normal of $\Gamma$. $K^\Gamma$ may also be considered for planar Lipschitz curves $\Gamma \subset \mathbb{R}^2$, in which case the kernel is given by $K^\Gamma(\bm{r}, \bm{r}') = 2\frac{\langle \bm{r} - \bm{r}', \bm{\nu}_{\bm{r}} \rangle}{|\bm{r}' - \bm{r}|^2}$.

Knowledge about the spectrum of $K^\Gamma$ leads to existence and uniqueness results for boundary value problems involving the Laplacian on the interior and exterior domains $\inte(\Gamma)$  and $\exte(\Gamma)$ of $\Gamma$. For example, layer potential techniques may be used to solve the classical Dirichlet and Neumann problems for $\inte(\Gamma)$ by understanding the Fredholm theory of $K^\Gamma + I$ and $K^\Gamma - I$, respectively \cite{Ver84}.

When $\Gamma$ is non-smooth, for example if $\Gamma$ has corners in 2D, or edges or conical points in 3D, the spectrum of $K^\Gamma \colon X \to X$ is highly dependent on the space $X$. For example, suppose that $\Gamma$ is a curvilinear polygon in the plane. $K^\Gamma + I \colon L^2(\Gamma) \to L^2(\Gamma)$ is always invertible \cite{Ver84} when $\Gamma$ is Lipschitz, but in the polygonal case there always exist $p_0 > 2$, depending on the opening angles of the corners of $\Gamma$, such that $K^\Gamma + I \colon L^{p_0}(\Gamma) \to L^{p_0}(\Gamma)$ is not Fredholm \cite{Lew90,Shele90}. The underlying explanation for this is that when $\Gamma$ is an infinite wedge, the model domain to analyze corners; then, by homogeneity of its kernel, $K^\Gamma$ may be realized as a block matrix of Mellin convolution operators. These convolution kernels depend on $p_0$, accounting for the dependence on $p_0$ of the spectrum \cite{FJL76}. In 3D, similar results were shown in \cite{FJL77} in the idealized cases of $\Gamma$ being an infinite straight cone or an infinite three-dimensional wedge. We refer also to \cite{Mit02} for an extensive account of the $L^p$-theory in 2D, although with results only stated for $p \leq 2$.

In this paper we will, for surfaces $\Gamma \subset \mathbb{R}^3$, consider the action of $K^\Gamma$ on two different spaces: $L^2(\Gamma)$ and the energy space $\mathcal{E}$. The energy space consists of the distributions on $\Gamma$ whose single layer potentials have finite energy in $\mathbb{R}^3$. It is identifiable with the Sobolev space $H^{-1/2}(\Gamma)$ of index $-1/2$ on the boundary. The energy space stands out as the most natural space on which to consider $K^\Gamma$ for many reasons, one of them being that $K^\Gamma \colon \mathcal{E} \to \mathcal{E}$ is self-adjoint and therefore, in contrast with the $L^p$-theory, has a real spectrum.

Our interest in the entire spectrum of $K^\Gamma$ arises from the transmission problem
\begin{equation} \label{eq:introtrans}
\begin{cases}
\int_{\mathbb{R}^3} |\nabla U|^2 \, dV < \infty, \\
\Delta U(\bm{r}) = 0, \quad \bm{r} \in \mathbb{R}^3 \setminus \Gamma, \\
\Tr_{\inte} U(\bm{r}) = \Tr_{\exte} U(\bm{r}), \quad \bm{r} \in \Gamma, \\
\partial_{\bm{\nu}}^{\exte} U(\bm{r}) = \epsilon \partial_{\bm{\nu}}^{\inte} U(\bm{r}) - g(\bm{r}), \quad \bm{r} \in \Gamma.
\end{cases}
\end{equation}
Here $\epsilon \neq 1$, and $\Tr_{\inte}$ and $\partial_{\bm{\nu}}^{\inte}$ denote the boundary trace and normal derivative of interior approach, $\Tr_{\exte}$ and $\partial_{\bm{\nu}}^{\exte}$ the corresponding operators of exterior approach. If $g \in \mathcal{E}$, it turns out that there exists a solution $U$ of \eqref{eq:introtrans} satisfying $\lim_{\bm{r} \to \infty} U(\bm{r}) = 0$ if and only if there is $f \in \mathcal{E}$ such that
$$(K^\Gamma - z)f = \frac{g}{1-\epsilon}, \quad z = -\frac{1+\epsilon}{1-\epsilon}.$$

In the special case that $g(\bm{r}) = \bm{e} \cdot \bm{\nu}_{\bm{r}}$
for a vector $\bm{e} \in \mathbb{R}^3$, then solving the transmission
problem is involved in computing the polarizability tensor \cite{CM16,Fuchs75,HP13,SYJA04} of $\inte(\Gamma)$. In this setting, the domain is an inclusion
with complex permittivity $\epsilon$ in an infinite space of
permittivity 1. The polarizability tensor is associated with a set of spectral measures that arise from the spectral measure of $K^\Gamma \colon \mathcal{E} \to \mathcal{E}$, see Section~\ref{sec:introtrans}. Atoms in these spectral measures correspond to values of $\epsilon$ for which surface plasmon resonances can be excited \cite{AMRZ, APRY}. However, not every eigenvalue of $K^\Gamma$ necessarily produces a singularity in the polarizability tensor; we call such eigenvalues dark plasmons. In Section~\ref{sec:numres} we will observe an abundance of dark spectra for the type of surface $\Gamma$ that we will consider. More precisely, the described relationship between the transmission problem \eqref{eq:introtrans} and plasmonic resonances holds in the quasi-static approximation of the Maxwell equations. In the setting of smooth surfaces $\Gamma$, detailed analysis of plasmonic resonances using the full Maxwell equations and justification of the quasi-static approximation can be found in \cite{ADM, ARYZ}. Note that the spectrum of $K^\Gamma$ is pure point when $\Gamma$ is smooth.

For $\Gamma$ a plane polygon and $g \in H^{s}(\Gamma)$, $s > -1/2$, the spectrum of (a more general version of) the transmission problem \eqref{eq:introtrans} was studied in \cite{CS85}. In \cite{KLY15}, the spectral resolution of $K^\Gamma \colon \mathcal{E} \to \mathcal{E}$ was determined in a model case where $\Gamma$ is constructed from two intersecting circles (equivalent to the infinite wedge). For a general curvilinear polygon in 2D, the essential spectrum of $K^\Gamma \colon \mathcal{E} \to \mathcal{E}$ was determined in \cite{PP16}.
\begin{thm}[\cite{PP16}] \label{thm:PP16}
Suppose that $\Gamma \subset \mathbb{R}^2$ is a curvilinear polygon with corners of angles $\alpha_1, \ldots, \alpha_n$. Then the spectrum of $K^\Gamma \colon \mathcal{E} \to \mathcal{E}$ consists of an interval and a sequence $\{\lambda_k\}$ of eigenvalues with no limit point outside the interval,
$$\sigma(K^\Gamma, \mathcal{E}) = \left\{ x \in \mathbb{R} \, : \, |x| \leq \max_{1 \leq j \leq N} \left|1 - \frac{\alpha_j}{\pi} \right|\right\} \cup \{\lambda_k\}.$$ 
\end{thm}
See also \cite{HKL16} for a numerical experiments in agreement with this theorem. In the special case that $\Gamma$ coincides with two line segments in a neighborhood of each corner, a different approach to Theorem~\ref{thm:PP16}, yielding more information, very recently appeared in \cite{BZ17}. In three dimensions, only a few results concerning the entire spectrum seem to be available. As mentioned, the $L^p$-theory for infinite straight cones and wedges was considered in \cite{FJL77}. Also for the infinite straight cone, the generalized eigensolutions to the transmission problem \eqref{eq:introtrans} were explicitly computed in \cite{KCHWA14,MRL77,SPG14} -- these will be important in our determination of the spectrum of $K^\Gamma \colon \mathcal{E} \to \mathcal{E}$. For the more general type of infinite cone $\Gamma = \mathbb{R}^+ \omega$, where $\omega$ is a smooth curve on the sphere, the invertibility of $K^\Gamma - z$ on certain weighted Sobolev spaces has via Mellin convolutions been reduced to the invertibility of a parametric system of operators on $\omega$ \cite{CQ13,QN12}.

In the current contribution, we will characterize the spectrum of $K^\Gamma \colon L^2(\Gamma) \to L^2(\Gamma)$ and $K^\Gamma \colon \mathcal{E} \to \mathcal{E}$ in the case that $\Gamma$ is a rotationally symmetric surface with a conical point, see Figure~\ref{fig:surface}. Our main results straightforwardly generalize to surfaces with a finite number of conical points, each of which is locally rotationally symmetric around some axis. However, since the level of complexity is already quite high, we will never do so explicitly.
\begin{figure}[t!]
        \centering
        \includegraphics[width =0.20\linewidth]{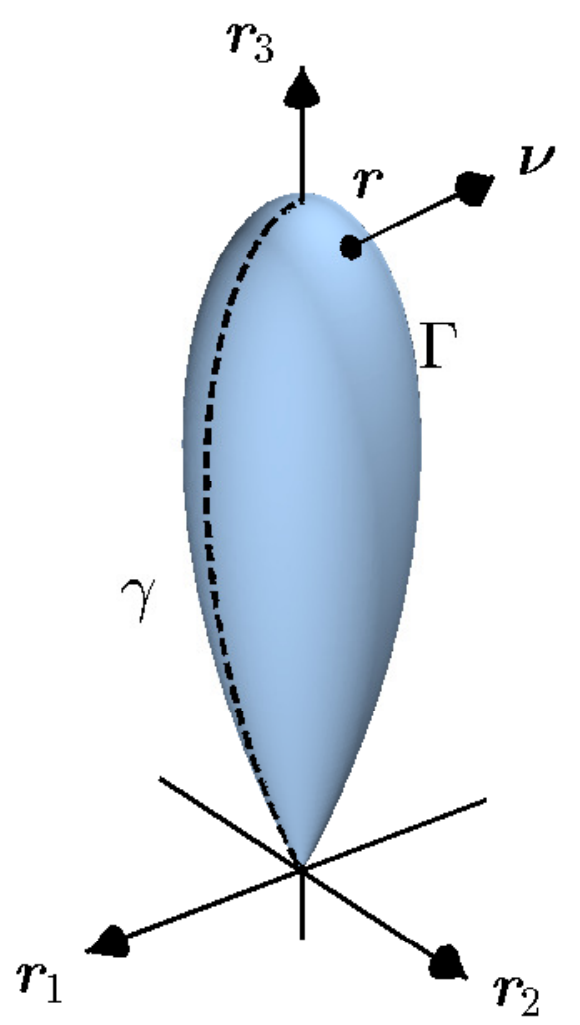}
        \qquad \qquad
        \includegraphics[width =0.30\linewidth]{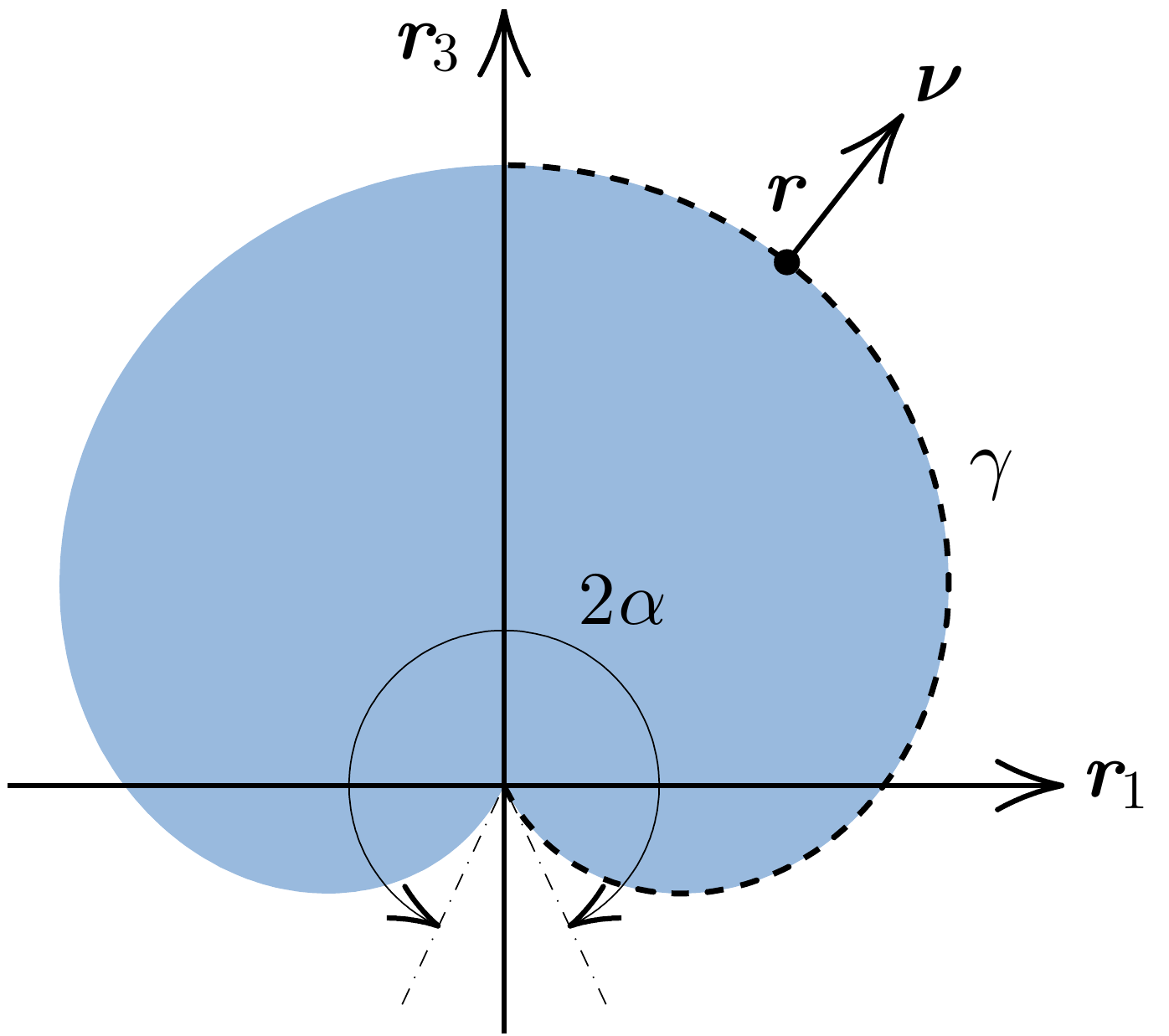}
        \caption{(a): An axially symmetric surface $\Gamma$ with a conical point of opening angle $2\alpha = 5\pi/18$. (b): A cross-section of $\inte(\Gamma)$ with opening angle $2\alpha = 31\pi/18$.}
        \label{fig:surface}
\end{figure}

We now state our main theorems, beginning with our result on the $L^2(\Gamma)$-spectrum. In the statement, $P_\lambda^n(x)$ denotes an associated Legendre function of the first kind (see the Appendix), and $\dot{P}_\lambda^n$ denotes its derivative in $x$.
\newtheorem*{thm:winding}{Theorem~\ref{thm:winding}}
\begin{thm:winding} \it
Let $\Gamma$ be a closed surface of revolution with a conical point of opening angle $2\alpha$, obtained by revolving a $C^5$-curve $\gamma$. For $n \in \mathbb{Z}$, denote by $\Pi_n$ the closed curve
$$
\Pi_n =\left\{ \frac{P_{i\xi}^n(\cos \alpha)\dot{P}_{i\xi}^n(-\cos \alpha)-P_{i\xi}^n(-\cos \alpha)\dot{P}_{i\xi}^n(\cos \alpha)}{P_{i\xi}^n(-\cos \alpha)\dot{P}_{i\xi}^n(\cos \alpha)+P_{i\xi}^n(\cos \alpha)\dot{P}_{i\xi}^n(-\cos \alpha)} \, : \, -\infty \leq \xi \leq \infty \right\},
$$
with orientation given by the $\xi$-variable.
Then the operator $K^\Gamma \colon L^2(\Gamma, \, d\sigma) \to L^2(\Gamma, \, d\sigma)$ has essential spectrum
$$
\sigma_{\ess}(K^\Gamma, L^2) = \bigcup_{n=-\infty}^\infty \Pi_n.
$$
If $z \notin \sigma_{\ess}(K^\Gamma, L^2)$, then $K^\Gamma-z$ has Fredholm index
$$\ind(K^\Gamma-z) = \sum_{n=-\infty}^\infty W(z, \Pi_n) = W(z, \Pi_0) + 2\sum_{n=1}^\infty W(z, \Pi_n)$$
where $W(z, \Pi_n) \leq 0$ denotes the winding number of $z$ with respect to $\Pi_n$ and the right-hand side is always a finite sum. In particular, every point $z$ lying inside one of the curves $\Pi_n$ belongs to the spectrum $\sigma(K^\Gamma, L^2)$.

Whenever $z$ is not a real number, it holds that $\dim \ker (K^\Gamma - z) = 0$, so that
$$\ind(K^\Gamma - z) = -\codim \ran K^\Gamma, \quad z \in \mathbb{C}\setminus\mathbb{R}.$$
 In particular,  if $\ind(K^\Gamma-z) = 0$ (so that $z$ lies outside every curve $\Pi_n$), then either $K^\Gamma-z$ is invertible or $z = x$ is real and an eigenvalue of $K^\Gamma$.
\end{thm:winding}
\begin{figure}[ht]
\includegraphics[width =0.48\linewidth]{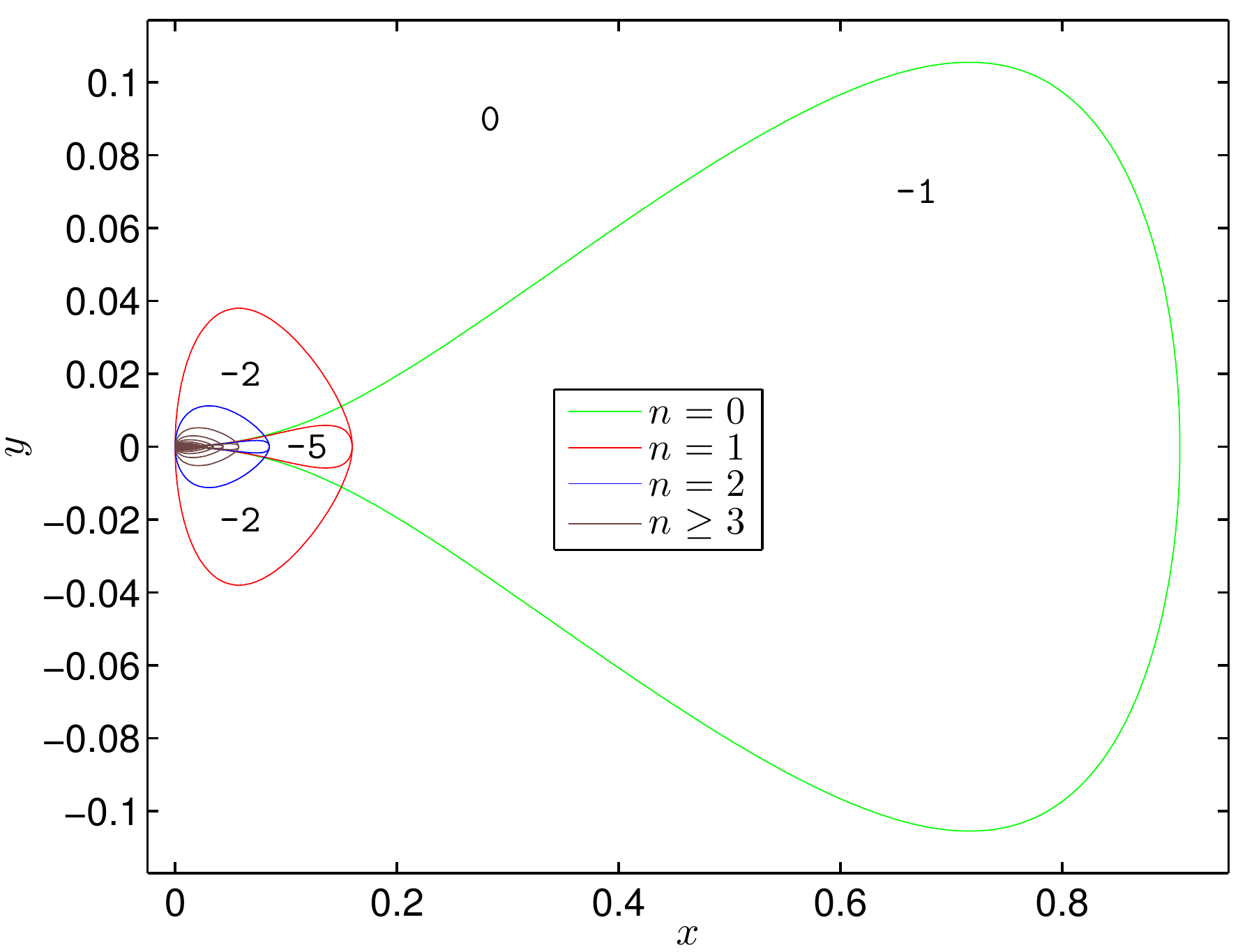}
\hfill
\includegraphics[width =0.48\linewidth]{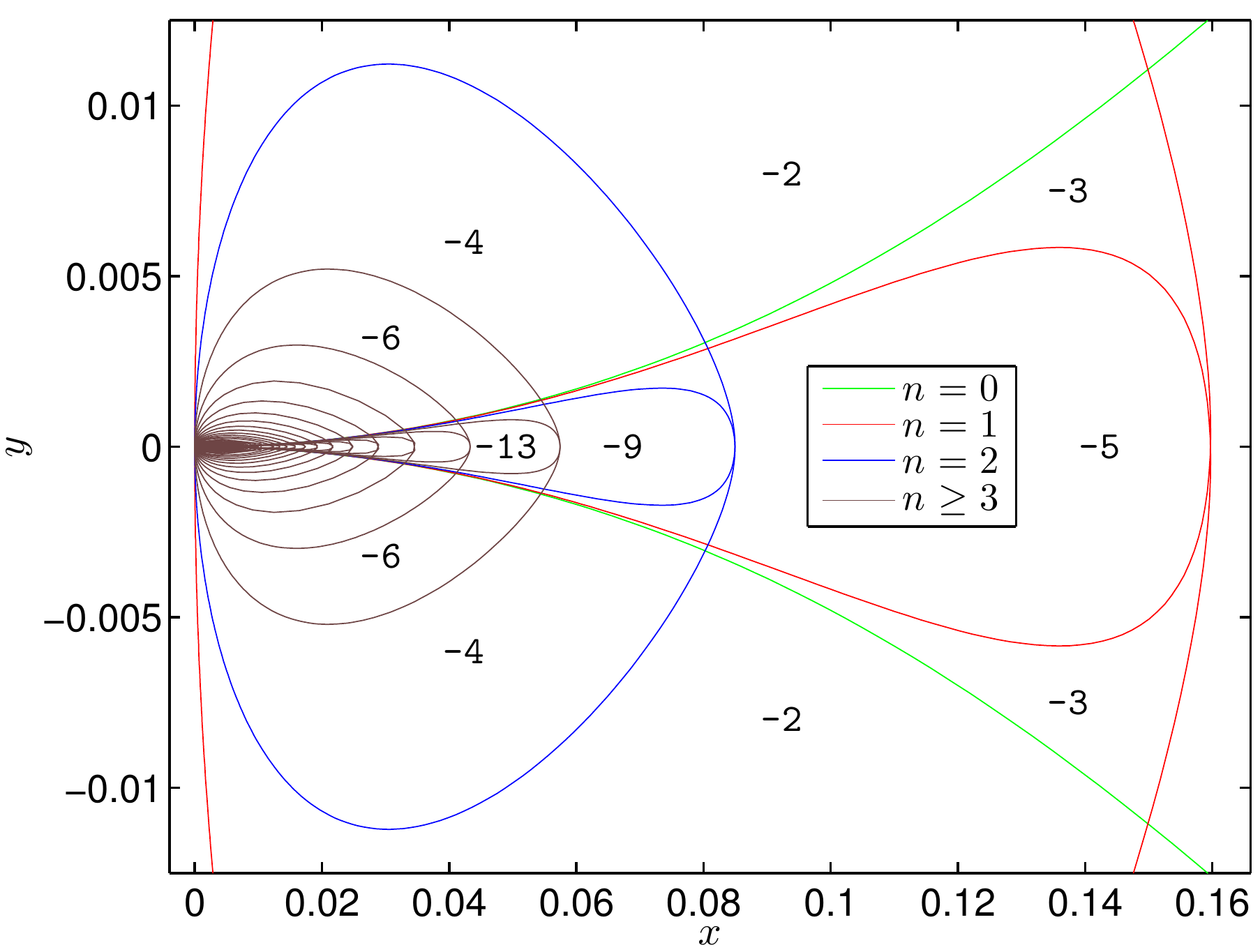}
\caption{The essential spectrum of $K^\Gamma$ for a surface $\Gamma$ with a conical point of opening angle $2\alpha = 5\pi/18$.  $\ind(K^\Gamma - z)$ is marked in the regions of non-essential spectrum. Zoom to the right.}
\label{fig:leaves}
\end{figure}
Theorem~\ref{thm:winding} is illustrated in Figure~\ref{fig:leaves}. After reversing the signs of the winding numbers, the first paragraph of the theorem applies equally well to the adjoint operator $(K^\Gamma)^\ast \colon L^2(\Gamma) \to L^2(\Gamma)$, known as the Neumann--Poincar\'e operator. As a consequence, the number of eigenfunctions of $(K^\Gamma)^\ast \colon L^2(\Gamma) \to L^2(\Gamma)$ to the eigenvalue $z \in \mathbb{C}$ is equal to the winding number of $z$ with respect to the essential spectrum, except at certain exceptional real values $z \in \mathbb{R}$.

Next, we state our characterization of the $\mathcal{E}$-spectrum.
\newtheorem*{thm:Eessspec}{Theorem~\ref{thm:Eessspec}}
\begin{thm:Eessspec} \it
        Let $\Gamma$ be a closed surface of revolution with a conical point of opening angle $2\alpha$, obtained by revolving a $C^5$-curve $\gamma$. For $n \in \mathbb{Z}$, denote by $\Sigma_n$ the closed interval
        \begin{equation*}
        \Sigma_n = \left\{\frac{P_{i\xi - 1/2}^n(\cos \alpha)\dot{P}_{i\xi - 1/2}^n(-\cos \alpha) - P_{i\xi - 1/2}^n(-\cos \alpha)\dot{P}_{i\xi - 1/2}^n(\cos \alpha)}{P_{i\xi - 1/2}^n(-\cos \alpha)\dot{P}_{i\xi - 1/2}^n(\cos \alpha)+P_{i\xi - 1/2}^n(\cos \alpha)\dot{P}_{i\xi - 1/2}^n(-\cos \alpha)} \, : \, -\infty \leq \xi \leq \infty \right\}.
        \end{equation*}
        Then the self-adjoint operator $K^\Gamma \colon \mathcal{E} \to \mathcal{E}$, where $\mathcal{E}$ is the energy space of $\Gamma$, has essential spectrum
        \begin{equation*}
        \sigma_{\ess}(K^\Gamma, \mathcal{E}) = \bigcup_{n=-\infty}^\infty \Sigma_n.
        \end{equation*}
        Hence, the spectrum of $K^\Gamma$ consists of this interval and a sequence of real eigenvalues $\{\lambda_k\}$ with no limit point outside of it,
                \begin{equation*}
                \sigma(K^\Gamma, \mathcal{E}) = \{\lambda_k\} \cup \sigma_{\ess}(K^\Gamma, \mathcal{E}).
                \end{equation*}
\end{thm:Eessspec}
In Section~\ref{sec:numerical} we will develop a method to numerically determine the polarizability tensor and spectrum of $K^\Gamma$ for rotationally symmetric surfaces $\Gamma$. We offer one of our numerical results already here, which at the same time illustrates Theorem~\ref{thm:Eessspec}. In the proof of Theorem~\ref{thm:Eessspec} we decompose $K^\Gamma \colon \mathcal{E} \to \mathcal{E}$ according to its Fourier modes, $K^\Gamma = \oplus K^\gamma_n$. Figure~\ref{fig:introobtuse} demonstrates the indicator function for mode $n=0$, which detects the spectrum of $K^\gamma_0$, for a surface $\Gamma$ of opening angle $2\alpha = 31\pi/18 > \pi$. The set where the indicator function is equal to $1/2$ coincides with the interval $\Sigma_0$ of Theorem~\ref{thm:Eessspec}, i.e. the essential spectrum of $K^\gamma_0$. The points where the indicator function is $1$ correspond to eigenvalues. It turns out (only numerically demonstrated) that in this case there is an infinite sequence of eigenvalues outside the essential spectrum, and every eigenvalue but one yields a plasmon resonance.
\begin{figure}[ht]
\centering
        \includegraphics[width =0.48\linewidth]{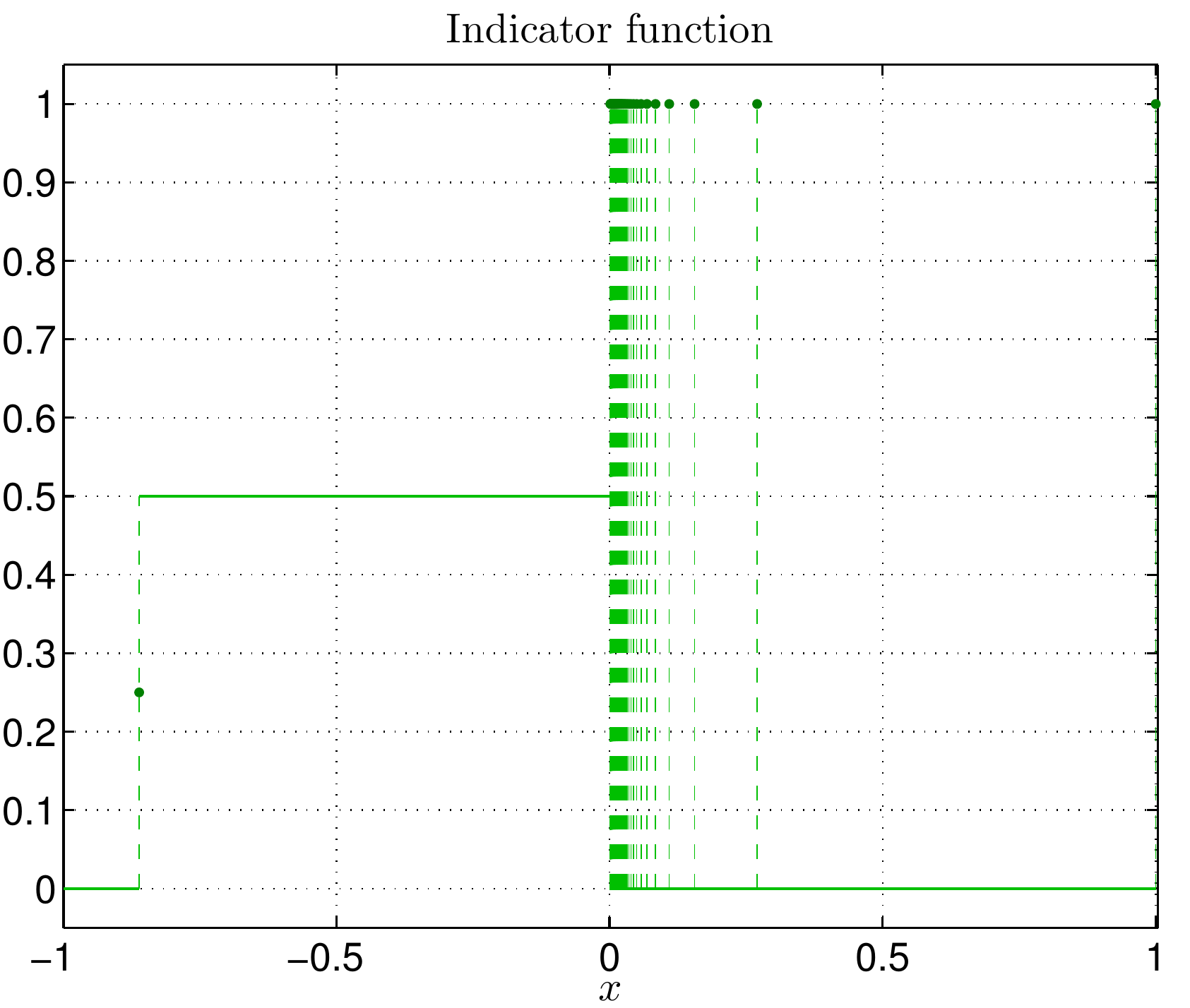}
        \caption{The indicator function for a surface $\Gamma$ with a conical point of reflex opening angle $2\alpha = 31\pi/18$ and mode $n=0$. 275 eigenvalues are drawn.}
        \label{fig:introobtuse}
\end{figure}

We now explain the layout of the paper, with some remarks on the content of each section. Section~\ref{sec:prelim} contains preliminary material on layer potentials, the energy space, the transmission problem, limit polarizability, Fredholm operators, Mellin transforms, Sobolev spaces and singular integral operators. 

In Section~\ref{sec:straightcone} we study the model case in which $\Gamma$ is an infinite straight cone. We provide the spectral resolution of both operators $K^\Gamma \colon L^2(\Gamma) \to L^2(\Gamma)$ and $K^\Gamma \colon \mathcal{E} \to \mathcal{E}$. The first case is quite straightforward, and the relevant analysis appears implicitly in \cite{FJL77}. Each modal operator $K_n^\gamma$ is in this case unitarily equivalent to a Mellin convolution operator, and this leads to the spectral resolution. On the energy space we make use of a special norm related to the single layer potential which has several advantages. For one, this norm decomposes orthogonally with respect to the Fourier modes. Secondly, it allows us to exploit that we can calculate the action of the single layer potential on the generalized eigenfunctions of $K^\Gamma$ on the infinite cone. We remark that for the case of intersecting disks, the same norm was used in \cite{KLY15} to realize the spectral theorem of $K^\Gamma$.

In Section~\ref{sec:perturb} we show, in a certain sense, that $K^\Gamma$ is a compact perturbation of $K^{\Gamma_\alpha}$, where $\Gamma_\alpha$ is a straight cone of the same opening angle. The proof proceeds by writing the difference of kernels as a sum of products of Riesz kernels with smooth, small kernels. The Riesz transforms are however not bounded on $\mathcal{E}$; since $\bm{\nu}_{\bm{r}} \notin H^{1/2}(\Gamma)$ this would contradict \cite[Eq. (6.50)]{MMV16}. Hence, the indicated argument provides compactness on $L^2(\Gamma)$, but for $\mathcal{E}$ we have to work harder. We combine certain algebraic identities (the Plemelj formula) with further estimates and real interpolation in this case.

In Section~\ref{sec:L2spec} we prove Theorem~\ref{thm:winding}. The index formula is proven by showing that the modal operators $K^\gamma_n$ are pseudodifferential operators of Mellin type, for which there is a well-developed symbolic calculus \cite{Els87,Lew91,LP83}. Theorem~\ref{thm:Eessspec} is proven in Section~\ref{sec:Espec}. The method of proof is to first show that on the infinite cone, the singularities of $K^{\Gamma_\alpha}$ at the origin and at infinity contribute equally to the essential spectrum. The theorem is then pieced together from the results in Sections~\ref{sec:straightcone} and \ref{sec:perturb}.

Section~\ref{sec:numerical} contains our numerical results. We first define the indicator function, and establish its properties. Then we give an overview of the numerical method, after which we present numerical results on the polarizability tensor and the spectrum of $K^\Gamma$ for the two surfaces $\Gamma$ illustrated in Figure~\ref{fig:surface}.

Finally, the Appendix contains explicit expressions for the various modal kernels we will consider, in terms of special functions. Our theory and numerical method both depend on these explicit formulas. In particular, we will defer the proof of the technical Lemma~\ref{lem:specestimate} to the Appendix. 

\subsubsection*{Notation} If $A$ and $B$ are two non-negative quantities depending on some variables, we write $A \lesssim B$ to signify that there is a constant $c > 0$ such that $A \leq cB$. If $A \lesssim B$ and $B \lesssim A$, we write $A \simeq B$. 

\section{Background, definitions and notation} \label{sec:prelim}
\subsection{Single and double layer potentials}
For a function $f$ on $\Gamma$, its single layer potential $S^\Gamma f$ is given by
\begin{equation} \label{eq:Sdef}
 S^{\Gamma}f(\bm{r}) = \int_{\Gamma} S^{\Gamma}(\bm{r}, \bm{r'}) f(\bm{r'}) \, d\sigma(\bm{r'}), \quad \bm{r} \in \mathbb{R}^3,
 \end{equation}
where 
$$ S^{\Gamma}(\bm{r}, \bm{r'}) = \frac{1}{2 \pi} \frac{1}{|\bm{r} - \bm{r'}|}.$$
Note that $S^{\Gamma}f(\bm{r})$ is a harmonic function for $\bm{r} \in \mathbb{R}^3 \setminus \Gamma = \inte(\Gamma)\cup \exte(\Gamma)$. If $f$ is any reasonable function or distribution, $S^{\Gamma}f(\bm{r})$ will have traces from both the interior domain $\inte(\Gamma)$ and the exterior domain $\exte(\Gamma)$. Due to the weak singularity of the kernel, these traces coincide with evaluating $S^\Gamma$ directly on $\Gamma$,
$$\Tr_{\inte} S^{\Gamma}f(\bm{r}) = \Tr_{\exte} S^{\Gamma}f(\bm{r}) = S^{\Gamma}f(\bm{r}), \quad \bm{r} \in \Gamma.$$
In the most general case, these traces may be understood in the sense of convergence in nontangential cones for almost every point of $\Gamma$ \cite{Ver84}, or in a distributional sense \cite{PP14}. Most of the time we will consider $S^\Gamma$ as map directly on $\Gamma$, since the well-posedness of the interior and exterior Dirichlet problems ensure that $S^{\Gamma}f$ can be uniquely identified with its values on $\Gamma$, see \cite{PP14}.

The layer potential $K^\Gamma f$, evaluated on the boundary, is given by the principal value integral \eqref{eq:kdef} with kernel defined in \eqref{eq:kdefker}. The adjoint operator $(K^\Gamma)^\ast$ (with respect to $L^2(\Gamma)$) is usually referred to as the boundary double layer potential, or the Neumann--Poincar\'e operator. Note also that the choice of normalizing constant in front of \eqref{eq:kdef} may be different in other works. 

We will consider two different function/distribution spaces for the action of $K^\Gamma$. First, we will consider $K^\Gamma \colon L^2(\Gamma) \to L^2(\Gamma)$ as an operator on $L^2(\Gamma) = L^2(\Gamma, d\sigma)$. $K^\Gamma$ is always bounded on $L^2(\Gamma)$ \cite{Ver84}, but note that $K^\Gamma$ is not a self-adjoint operator in this space. When $\Gamma$ has singularities, so that $K^\Gamma$ is not a compact operator, the spectrum of $K^\Gamma$ on $L^2(\Gamma)$ is typically not real. This is illustrated by our main theorem on the $L^2(\Gamma)$-spectrum, Theorem~\ref{thm:winding}. See also \cite{Mit02} for the 2D-case.

The second space we will consider is the Hilbert space $\mathcal{E}$, obtained by completing $L^2(\Gamma)$ in the positive definite scalar product
$$\langle f, g \rangle_{\mathcal{E}} = \langle S^\Gamma f, g \rangle_{L^2(\Gamma)}.$$
By applying the classical jump formulas for the interior and exterior normal derivatives of $S^\Gamma f$ and Green's formula, we have that
\begin{equation} \label{eq:energynorm}
\langle f, f \rangle_{\mathcal{E}} = \frac{1}{2} \int_{\mathbb{R}^3} | \nabla S^\Gamma f|^2 \, dV.
\end{equation}
A proof, which carries over verbatim to the Lipschitz case (see \cite{PP14}) may be found in \cite[Lemma 1]{KPS07}. Here $dV$ denotes the usual volume element on $\mathbb{R}^3$. Hence, we refer to $\mathcal{E}$ as the \textit{energy space}, as it consists of charges generating single layer potentials $S^{\Gamma}f$ with finite energy in $\mathbb{R}^3$. In light of this physical interpretation, it is not a complete surprise that $K^\Gamma \colon \mathcal{E} \to \mathcal{E}$ is self-adjoint as an operator on $\mathcal{E}$. Indeed, from the Plemelj formula
\begin{equation} \label{eq:plemelj}
S^\Gamma K^\Gamma f(\bm{r}) = (K^\Gamma)^\ast S^\Gamma f(\bm{r}), \quad \bm{r} \in \Gamma
\end{equation}
it follows that
$$\langle K^\Gamma f, g \rangle_{\mathcal{E}} = \langle f, K^\Gamma g \rangle_{\mathcal{E}},$$
see \cite{KLY15,KPS07,PP14}. By considering trace theorems and well-posedness of Dirichlet problems, it can be deduced that the $\mathcal{E}$-norm is equivalent to the Sobolev norm of index $-1/2$ on the boundary (see Section~\ref{sec:sobolev}),
\begin{equation} \label{eq:energynormequiv}
\|f\|_{\mathcal{E}} \simeq \|f\|_{H^{-1/2}(\Gamma)}.
\end{equation}
Again we refer the reader to \cite{KPS07}, or to \cite{PP14} for a treatment explicitly for the Lipschitz case. By interpolating between $L^2(\Gamma)$ and $H^{-1}(\Gamma)$, where $K^\Gamma$ is bounded by the classical theory \cite{Ver84}, it now follows that $K^\Gamma : \mathcal{E} \to \mathcal{E}$ is bounded.

It is known \cite[Theorem 2.5]{CL08} that as an operator on $\mathcal{E}$ the spectrum of $K^\Gamma$ is contained in $(-1,1]$,
\begin{equation}\label{eq:Kspeccontain}
\sigma(K^\Gamma, \mathcal{E}) \subset (-1,1].
\end{equation}
However, without additional hypotheses on $\Gamma$ such as convexity or smoothness, it is not even known if the essential norm of $K^\Gamma \colon L^2(\Gamma) \to L^2(\Gamma)$ is less than $1$,
$$\|K^\Gamma\|_{L^2(\Gamma) \to L^2(\Gamma), \ess} \leq 1?$$
We refer to \cite{Wendland09} for a discussion.

In addition to bounded domains, we will consider one instance of an unbounded surface. Namely, we will consider an infinite straight cone $\Gamma_\alpha$ of opening angle $2\alpha$, $0 < \alpha < \pi$, $\alpha \neq \pi/2$. In general the layer potential theory for domains with non-compact boundary is rather delicate, but in our particular case $\Gamma_\alpha$ is a Lipschitz graph. In any case, since $\Gamma_\alpha$ will be our model for studying domains with axially symmetric conical points, we will make precise calculations from which the boundedness and other basic properties of $K^{\Gamma_\alpha} \colon L^2(\Gamma_\alpha) \to L^2(\Gamma_\alpha)$ and $K^{\Gamma_\alpha} \colon \mathcal{E} \to \mathcal{E}$ will be apparent. All of the properties of $S^\Gamma$, $K^\Gamma$, and $\mathcal{E}$ mentioned in this subsection continue to hold, except that we (the authors) are not entirely sure about the available results on the Dirichlet problem. In particular, we are not sure if \eqref{eq:energynormequiv} holds. However, in view of \eqref{eq:energynorm} and the boundedness of the trace \cite{Mik11}, we at least have that
\begin{equation} \label{eq:energynormdom}
\|f\|_{\mathcal{E}} = \sup_{\|g\|_{\mathcal{E}}=1} \langle S^{\Gamma_\alpha} g, f \rangle_{L^2(\Gamma_\alpha)} \leq \sup_{\|g\|_{\mathcal{E}}=1} \|\Tr S^{\Gamma_\alpha} g\|_{H^{1/2}(\Gamma_\alpha)} \|f\|_{H^{-1/2}(\Gamma_\alpha)} \lesssim \|f\|_{H^{-1/2}(\Gamma_\alpha)}.
\end{equation}
Furthermore, if $\rho \in C^\infty_c(\Gamma_\alpha)$ is a smooth compactly supported function, then
\begin{equation} \label{eq:energynormcutoff}
\|\rho f\|_{\mathcal{E}} \simeq \|\rho f\|_{H^{-1/2}(\Gamma_\alpha)},
\end{equation}
with implicit constants possibly depending on $\rho$.
\subsection{The transmission problem and limit polarizability} \label{sec:introtrans}
In the transmission problem \eqref{eq:introtrans}, with $g \in \mathcal{E}$, the normal derivatives $\partial_{\bm{\nu}}^{\exte} U$ and $\partial_{\bm{\nu}}^{\inte} U$ of exterior and interior approach need to be understood in a distributional sense. Making the ansatz $U = S^\Gamma f$, the jump formulas
\begin{equation} \label{eq:jumpformula}
\partial_{\bm{\nu}}^{\inte} S^{\Gamma}f(\bm{r}) = f(\bm{r}) - K^\Gamma f(\bm{r}), \quad \partial_{\bm{\nu}}^{\exte} S^{\Gamma}f(\bm{r}) = -f(\bm{r}) - K^\Gamma f(\bm{r}), \quad \bm{r} \in \Gamma
\end{equation}
imply that $U$ solves the transmission problem if and only if $f \in \mathcal{E}$ and
$$(K^\Gamma - z)f = \frac{g}{1-\epsilon}, \quad z = -\frac{1+\epsilon}{1 - \epsilon}.$$
In fact, in the case that $\Gamma$ is a bounded surface, any solution to the transmission problem which satisfies $\lim_{\bm{r} \to \infty} U(\bm{r}) = 0$ must be of this form, as mentioned in the introduction. See \cite[Proposition 5.1]{HP13}. 

To define the polarizability tensor of $\Gamma$ we understand
$\inte(\Gamma)$ as an inclusion with permittivity $\epsilon$, embedded
in infinite space of permittivity $1$. For a unit field $\bm{e} \in
\mathbb{R}^3$, we seek a potential $U$ such that
\begin{equation*}
\begin{cases}
\int_{\mathbb{R}^3} |\nabla U - \bm{e}|^2 \, dV < \infty, \\
\Delta U(\bm{r}) = 0, \quad \bm{r} \in \mathbb{R}^3 \setminus \Gamma, \\
\Tr_{\inte} U(\bm{r}) = \Tr_{\exte} U(\bm{r}), \quad \bm{r} \in \Gamma, \\
\partial_{\bm{\nu}}^{\exte} U(\bm{r}) = \epsilon \partial_{\bm{\nu}}^{\inte} U(\bm{r}), \quad \bm{r} \in \Gamma.
\end{cases}
\end{equation*}
The single layer potential ansatz 
$$U(\bm{r}) = \bm{e} \cdot \bm{r} + S^\Gamma \rho(\bm{r})$$
 yields \cite[Section~2]{HP13} the equation
$$(K^\Gamma - z)\rho = g_{\bm{e}}, \quad g_{\bm{e}}(\bm{r}) = (\bm{e} \cdot \bm{\nu_r}), \quad z = -\frac{1+\epsilon}{1 - \epsilon}.$$
If the solution $U$ exists uniquely for all $\bm{e}$, then the polarizability tensor $\omega$, a linear map on $\mathbb{R}^3$, scaled by the volume $|\inte(\Gamma)|$ of $\inte(\Gamma)$, is defined by 
$$\omega(z)\bm{e} = \frac{(\epsilon - 1)}{|\inte(\Gamma)|}\int_{\inte(\Gamma)} \nabla U(\bm{r}) \, d V(\bm{r}).$$
To evaluate the polarizability, we make use of Green's formula
\begin{equation} \label{eq:green}
\int_{\inte(\Gamma)} \langle \nabla U, \nabla V \rangle \, dV = \int_{\Gamma} U \overline{\partial_{\bm{\nu}}^{\inte} V} \, d\sigma,
\end{equation} 
valid for $U$ and $V$ harmonic in $\inte(\Gamma)$ and of sufficient smoothness.

We suppose now that $\Gamma$ is rotationally symmetric around the $\bm{r}_3$-axis. Then $\omega(z)$ is diagonal, and its first two diagonal entries are equal, $\omega_{11}(z) = \omega_{22}(z)$. We refer to $\omega_{jj}(z)$ as polarizability in the $\bm{r}_j$-direction, $j=1,2,3$. Applying \eqref{eq:green} and the jump formulas \eqref{eq:jumpformula} yields that
\begin{equation} \label{eq:polariz}
\omega_{jj}(z) = \bm{e}_j \cdot \omega(z)\bm{e}_j = \int_{\Gamma}\rho(\bm{r}) h_{\bm{e}_j}(\bm{r}) \, d\sigma(\bm{r}) =  \langle (K^\Gamma-z)^{-1}g_{\bm{e}_j}(\bm{r}), h_{\bm{e}_j}(\bm{r}) \rangle_{L^2(\Gamma)},
\end{equation}
where $\bm{e}_j$ denotes the $j$th unit vector in the standard basis of $\mathbb{R}^3$, and
$$h_{\bm{e}}(\bm{r}) = -2\frac{\bm{e}\cdot\bm{r}}{|\inte(\Gamma)|}.$$
In Section~\ref{sec:indicator} we will see that \eqref{eq:polariz} is associated with a spectral measure $\mu_j$,
\begin{equation} \label{eq:polarcauchy}
\omega_{jj}(z) = \int_{-1}^{1} \frac{d\mu_j(s)}{s - z}.
\end{equation}
This statement is a little more subtle than it appears, since
$K^\Gamma$ is not a self-adjoint operator in the $L^2$-pairing. An
appropriate formalism was developed in \cite[Section~5]{HP13}, and we
shall carry out the corresponding details for our situation in
Section~\ref{sec:indicator}. Alternative approaches may be found in
\cite{CM16,Fuchs75,GP83,MCHG15}. Some of these references concern
the effective permittivity tensor rather than polarizability. However,
the polarizability tensor may be viewed as a limiting case of the
effective permittivity tensor.

By the representation of the polarizability as a Cauchy integral \eqref{eq:polarcauchy}, the limit
$$\omega_{jj}^-(x) = \lim_{y \to 0^-} \omega_{jj}(x+iy)$$
exists almost everywhere $x \in \mathbb{R}$, even when $x$ lies in the support of $\mu_{j}$. We refer to $\omega_{jj}^-(x)$ as the limit polarizability. When $x$ lies outside the support of $\mu_j$ the limit polarizability and polarizability coincide. For axially symmetric domains with a conical point, we will find that the spectral measure $\mu_{j}$ typically has an absolutely continuous part, in addition to a possible singular part. The absolutely continuous part is recognized by the fact that almost everywhere $x \in \mathbb{R}$ it holds that
$$\mu_{j}'(x) = -\frac{\Im \mathrm{m} \, \omega_{jj}^-(x)}{\pi}.$$

By \cite[Remark~5.1 and Theorem~5.2]{HP13}, $-\mu_j$ is a positive measure, and $\rho$ and $\mu_{j}$ satisfy that
\begin{equation} \label{eq:rhorule}
\int_{\Gamma} \rho(\bm{r}) \, d\sigma(\bm{r}) = 0,
\end{equation}
\begin{equation} \label{eq:murule}
\int_{-1}^{1} \, d\mu_{j}(s) = -2.
\end{equation}
Let $\mu_{j}^{\textrm{pp}}$ be the pure point part of $\mu_{j}$.
By \cite[Theorem~5.6]{HP13} there are eigenvectors $\phi_i \in \mathcal{E}$ and $\psi_i \in H^{1/2}(\Gamma)$ of $K^{\Gamma}$ and $(K^\Gamma)^*$, respectively, normalized so that
$$\langle \phi_i, \psi_k \rangle_{L^2(\Gamma)} = \delta_{ik},$$
such that
$$\int_\Gamma \, d\mu_{j}^{\textrm{pp}}(s) = \sum_i u_i v_i, \quad u_i = \langle \phi_i, h_{\bm{e}_j} \rangle_{L^2(\Gamma)}, \quad v_i = \langle \psi_i, g_{\bm{e}_j} \rangle_{L^2(\Gamma)}.$$
In particular, if $\mu_{j}$ has no singular continuous part, then \eqref{eq:murule} takes the form
\begin{equation} \label{eq:murule2}
\sum_i u_i v_i + \int_{-1}^{1} \mu'_j(x) \, dx = -2.
\end{equation}
We strongly believe, but will not prove, that $\mu_{j}$ never has a singular continuous part for the surfaces $\Gamma$ we consider. We shall use the rules \eqref{eq:rhorule} and \eqref{eq:murule2} to verify the accuracy of our numerical results.
\subsection{Fredholm operators}
Recall that a bounded operator $S \colon \mathcal{H} \to \mathcal{H}$ on a Hilbert space $\mathcal{H}$ is \textit{Fredholm} if it has closed range and both its kernel and cokernel are finite-dimensional. Equivalently, $S$ is Fredholm if and only if it is invertible modulo compact operators. If $S$ is Fredholm, its index is given by
$$\ind S = \dim \ker S - \codim \ran S.$$
\begin{defin}
If two operators $S \colon \mathcal{H}_1 \to \mathcal{H}_1$ and $T \colon \mathcal{H}_2 \to \mathcal{H}_2$ on Hilbert spaces $\mathcal{H}_1$ and $\mathcal{H}_2$ are unitarily equivalent, we write that $S \simeq_{\ue} T$.    
\end{defin}
\begin{defin} \label{def:fredholm}
We write $S \simeq T$ if there exist Hilbert spaces $\mathcal{H}_1'$ and $\mathcal{H}_2'$ and a compact operator $K : \mathcal{H}_2 \oplus \mathcal{H}_2' \to \mathcal{H}_2 \oplus \mathcal{H}_2'$ such that $S \oplus 0 : \mathcal{H}_1 \oplus \mathcal{H}_1' \to \mathcal{H}_1 \oplus \mathcal{H}_1'$ is similar to $(T \oplus 0) + K : \mathcal{H}_2 \oplus \mathcal{H}_2' \to \mathcal{H}_2 \oplus \mathcal{H}_2'$.
\end{defin}
The point of the above definition is that if $S \simeq T$ and $z \neq 0$, it holds that $S - z$ is Fredholm if and only if $T - z$ is Fredholm and then the Fredholm indices satisfy
$$\ind(S - z) = \ind(T - z).$$
For a (not necessarily self-adjoint) operator $S \colon \mathcal{H}_1 \to \mathcal{H}_1$ we will denote its essential spectrum in the sense of invertibility modulo compacts by $\sigma_{\ess}(S, \mathcal{H}_1)$.
\begin{defin}
The essential spectrum of $S \colon \mathcal{H}_1 \to \mathcal{H}_1$ is the set
$$\sigma_{\ess}(S, \mathcal{H}_1) = \{z \in \mathbb{C} \, : \, K - z \textrm{ is not Fredholm} \}.$$
\end{defin}
We will also make use of another concept of essential spectrum, also invariant under compact perturbations. We say that a bounded sequence $(x_n) \subset \mathcal{H}_1$ is a singular sequence for the operator $S \colon \mathcal{H}_1 \to \mathcal{H}_1$ and spectral point $\lambda$ if $(x_n)$ has no convergent subsequences and $(S-\lambda) x_n \to 0$ in $\mathcal{H}_1$ as $n \to \infty$.
\begin{defin}
The point $\lambda \in \mathbb{C}$ belongs to $\sigma_{\ea}(S, \mathcal{H}_1)$ if and only if there is a singular sequence $(x_n)$ for $S$ and $\lambda$.
\end{defin}
Note that if $S \colon \mathcal{H}_1 \to \mathcal{H}_1$ is a self-adjoint operator, then the two type of essential spectra agree by Weyl's criterion, $\sigma_{\ea}(S, \mathcal{H}_1) = \sigma_{\ess}(S, \mathcal{H}_1)$. Furthermore, in this case $\ind(S - \lambda) = 0$ whenever $\lambda \notin \sigma_{\ess}(S, \mathcal{H}_1)$.
\subsection{Mellin transforms}
For $g \in L^1([0, \infty), \, dt/t)$, let $\mathcal{M}g = \hat{g}$ be its Mellin transform,
$$\mathcal{M}g(\zeta) = \hat{g}(\zeta) = \int_0^\infty t^\zeta g(t) \frac{dt}{t}.$$
The $L^1$-hypothesis on $g$ implies that $\mathcal{M}g(\zeta)$ is well-defined and bounded at least for $\zeta = i \xi \in i \mathbb{R}$.
We will denote Mellin convolution by $\star$,
$$ u \star v(t) = \int_0^\infty u(t/t') v(t') \frac{dt'}{t'}, \quad t > 0.$$
The Mellin transform is the Fourier transform of the multiplicative group of $(0, \infty)$; for sufficiently nice functions $u$ and $v$ and appropriate $\zeta$ it holds that
$$\mathcal{M}(u \star v)(\zeta) = \mathcal{M}u(\zeta) \mathcal{M}v(\zeta).$$
Young's inequality for the Mellin transform says that
\begin{equation} \label{eq:young}
\|u \star v \|_{L^2(dt/t)} \leq \|u\|_{L^1(dt/t)} \|v\|_{L^2(dt/t)}.
\end{equation}
Another way to see this is by noting that $W \colon L^2(dt/t) \to L^2(\mathbb{R})$ is a unitary operator,
\begin{equation} \label{eq:mellinunitary}
Wv(\xi) = \frac{1}{\sqrt{2\pi}} \mathcal{M}v(i\xi),
\end{equation}
 with inverse 
$$W^{-1} \psi(t) = \frac{1}{\sqrt{2\pi}}\int_{-\infty}^{\infty} t^{-i\xi} \psi(\xi) \, d\xi.$$
In particular, Plancherel's formula takes the form
\begin{equation} \label{eq:plancharel}
\frac{1}{2\pi}\int_{-\infty}^{\infty} \left| \int_0^\infty t^{i\xi} v(t) \, \frac{dt}{t} \right|^2 \, d\xi = \int_0^\infty |v(t)|^2 \, \frac{dt}{t}.
\end{equation}
\subsection{Singular integral estimates on Sobolev spaces} \label{sec:sobolev}
Suppose first that $\Gamma$ is a Lipschitz graph
$$\Gamma = \{(x,y, \varphi(x,y)) \, : \, (x,y) \in \mathbb{R}^2\}.$$
The parametrization $(x,y) \to (x,y, \varphi(x,y))$ then induces tangential derivatives $\partial_x$ on $\partial_y$ on $\Gamma$. The (inhomogeneous) Sobolev space $H^1(\Gamma)$ consists of those functions $f$ such that
$$\|f\|_{H^1(\Gamma)}^2 = \|f\|_{L^2(\Gamma)} + \|\partial_x f\|_{L^2(\Gamma)} + \|\partial_y f\|_{L^2(\Gamma)} < \infty.$$
This also allows us to define $H^1(\Gamma)$ in the case that $\Gamma$ is a bounded Lipschitz surface, via its Lipschitz manifold structure. In this setting, we will make use of the fact that $H^1(\Gamma)$ is characterized by single layer potentials.
\begin{lem}[\cite{Ver84}, Theorem 3.3] \label{lem:Siso}
Let $\Gamma$ be a bounded and simply connected Lipschitz surface. Then 
$$S^{\Gamma} \colon L^2(\Gamma) \to H^1(\Gamma)$$
is a continuous isomorphism.
\end{lem}

For $0 < s < 1$ we define the Sobolev-Besov space $H^{s}(\Gamma)$ via the Gagliardo-Slobodeckij norm
\begin{equation} \label{eq:gagliardo}
\|u\|_{H^{s}(\Gamma)}^2 = \|u\|_{L^2(\Gamma_\alpha)}^2 + \int_{\Gamma} \int_{\Gamma} \frac{|u(\bm{r}) - u(\bm{r'})|^2}{|\bm{r} - \bm{r'}|^{2+2s}} \, d\sigma(\bm{r}) \, d\sigma(\bm{r'}).
\end{equation}
The spaces $H^{s}(\Gamma)$, $0 < s < 1$, coincide with the real interpolation scale between $L^2(\Gamma)$ and $H^1(\Gamma)$, see for instance \cite{Triebel78}.
For $0 < s \leq 1$ we define the space of distributions $H^{-s}(\Gamma)$ as the dual space of $H^{s}(\Gamma)$ with respect to the scalar product of $L^2(\Gamma)$. Recall that $H^{-1/2}(\Gamma)$ coincides with the energy space $\mathcal{E}$, with equivalent norms.

Our goal in Section~\ref{sec:perturb} is to view the operator $K^{\Gamma}$, where $\Gamma$ has a single axially symmetric conical point, as a perturbation of $K^{\Gamma_\alpha}$, where $\Gamma_\alpha$ is a straight cone. In doing so we will encounter many integral operators with weakly singular kernels. It is well known that such kernels generate compact operators, see for example \cite{Calderon65} and \cite{Torres91}. However, we have been unable to locate a precise statement which covers all of our cases. We therefore sketch a proof of a statement which is far from sharp, but sufficient for our purposes.
\begin{lem} \label{lem:cpctweaksing}
Let $\Gamma$ be a bounded and simply connected Lipschitz surface, and let $H(\bm{r}, \bm{r}')$ be a kernel on $\Gamma \times \Gamma$ satisfying
\begin{equation} \label{eq:czo1}
|H(\bm{r}, \bm{r'})| \lesssim \frac{1}{|\bm{r}-\bm{r'}|},
\end{equation}
\begin{equation} \label{eq:czo2}
|H(\bm{r}, \bm{r'}) - H(\bm{r^\ast}, \bm{r'})| \lesssim \frac{|\bm{r} - \bm{r^\ast}|}{|\bm{r}-\bm{r'}|^2}, \quad |\bm{r}-\bm{r^\ast}| < \frac{1}{2}|\bm{r}-\bm{r'}|,
\end{equation}
and
\begin{equation} \label{eq:czo3}
|H(\bm{r'}, \bm{r}) - H( \bm{r'}, \bm{r^\ast})| \lesssim \frac{|\bm{r} - \bm{r^\ast}|}{|\bm{r}-\bm{r'}|^2}, \quad |\bm{r}-\bm{r^\ast}| < \frac{1}{2}|\bm{r}-\bm{r'}|.
\end{equation}
Then the integral operator
$$Hf(\bm{r}) = \int_{\Gamma} H(\bm{r}, \bm{r}') f(\bm{r'}) \, d\sigma(\bm{r'})$$
defines compact operators $H \colon H^{-1/2}(\Gamma) \to H^{-1/2}(\Gamma)$, $H \colon L^2(\Gamma) \to L^2(\Gamma)$, and $H \colon H^{1/2}(\Gamma) \to H^{1/2}(\Gamma)$.
\end{lem}
\begin{proof}
For $\beta < 2$ it is easy to show that the operator
$$G_\beta f(\bm{r}) = \int_{\Gamma} \frac{1}{|\bm{r}-\bm{r'}|^\beta} f(\bm{r'}) \, d\sigma(\bm{r'})$$
is bounded on $L^2(\Gamma)$, for instance by interpolating between the spaces $L^1(\Gamma)$ and $L^\infty(\Gamma)$, on which the boundedness property is evident. Next, inequalities \eqref{eq:czo1} and \eqref{eq:czo2} imply that
$$|H(\bm{r}, \bm{r'}) - H(\bm{r^\ast}, \bm{r'})| \lesssim \frac{|\bm{r} - \bm{r^\ast}|^{3/4}}{|\bm{r}-\bm{r'}|^{7/4}} + \frac{|\bm{r} - \bm{r^\ast}|^{3/4}}{|\bm{r^\ast}-\bm{r'}|^{7/4}}, \quad \bm{r}, \bm{r^\ast}, \bm{r'} \in \Gamma.$$
Hence, for $f \in L^2(\Gamma)$
$$|Hf(\bm{r}) - Hf(\bm{r^\ast})| \lesssim |r-r^\ast|^{3/4} ( G_{7/4}|f|(\bm{r}) + G_{7/4}|f|(\bm{r^\ast})).$$
From this estimate we obtain that
$$\|Hf\|_{H^{1/2}(\Gamma)} \lesssim 2\int_{\Gamma} \left[G_{7/4}|f|(\bm{r})\right]^2 \int_{\Gamma} \frac{1}{|\bm{r} - \bm{r^\ast}|^{3/2}} \, d\sigma(\bm{r^\ast}) \, d\sigma(\bm{r}) \lesssim \int_{\Gamma} \left[G_{7/4}|f|(\bm{r})\right]^2 \, d\sigma(\bm{r}) \lesssim \|f\|_{L^2(\Gamma)}.$$
Hence $H \colon L^2(\Gamma) \to H^{1/2}(\Gamma)$ is bounded. In particular $H \colon L^2(\Gamma) \to L^2(\Gamma)$ is compact, since $H^{1/2}(\Gamma)$ is compactly contained in $L^2(\Gamma)$. By \eqref{eq:czo3} the same argument yields that the $L^2(\Gamma)$-adjoint $H^\ast$ also maps $L^2(\Gamma)$ into $H^{1/2}(\Gamma)$. Equivalently, by duality, $H$ maps $H^{-1/2}(\Gamma)$ into $L^2(\Gamma)$ boundedly. Since $L^2(\Gamma)$ is compactly contained in $H^{-1/2}(\Gamma)$ it follows that $H \colon H^{-1/2}(\Gamma) \to H^{-1/2}(\Gamma)$ is compact. By duality, this is equivalent to saying that $H^{\ast} \colon H^{1/2}(\Gamma) \to H^{1/2}(\Gamma)$ is compact. Since the statement of the lemma is symmetric with respect to $H$ and $H^\ast$, it follows that also $H \colon H^{1/2}(\Gamma) \to H^{1/2}(\Gamma)$ is compact.
\end{proof}
\begin{rmk}
If $\Gamma$ is a $C^2$-surface, then $H = K^\Gamma$ satisfies the hypotheses of the lemma. Hence $K^\Gamma$ is a compact operator in this case (as is well known). Another example we have in mind is given by the kernel
$$H_j(\bm{r}, \bm{r'}) = \frac{\bm{r}_j - \bm{r'}_j}{|\bm{r}-\bm{r'}|^2},$$
where $\bm{r}_j$ denotes the $j$th coordinate of $\bm{r}$, $j=1,2,3$.
\end{rmk}
We will also make use of the fact that the Riesz transforms are bounded when $\Gamma$ is Lipschitz, which was first proven in \cite[Theorem IX]{CMM82}. See also \cite{DS85}.
        \begin{lem} \label{lem:riesz}
                Let $\Gamma$ be a bounded Lipschitz surface, or a Lipschitz graph.
                For $j=1,2,3$, let $R_j^\Gamma$ be the corresponding Riesz transform on $\Gamma$,
                $$R_j^\Gamma f(\bm{r}) = \int_{\Gamma} \frac{\bm{r}_j - \bm{r'}_j}{|\bm{r} - \bm{r'}|^3} f(\bm{r'})\, d\sigma(\bm{r'}).$$
                Then $R_j^\Gamma \colon L^2(\Gamma) \to L^2(\Gamma)$ is bounded. In fact, if 
                $$MR_j^\Gamma f(\bm{r}) = \sup_{\varepsilon > 0} \left| \int_{|\bm{r}-\bm{r'}| > \varepsilon} \frac{\bm{r}_j - \bm{r'}_j}{|\bm{r} - \bm{r'}|^3} f(\bm{r'})\, d\sigma(\bm{r'})\right|,$$
                then
                $$\|MR_j^\Gamma f(\bm{r})\|_{L^2(\Gamma)} \lesssim \|f\|_{L^2(\Gamma)}.$$
        \end{lem}
\section{Fourier analysis on a straight cone} \label{sec:straightcone}
\subsection{Spectral resolution on $L^2$} \label{sec:infcone}
Let $\Gamma_\alpha$ be the infinite straight cone with opening $2\alpha$, $0 < \alpha < \pi$, $\alpha \neq \pi/2$, parametrized by
$$\bm{r}(t, \theta) = (\sin(\alpha) t \cos \theta, \sin(\alpha) t \sin \theta, \cos(\alpha) t ), \qquad \theta \in [0, 2\pi], \; t > 0. $$
It is generated by revolution of the straight line $\gamma_\alpha(t) = (\sin(\alpha) t, \cos(\alpha) t)$, $t > 0$.
The surface element on $\Gamma_\alpha$ is given by
$$d \sigma(t, \theta) = \sin(\alpha) t \, dt \, d\theta,$$
and the outward normal $\bm{\nu}_{\bm{r}}$ by
$$\bm{\nu}_{\bm{r}} = (\cos \alpha \cos \theta, \cos \alpha \sin \theta, -\sin \alpha).$$
Note that the kernel  $K^{\Gamma_\alpha}(t, \theta, t', \theta') := K^{\Gamma_\alpha}(\bm{r}(t, \theta), \bm{r}(t', \theta'))$, defined in \eqref{eq:kdef}, only depends on $t$, $t'$, and $\theta - \theta'$,
\begin{equation} \label{eq:inv}
K^{\Gamma_\alpha}(t, \theta, t', \theta') = K^{\Gamma_\alpha}(t, \theta - \theta', t', 0).
\end{equation}
For a function $f : \Gamma_\alpha \to \mathbb{C}$, $f(t,\theta):=
f(\bm{r}(t,\theta))$, let $f_n$ be its $n$th Fourier coefficient,
$$f_n(t) = \frac{1}{\sqrt{2 \pi}} \int_0^{2\pi} e^{-in\theta} f(t, \theta) \, d\theta, \qquad t > 0,$$
so that
$$f(t, \theta) = \frac{1}{\sqrt{2\pi}} \sum_{n=-\infty}^\infty f_n(t) e^{in\theta}.$$
Then 
$$\int_{\Gamma_\alpha} |f(\bm{r})|^2 \, d\sigma(\bm{r}) = \sum_{n=-\infty}^\infty \int_0^\infty |f_n(t)|^2 \sin(\alpha) t \, dt,$$
reflecting the fact that $L^2(\Gamma_\alpha, d\sigma)$ decomposes into the direct sum
\begin{equation} \label{eq:L2decomprev}
L^2(\Gamma_\alpha, d\sigma) \simeq \bigoplus_{n=-\infty}^\infty L^2([0,\infty), \sin(\alpha) t \, dt).
\end{equation}
For $n \in \mathbb{Z}$, let 
$$K_n^{\alpha} (t, t') = \int_0^{2\pi} e^{-in\theta} K^{\Gamma_\alpha}(t, \theta, t', 0) \, d\theta, \qquad t, t' > 0.$$
Then property \eqref{eq:inv} implies that
\begin{equation}\label{eq:Kfourier}
(K^{\Gamma_\alpha}f)_n(t) = \int_0^\infty K^{\alpha}_n(t, t') f_n(t') \sin(\alpha) t' \, dt', \qquad f \in L^2(\Gamma_\alpha).
\end{equation}
If by $K_n^{\alpha}$ we also denote the associated integral operator $K_n^{\alpha} : L^2(\sin(\alpha)t \, dt) \to L^2(\sin(\alpha)t \, dt)$,
\begin{equation}\label{eq:infconeKn}
K_n^{\alpha} f(t) = \int_0^\infty K_n^{\alpha}(t, t') f(t') \sin(\alpha) t' \, dt',
\end{equation}
then we have observed that
$$K^{\Gamma_\alpha} \simeq_{\ue} \bigoplus_{n=-\infty}^\infty K_n^{\alpha}.$$

Since the kernel $K^{\Gamma_\alpha}$ is homogeneous of degree $-2$, the same is true of $K_n^\alpha$,
\begin{equation} \label{eq:homog}
K_n^\alpha (\lambda t, \lambda t') = \frac{1}{\lambda^2} K_n^\alpha(t, t'), \qquad \lambda > 0.
\end{equation}
Consider the unitary map $V : L^2(\sin(\alpha)t \, dt) \to L^2(dt/t)$,
$$ Vf(t) = \sqrt{\sin(\alpha)} tf(t).$$
Observe that
$$V\left(K_n^\alpha f\right)(t) = \sin(\alpha)\int_0^\infty \frac{t}{t'}K_n^\alpha\left(\frac{t}{t'}, 1\right) Vf(t') \, \frac{dt'}{t'} = \left(h_n^\alpha \star Vf \right)(t),$$
where $h_n^\alpha(t) = \sin(\alpha) t K_n^\alpha(t, 1)$ and $\star$, as before, denotes Mellin convolution. 

In other words, $K_n^\alpha$ is unitarily equivalent to the Mellin convolution operator on $L^2(dt/t)$ with kernel $h_n^\alpha(t)$. This allows us to determine the spectral resolution of $K_n^\alpha$ and $K^{\Gamma_\alpha}$. Before proceeding we will establish the following result on the properties of $K_n^\alpha$. Its proof is rather lengthy and depends on an explicit formula for $K_n^\alpha(t,t')$ in terms of special functions. As to not break the flow of this section we defer the proof to the Appendix.
\begin{lem} \label{lem:specestimate}
For all $t > 0$ it holds that $K^\alpha_n(1,t) = tK^\alpha_n(t,1)$.
There is a constant $C > 0$, depending only on $\alpha$, such that 
\begin{equation}\label{eq:Kdecayinf}
|K_0^\alpha(t,1)| \leq \frac{C}{t^3}, \; t \geq \frac{3}{2}, \quad |K_{n}^\alpha(t,1)| \leq \frac{C}{t^{|n|+2}}, \; t \geq \frac{3}{2}, \; n\neq 0,
\end{equation}
and such that 
\begin{equation}\label{eq:Kdecayzero}
|K_0^\alpha(t,1)| \leq C, \; t \leq \frac{1}{2}, \quad |K_{n}^\alpha(t,1)| \leq Ct^{|n|-1}, \; t \leq \frac{1}{2}, \; n\neq 0.
\end{equation}
At $t=1$, $K_n^\alpha(t,1)$ has a logarithmic singularity: there is an analytic function $G(t)$ on $[1/2, 3/2]$ such that $K_n^\alpha(t,1) - \log |1-t| G(t)$ is analytic on $[1/2, 3/2]$.

Furthermore, for every $\beta$, $-1 < \beta < 2$, the functions $b_n(t) = t^\beta K_n^\alpha(t, 1)$ satisfy
\begin{equation} \label{eq:hndecay}
\|b_n\|_{L^1(dt/t)} \lesssim \frac{1}{n}.
\end{equation}
\end{lem}
 For every $n \in \mathbb{Z}$, let $\Pi_n$ be the set
$$\Pi_n = \{\mathcal{M}h_n^\alpha(i\xi) \, : \, -\infty \leq \xi \leq \infty \} = \{ \sin(\alpha)\mathcal{M}[K_n^\alpha(\cdot, 1)](1+i\xi) \, : \, -\infty \leq \xi \leq \infty\}.$$
Since $h_n^\alpha \in L^1(dt/t)$ we have by the Riemann-Lebesgue lemma that $\mathcal{M}h_n^\alpha(i\xi)$ is a continuous function vanishing at infinity, so that $\Pi_n$ is actually a closed curve in $\mathbb{C}$. 

\begin{thm} \label{thm:infconespec} For each $n$, let $M_n \colon L^2(\mathbb{R}) \to L^2(\mathbb{R})$ be the multiplication operator
$$M_n \psi(\xi) = \sin(\alpha)\mathcal{M}[K_n^\alpha(\cdot, 1)](1+i\xi) \psi(\xi).$$ 
Then 
\begin{equation} \label{eq:L2specres}
K^{\Gamma_\alpha} \simeq_{\ue} \bigoplus_{n=-\infty}^\infty M_n.
\end{equation}
In particular, the spectrum of $K^{\Gamma_\alpha} \colon L^2(\Gamma_\alpha, d\sigma) \to  L^2(\Gamma_\alpha, d\sigma)$ is given by the union of the closed curve $\Pi_n$,
$$\sigma\left( K^{\Gamma_\alpha}, L^2\right) = \bigcup_{n=-\infty}^\infty \Pi_n.$$
The curves $\Pi_n$ tend to $0$ as $1/|n|$ when $|n| \to \infty$,
$$ \max \{ |z| \, : \, z \in \Pi_n\} \lesssim \frac{1}{|n|+1}.$$
\end{thm}
\begin{rmk}
In Theorem~\ref{thm:curveformula} we will compute $\mathcal{M}[K_n^\alpha(\cdot, 1)]$ explicitly to show that
$$\Pi_n = \left\{ \frac{P_{i\xi}^n(\cos \alpha)\dot{P}_{i\xi}^n(-\cos \alpha)-P_{i\xi}^n(-\cos \alpha)\dot{P}_{i\xi}^n(\cos \alpha)}{P_{i\xi}^n(-\cos \alpha)\dot{P}_{i\xi}^n(\cos \alpha)+P_{i\xi}^n(\cos \alpha)\dot{P}_{i\xi}^n(-\cos \alpha)} \, : \, -\infty \leq \xi \leq \infty \right\},$$
where $P_\lambda^n(x)$ denotes an associated Legendre function of the first kind, and $\dot{P}_\lambda^n$ denotes the derivative in $x$.
\end{rmk}
\begin{proof}
We have shown that 
$$K^{\Gamma_\alpha} \simeq_{\ue} \bigoplus_{n=-\infty}^\infty K_n^{\alpha},$$
where each operator $K_n^\alpha \colon L^2(\sin(\alpha) t \, dt) \to L^2(\sin(\alpha) t \, dt)$ is unitarily equivalent to the operator of Mellin convolution with $h_n^\alpha$ on $L^2(dt/t)$. Hence \eqref{eq:L2specres} follows from applying the unitary Mellin transform operator of \eqref{eq:mellinunitary}. Note also that
$$\|M_n\| \leq  \|h_n^\alpha\|_{L^1(dt/t)} \lesssim \frac{1}{|n|+1}$$ 
by \eqref{eq:young} and \eqref{eq:hndecay}. The spectrum of $M_n$ is equal to $\Pi_n$, and therefore
$$ \max \{ |z| \, : \, z \in \Pi_n\} \leq \|M_n\| \lesssim \frac{1}{|n|+1}.$$
Hence
$$\sigma\left( K^{\Gamma_\alpha}, L^2\right) = \mathrm{clos}_{\mathbb{C}} \bigcup_{n=-\infty}^\infty \Pi_n = \bigcup_{n=-\infty}^\infty \Pi_n,$$
where the last equality follows since $\Pi_n$ are closed curves tending to the origin as $|n| \to \infty$.
\end{proof}
Another interpretation of Theorem~\ref{thm:infconespec} is the following. Suppose that $\widehat{h}_n^\alpha(i\xi)$ is a point in $\sigma\left( K^{\Gamma_\alpha}, L^2\right)$. The change of variable $t = s/s'$ and the homogeneity \eqref{eq:homog} then gives us that
\begin{equation} \label{eq:l2unboundstate}
\widehat{h}_n^\alpha(i\xi) = \int_0^\infty t^{i\xi + 1} K_n^{\alpha}(t, 1) \sin(\alpha) \, \frac{dt}{t} = s^{i\xi+1} \int_0^\infty (s')^{-i\xi - 1} K_n^\alpha(s, s') \sin(\alpha) s' \, ds'.
\end{equation}
Comparing with \eqref{eq:infconeKn} and letting $f(t) = t^{-i\xi -1}$ we see that
$$K_n^\alpha f (t) = \widehat{h}_n^\alpha(i\xi)  f(t).$$
Hence $f(t, \theta) = t^{-i\xi -1}e^{i n \theta}$ is an eigenfunction of $K^{\Gamma_\alpha}$ for the eigenvalue $\widehat{h}_n^\alpha(i\xi)$. The function $f(t, \theta) = t^{-i\xi -1}e^{i n \theta}$ just barely fails to belong to $L^2(\Gamma_\alpha, d \sigma)$ and we think of it as a generalized eigenfunction for the point $\widehat{h}_n^\alpha(i\xi)$ of the spectrum.

Note that the parity relation $K_n^\alpha(1, t) = t K_n^\alpha(t,1)$, the change of variable $t = s'/s$ and the homogeneity \eqref{eq:homog} also gives that
$$\widehat{h}_n^\alpha(i\xi) = \int_0^\infty t^{i\xi} K_n^{\alpha}(1, t) \sin(\alpha)\, \frac{dt}{t} = s^{-(i \xi -2)} \int_0^\infty (s')^{i\xi - 2} K_n^\alpha(s, s') \sin(\alpha) s' \, ds'.$$
Hence $f(t, \theta) = t^{i\xi -2}e^{i n \theta}$ is a second eigenfunction of $K^{\Gamma_\alpha}$ for the eigenvalue $\widehat{h}_n^\alpha(i\xi)$.
\subsection{The transmission problem} \label{sec:transmission}
It will follow from Lemma~\ref{lem:mellinsmooth} that $\mathcal{M}[K_n^\alpha(\cdot, 1)](\zeta)$ is well-defined and holomorphic in the strip $0 < \Re\mathrm{e} \, \zeta < 3$. Considering a value $\sin(\alpha)\mathcal{M}[K_n^\alpha(\cdot, 1)](3/2+i\xi)$, $\xi \in \mathbb{R}$, we find as in \eqref{eq:l2unboundstate} that
$$
\sin(\alpha)\mathcal{M}[K_n^\alpha(\cdot, 1)](3/2+i\xi) = \int_0^\infty t^{i\xi + 3/2} K_n^{\alpha}(t, 1) \sin(\alpha) \, \frac{dt}{t} = s^{i\xi+3/2} \int_0^\infty (s')^{-i\xi - 3/2} K_n^\alpha(s, s') \sin(\alpha) s' \, ds'.
$$
Hence 
\begin{equation} \label{eq:Kenergyeigfcns}
e_{\xi, n}(t, \theta) = t^{-i\xi - 3/2} e^{in\theta}
\end{equation} 
is an eigenfunction of $K^{\Gamma_\alpha}$ to the eigenvalue $\sin(\alpha) \mathcal{M}[K_n^\alpha(\cdot, 1)](3/2+i\xi)$. We will see that this function just barely fails to lie in $\mathcal{E}$ and we therefore think of it as a generalized eigenfunction for this space. The parity relation $K_n^\alpha(1,t) = tK_n^\alpha(t,1)$ also yields a second generalized eigenfunction $d_{\xi, n}(t, \theta) = t^{i\xi - 3/2} e^{in\theta}$. 

For a complex number  $\epsilon \neq 1$, we now consider the transmission problem 
\begin{equation} \label{eq:electrostatic}
\begin{cases}
U \textrm{ continuous on }\mathbb{R}^3 \setminus \{0\}, \\
\Delta U(\bm{r}) = 0, \quad \bm{r} \in \mathbb{R}^3 \setminus \Gamma_\alpha, \\
\partial_{\bm{\nu}}^{\exte} U(\bm{r}) = \epsilon \partial_{\bm{\nu}}^{\inte} U(\bm{r}), \quad 0 \neq \bm{r} \in \Gamma_\alpha.
\end{cases}
\end{equation}
To solve it, we make an ansatz with the single layer potential $S^{\Gamma_\alpha}$ of $\Gamma_\alpha$, see \eqref{eq:Sdef}.
For points $\bm{r} = \bm{r}(t, \theta) \in \Gamma_\alpha$, using the same parametrization of $\Gamma_\alpha$ as in Section~\ref{sec:infcone}, consider for $n \in \mathbb{Z}$ the operator $S_n^\alpha$,
$$S_n^\alpha f(t) = \int_0^\infty S_n^\alpha(t, t')f(t') \sin(\alpha) t' \, dt',$$
 with kernel 
$$S_n^\alpha(t, t') = \int_{0}^{\infty} e^{-in\theta} S^{\Gamma_\alpha}(t, \theta, t', 0) \, d\theta.$$
The kernel of $S^{\Gamma_\alpha}$ is rotationally invariant, so just as for $K^{\Gamma_\alpha}$ the Fourier coefficients of $S^{\Gamma_\alpha}$ can be obtained from $S_n^\alpha$,
$$(S^{\Gamma_\alpha}f)_n(t) = S_n^\alpha f_n(t).$$
Hence, for sufficiently nice $f$ it holds that
\begin{equation} \label{eq:Sfourier}
S^{\Gamma_\alpha} f (t, \theta) = \frac{1}{\sqrt{2\pi}} \sum_{n=-\infty}^\infty S_n^\alpha f_n(t) e^{in\theta}.
\end{equation}
If $f \in L^1\left(\Gamma_\alpha, \frac{d\sigma(t, \theta)}{1+t}\right) \cap C^1((0, \infty))$, say, the jump formulas hold with pointwise convergence,
$$\partial_{\bm{\nu}}^{\inte} S^{\Gamma_\alpha}f(\bm{r}) = f(\bm{r}) - Kf(\bm{r}), \quad \partial_{\bm{\nu}}^{\exte} S^{\Gamma_\alpha}f(\bm{r}) = -f(\bm{r}) - Kf(\bm{r}), \quad \bm{r} \in \Gamma_\alpha \setminus \{0\}.$$
It follows that $S^{\Gamma_\alpha}f$ solves the transmission problem if and only if
$$(K^{\Gamma_\alpha} - z)f = 0, \quad z = -\frac{1+\epsilon}{1 - \epsilon}.$$
In particular, $S^{\Gamma_\alpha}e_{\xi, n}$ solves the transmission problem \eqref{eq:electrostatic} for the number $\epsilon$ such that 
\begin{equation} \label{eq:elecer}
\sin(\alpha)\mathcal{M}[K_n^\alpha(\cdot, 1)](3/2+i\xi) = -\frac{1+\epsilon}{1 - \epsilon}.
\end{equation}
\begin{lem} \label{lem:smellinconstant}
Let $f_{\xi, n}(t, \theta) = t^{-i\xi - 1/2} e^{in\theta}$. There is a constant $C = C_n(\xi, \alpha)$ such that
$$S^{\Gamma_\alpha}e_{\xi, n}|_{\Gamma_\alpha} = Cf_{\xi, n}.$$
\end{lem}
\begin{proof}
For $\bm{r} \in \Gamma_\alpha$, we may compute $S^{\Gamma_\alpha}e_{\xi, n}(\bm{r})$ as follows, using the change of variable $t' = st$ and that $S_n^\alpha$ is homogeneous of degree $-1$,
\begin{align*}
\frac{1}{\sqrt{2\pi}}(S^{\Gamma_\alpha}e_{\xi, n})_n(t) &= \int_0^\infty t^{-i\xi - 3/2} S_n^\alpha(t, t') \sin(\alpha) t' \, dt' \\ &= \int_0^\infty t^{-i\xi - 1/2} s^{-i\xi - 3/2} S_n^\alpha(1, s) \sin(\alpha) s \, ds = t^{-i\xi - 1/2} \sin(\alpha) \mathcal{M}[S_n^\alpha(1, \cdot)](-i\xi + 1/2).
\end{align*}
This shows that $S^{\Gamma_\alpha}e_{\xi, n}|_{\Gamma_\alpha} = Cf_{\xi, n}$, for $C=\sin(\alpha) \mathcal{M}[S_n^\alpha(1, \cdot)](-i\xi + 1/2)$. 
\end{proof}

On the other hand, explicit computations have been made for the transmission problem \eqref{eq:electrostatic} \cite{KCHWA14,MRL77,SPG14}. Write $\bm{r} \sim (t, \varphi, \theta)$ in spherical coordinates and let
\begin{equation*}
V(\bm{r}) = 
\begin{cases} 
t^{-i\xi-1/2}e^{in\theta}P_{i\xi-1/2}^n(\cos(\varphi)), \quad \bm{r} \in \inte(\Gamma_\alpha), \\ 
Dt^{-i\xi-1/2}e^{in\theta}P_{i\xi-1/2}^n(-\cos(\varphi)), \quad \bm{r} \in \exte(\Gamma_\alpha),
 \end{cases}
\quad D = \frac{P_{i\xi-1/2}^n(\cos(\alpha))}{P_{i\xi-1/2}^n(-\cos(\alpha))}.
\end{equation*}
Then $V$ solves the transmission problem for
$$\epsilon = -\frac{P_{i\xi - 1/2}^n(\cos \alpha)\dot{P}_{i\xi - 1/2}^n(-\cos \alpha)}{P_{i\xi - 1/2}^n(-\cos \alpha)\dot{P}_{i\xi - 1/2}^n(\cos \alpha)}.$$
Here $P_\lambda^n(x)$ denotes an associated Legendre function of the first kind, and $\dot{P}_\lambda^n$ denotes its derivative in $x$, see the Appendix. We implement the Legendre function through the formula \cite{mw:legendrealt},
\begin{equation} \label{eq:legendreformula}
P^n_\lambda(x) = (1-x^2)^{n/2}(-\lambda)_n(\lambda+1)_n \sum_{k=0}^\infty \frac{(n-\lambda)_k (n+\lambda+1)_k}{k!(k+n)!2^{n+k}}(1-x)^k, \quad |1-x| < 2,
\end{equation}
where $(\lambda)_n$ denotes the Pochhammer symbol,
\begin{equation} \label{eq:pochhammer}
(\lambda)_n = \frac{\Gamma(\lambda+n)}{\Gamma(\lambda)},
\end{equation}
$\Gamma$ denoting the usual gamma function.
Note that $V|_{\Gamma_\alpha} = C'f_{\xi, n}$, where $C' = P_{i\xi-1/2}^n(\cos(\alpha)) \neq 0$ is a constant.
\begin{lem} \label{lem:SVlem}
$Se_{\xi, n} = (C/C')V$.
\end{lem}
\begin{proof}
We consider the interior of $\Gamma_\alpha$ to be the set where $0 < \varphi < \alpha$. Treating the interior first, we have found two solutions of the Dirichlet problem
\begin{equation*}
\begin{cases}
\Delta U(\bm{r}) = 0, \quad \bm{r} \in \inte(\Gamma_\alpha), \\ 
U|_{\Gamma_\alpha}(\bm{r}) = f_{\xi, n}(\bm{r}), \quad \bm{r} \in \Gamma_\alpha \setminus \{0\},
\end{cases}
\end{equation*}
and we want to show that they are the same.
For convenience, we may apply the translation $\bm{r} \mapsto \bm{r} + (0,0,1)$, so that $\Gamma_\alpha$ is the cone with vertex at $(0,0,1)$, and $U$ and $f_{\xi, n}$ instead denote the translated functions. 
Let $B(0, R)$ be a ball with center at the origin and radius $R$ chosen so small that $\overline{B(0, R)} \subset \exte(\Gamma_\alpha)$. For $\bm{r} \in \mathbb{R}^3$, let $\bm{r}^*$ denote its inversion in the surface of $B(0, R)$, 
$$\bm{r}^* = \frac{R^2}{|\bm{r}|^2} \bm{r}.$$
Let $\Gamma_\alpha^*$ be the inversion of $\Gamma_\alpha$. It is a bounded, rotationally symmetric lens domain, smooth except for two conical points at $0$ and $(0,0, R^2)$. In particular it is a Lipschitz domain.

The Kelvin transform of $U$ is the function
$$U^*(\bm{r}^*) = |\bm{r}^*|^{-1} U(\bm{r}),$$
harmonic on $\inte(\Gamma_\alpha^*)$. We similarly define $f_{\xi, n}^*$. It is continuous on $\Gamma_\alpha^* \setminus \{0, (0,0,R^2)\}$ and satisfies that
$$f_{\xi, n}^*(\bm{r}) \simeq \frac{1}{\sqrt{|\bm{r}|}}, \quad \Gamma_\alpha^* \ni \bm{r} \to 0$$
and
$$f_{\xi, n}^*(\bm{r}) \simeq \frac{1}{\sqrt{|\bm{r} - (0,0,R^2)|}}, \quad \Gamma_\alpha^* \ni \bm{r} \to (0,0,R^2).$$
Let $d\sigma^*$ denote surface measure on $\Gamma_\alpha^*$. In a Lipschitz parametrization $\bm{r} = \bm{r}(x,y)$ of $\Gamma_\alpha^*$, it is clear that 
$$\frac{d\sigma^*(\bm{r})}{d\bm{r}} \simeq \mathrm{dist}(\bm{r}, 0), \quad \Gamma_\alpha^* \ni \bm{r} \to 0,$$
and 
$$\frac{d\sigma^*(\bm{r})}{d\bm{r}} \simeq \mathrm{dist}(\bm{r}, (0,0,R^2)), \quad \Gamma_\alpha^* \ni \bm{r} \to (0,0,R^2),$$
cf. Section~\ref{sec:perturb}. It follows that $f_{\xi, n}^* \in L^2(\Gamma_\alpha, d\sigma^*)$. Hence, we have produced two solutions to the interior Dirichlet problem
\begin{equation*}
\begin{cases}
\Delta U^*(\bm{r}) = 0, \quad \bm{r} \in \inte(\Gamma_\alpha^*), \\ 
U^*|_{\Gamma_\alpha^*}(\bm{r}) = f_{\xi, n}^*(\bm{r}), \quad \bm{r} \in \Gamma_\alpha^* \setminus \{0, (0,0,R^2)\}.
\end{cases}
\end{equation*}
The equality on the boundary in particular holds almost everywhere $d\sigma^*$. But since $f_{\xi, n}^* \in L^2$ the Dirichlet problem is uniquely determined, by Dahlberg's theorem \cite{Dahl79}. Hence $Se_{\xi, n}(\bm{r}) = (C/C')V(\bm{r})$ on $\inte(\Gamma_\alpha)$. Equality in the exterior domain is proven in the same way.
\end{proof}
Since $\epsilon$ is unique for a given function, we deduce that
$$\sin(\alpha)\mathcal{M}[K_n^\alpha(\cdot, 1)](3/2+i\xi)  = \frac{P_{i\xi - 1/2}^n(\cos \alpha)\dot{P}_{i\xi - 1/2}^n(-\cos \alpha) - P_{i\xi - 1/2}^n(-\cos \alpha)\dot{P}_{i\xi - 1/2}^n(\cos \alpha)}{P_{i\xi - 1/2}^n(-\cos \alpha)\dot{P}_{i\xi - 1/2}^n(\cos \alpha)+P_{i\xi - 1/2}^n(\cos \alpha)\dot{P}_{i\xi - 1/2}^n(-\cos \alpha)}.$$
It follows that for every $\zeta$ in the strip $0 < \Re\mathrm{e} \, \zeta < 3$ we have that
$$\sin(\alpha)\mathcal{M}[K_n^\alpha(\cdot, 1)](\zeta)  = \frac{P_{\zeta - 2}^n(\cos \alpha)\dot{P}_{\zeta - 2}^n(-\cos \alpha) - P_{\zeta-2}^n(-\cos \alpha)\dot{P}_{\zeta - 2}^n(\cos \alpha)}{P_{\zeta - 2}^n(-\cos \alpha)\dot{P}_{\zeta - 2}^n(\cos \alpha)+P_{\zeta - 2}^n(\cos \alpha)\dot{P}_{\zeta - 2}^n(-\cos \alpha)},$$
since both sides are meromorphic in the strip and they agree on the line $\zeta = 3/2 + i\xi$. In particular, letting $\zeta = 1 + \xi$ yields an explicit formula for the curve $\Pi_n$. Note that we have made use of the parity identity $P^n_\lambda = P^n_{-\lambda-1}$. Incidentally, this identity is consistent with the existence of two eigenfunctions for each eigenvalue.
\begin{thm} \label{thm:curveformula}
The closed curve $\Pi_n$ is given by 
$$\Pi_n = \left\{ \frac{P_{i\xi}^n(\cos \alpha)\dot{P}_{i\xi}^n(-\cos \alpha) - P_{i\xi}^n(-\cos \alpha)\dot{P}_{i\xi}^n(\cos \alpha)}{P_{i\xi}^n(-\cos \alpha)\dot{P}_{i\xi}^n(\cos \alpha)+P_{i\xi}^n(\cos \alpha)\dot{P}_{i\xi}^n(-\cos \alpha)} \, : \, -\infty \leq \xi \leq \infty \right\} .$$
\end{thm}
Lemma~\ref{lem:SVlem} also lets us compute the constant $C = C_n(\xi, \alpha)$ explicitly.
\begin{lem}\label{lem:smellinconstantcomp}
The constant $C_n(\xi, \alpha)$ of Lemma~\ref{lem:smellinconstant} is given by
$$C_n(\xi, \alpha) = -\frac{2}{\sin \alpha} \frac{P_{i\xi - 1/2}^n(\cos \alpha) P_{i\xi - 1/2}^n(-\cos \alpha)}{P_{i\xi - 1/2}^n(\cos \alpha)\dot{P}_{i\xi - 1/2}^n(-\cos \alpha) + P_{i\xi - 1/2}^n(-\cos \alpha)\dot{P}_{i\xi - 1/2}^n(\cos \alpha)}.$$
In particular, $C_n(\xi, \alpha)$ is uniformly bounded in $n \in \mathbb{Z}$ and $\xi \in \mathbb{R}$, and $C_n(-\xi, \alpha) = C_n(\xi, \alpha)$ and $C_n(\xi, \alpha) > 0$ for all such $n$ and $\xi$.
\end{lem}
\begin{proof}
For $\bm{r} \in \Gamma_\alpha$ we have that
$$\partial_{\bm{\nu}}^{\inte} V(\bm{r}) = \frac{1}{t}\frac{\partial V}{\partial \varphi}(\bm{r}) = -\sin{\alpha} \dot{P}^n_{i\xi-1/2}(\cos(\alpha)) t^{-i\xi - 3/2} e^{in\theta}.$$
Since $V$ solves the transmission problem we also know that
$$\partial_{\bm{\nu}}^{\exte} V(\bm{r}) - \partial_{\bm{\nu}}^{\inte} V(\bm{r}) = (\epsilon - 1)\partial_{\bm{\nu}}^{\inte} V(\bm{r}).$$
On the other hand, having established that $V = (C'/C) S e_{\xi, n}$ in Lemma~\ref{lem:SVlem}, the jump formulas for the single layer potential give us that
$$\partial_{\bm{\nu}}^{\exte} V(\bm{r}) - \partial_{\bm{\nu}}^{\inte} V(\bm{r}) = (C'/C)\left(\partial_{\bm{\nu}}^{\exte} Se_{\xi, n}(\bm{r}) - \partial_{\bm{\nu}}^{\inte} Se_{\xi, n}(\bm{r})\right) = -2(C'/C)e_{\xi, n}(t, \theta)$$
Recalling that $C' = P_{i\xi-1/2}^n(\cos(\alpha))$ and comparing the two expressions yields the explicit formula.

From \eqref{eq:legendreformula} we have that
$$P^n_{i \xi - 1/2}(x) = (1-x^2)^{n/2}(1/2 - i\xi)_n(1/2+i\xi)_n \sum_{k=0}^\infty \frac{(n+ 1/2 - i \xi)_k (n+1/2 + i\xi)_k}{k!(k+n)!2^{n+k}}(1-x)^k.$$
Since the Pochhammer symbol is given by
$$(\lambda)_n = \lambda(\lambda+1)(\lambda + 2) \cdots (\lambda + n -1) $$
it is clear that $P^n_{i \xi - 1/2}(x) > 0$ and that $\dot{P}_{i\xi - 1/2}^n(x) < 0$ for every $\xi$ and $-1 < x < 1$. Hence $C_n(\xi, \alpha) > 0$. It is also clear that $C_n(-\xi, \alpha) = C_n(\xi, \alpha)$. The easiest way to see that $C_n$ is uniformly bounded is to recall from Lemma~\ref{lem:SVlem} that $C_n(\xi, \alpha)$ is the Mellin transform of
$$b_n(t) = \sin(\alpha) \sqrt{t} S_n^\alpha(1, t).$$
It is clear that
$$\sup_n \|b_n\|_{L^2(dt/t)} < \infty,$$
and hence the statement follows from an obvious estimate.
\end{proof}
\subsection{Spectral resolution on $\mathcal{E}$} \label{sec:Especres}
Let $\mathcal{S}$ denote the space of smooth compactly supported functions in $(0,\infty)$, and let $\mathcal{E}_n$ denote the Hilbert space completion of $\mathcal{S}$ in the positive definite scalar product
$$\langle f_n, g_n \rangle_{\mathcal{E}_n} = \langle S_n^\alpha f_n, g_n \rangle_{L^2(\sin(\alpha) t \,dt)}.$$
Since
$$\|f\|_{\mathcal{E}}^2 = \langle S^{\Gamma_{\alpha}}f, f \rangle_{L^2(\Gamma_\alpha)},$$
we deduce from \eqref{eq:L2decomprev} and \eqref{eq:Sfourier} that
\begin{equation} \label{eq:Edecomp}
\mathcal{E} \simeq \bigoplus_{n=-\infty}^\infty \mathcal{E}_n.
\end{equation}
By \eqref{eq:Kfourier} it follows that $K^{\Gamma_\alpha} \colon \mathcal{E} \to \mathcal{E}$ acts diagonally with respect to the decomposition \eqref{eq:Edecomp},
$$K^{\Gamma_\alpha} \simeq_{\ue} \bigoplus_{n=-\infty}^\infty K_n^\alpha,$$
where $K_n^\alpha \colon \mathcal{E}_n \to \mathcal{E}_n$ is considered as an operator on $\mathcal{E}_n$.

To understand the operator $K_n^\alpha$, let $V \colon \mathcal{E}_n \to \widetilde{\mathcal{E}}_n$ be the unitary map
$$Vf(t) = \sqrt{t}f(t),$$
where the space $\widetilde{\mathcal{E}}_n$ is defined by the requirement that $V$ be unitary. It is the completion of $\mathcal{S}$ under the scalar product
$$\langle f, g \rangle_{\widetilde{\mathcal{E}}_n} = \langle V^{-1} S_n^\alpha V^{-1} f, g \rangle_{L^2(\sin(\alpha) t \, dt)}.$$
Note that 
$$\widetilde{S}_n^\alpha = V^{-1} S_n^\alpha V^{-1}$$
is a positive definite (in particular symmetric) unbounded operator on $L^2(\sin(\alpha) t \, dt)$. By polarizing the Plancherel formula \eqref{eq:plancharel} and using that $\widetilde{S}_n^\alpha$ is symmetric in the $L^2(\sin(\alpha)t\, dt)$-pairing we obtain, initially for $f, g \in \mathcal{S}$, that
\begin{align*}
\langle \widetilde{S}_n^\alpha f, g \rangle_{L^2(\sin(\alpha) t \, dt)} &= \frac{\sin(\alpha)}{2\pi} \int_{-\infty}^{\infty} \int_0^\infty t^{i\xi-1} \widetilde{S}_n^\alpha f(t) t \, dt \overline{\int_0^\infty t^{i\xi-1} g(t) t \, dt} \, d\xi \\
&= \frac{\sin(\alpha)}{\sqrt{2\pi}} \int_{-\infty}^{\infty} C_n(\xi) \int_0^\infty t^{i\xi-1}  f(t) t \, dt \overline{\int_0^\infty t^{i\xi-1} g(t) t \, dt} \, d\xi.
\end{align*}
Here $C_n(\xi) = C_n(\xi, \alpha) > 0$ is the constant of Lemma~\ref{lem:smellinconstant} and Lemma~\ref{lem:smellinconstantcomp}, now interpreted as a function of $\xi$. Hence
$$\widetilde{W} f(\xi) = \int_0^\infty t^{i\xi-1}  f(t) t \, dt$$
defines a unitary map 
$$\widetilde{W} \colon \widetilde{\mathcal{E}}_n \to L^2\left(\frac{\sin(\alpha)}{\sqrt{2\pi}}C_n(\xi) \, d\xi\right),$$
since $\widetilde{W}$ maps $\mathcal{S}$ to a dense subset of $L^2(C_n(\xi) \, d\xi)$.

Next we unitarily transfer $K_n^\alpha \colon \mathcal{E}_n \to \mathcal{E}_n$ to the operator $\widetilde{K}_n^\alpha \colon \widetilde{\mathcal{E}}_n \to \widetilde{\mathcal{E}}_n$,
$$\widetilde{K}_n^\alpha = V K_n^\alpha V^{-1}.$$ 
Let $F_n$ denote the function of Section~\ref{sec:transmission},
$$F_n(\xi) = \sin(\alpha)\mathcal{M}[K_n^\alpha(\cdot, 1)](3/2+i\xi)  = \frac{P_{i\xi - 1/2}^n(\cos \alpha)\dot{P}_{i\xi - 1/2}^n(-\cos \alpha) - P_{i\xi - 1/2}^n(-\cos \alpha)\dot{P}_{i\xi - 1/2}^n(\cos \alpha)}{P_{i\xi - 1/2}^n(-\cos \alpha)\dot{P}_{i\xi - 1/2}^n(\cos \alpha)+P_{i\xi - 1/2}^n(\cos \alpha)\dot{P}_{i\xi - 1/2}^n(-\cos \alpha)}.$$
The change of variable $t = s'/s$ yields that
$$F_n(\xi) = \int_0^\infty t^{i\xi+3/2}K_n^\alpha(t,1) \sin(\alpha) \, \frac{dt}{t} = s^{-i\xi + 1 /2}\int_0^\infty (s')^{i\xi - 1/2} K_n(s', s)\sin(\alpha) s' \, ds'.$$
Hence $(K_n^\alpha)^\ast(s^{i\xi - 1/2})(t) = F_n(\xi) t^{i\xi - 1/2}$, where the adjoint is taken with respect to the scalar product of $L^2(\sin(\alpha) t \, dt)$. Therefore,
\begin{align*}
\langle \widetilde{K}_n^\alpha f, g \rangle_{\widetilde{\mathcal{E}}_n} &=  \langle \widetilde{S}_n^\alpha \widetilde{K}_n^\alpha f, g \rangle_{L^2(\sin(\alpha) t \, dt)} \\ 
&= \frac{\sin(\alpha)}{\sqrt{2\pi}} \int_{-\infty}^{\infty} C_n(\xi) \int_0^\infty t^{i\xi-1/2}  K_n^{\alpha}V^{-1}f(t) t \, dt \overline{\int_0^\infty t^{i\xi-1} g(t) t \, dt} \, d\xi, \\
&= \frac{\sin(\alpha)}{\sqrt{2\pi}} \int_{-\infty}^{\infty} F_n(\xi)C_n(\xi)  \int_0^\infty t^{i\xi-1} f(t) t \, dt \overline{\int_0^\infty t^{i\xi-1} g(t) t \, dt} \, d\xi,
\end{align*}
where we have used that $F_n(\xi) = F_n(-\xi)$ (cf. the proof of Lemma~\ref{lem:smellinconstantcomp}).
It follows that
$$\widetilde{K}_n^\alpha = \widetilde{W}^{-1} M_{F_n} \widetilde{W},$$
where $M_{F_n} \colon L^2\left(\frac{\sin(\alpha)}{\sqrt{2\pi}}C_n(\xi) \, d\xi\right) \to L^2\left(\frac{\sin(\alpha)}{\sqrt{2\pi}}C_n(\xi) \, d\xi\right)$ denotes the operator of multiplication by $F_n$. Since $C_n(\xi)$ is a strictly positive function it follows that $\widetilde{K}_n^\alpha$ is unitarily equivalent to the same multiplication operator $M_{F_n} \colon L^2(\mathbb{R}) \to L^2(\mathbb{R})$ acting on the usual $L^2$-space of the real line.

We have realized the spectral theorem for $K^{\Gamma_\alpha} \colon \mathcal{E} \to \mathcal{E}$.

\begin{thm} \label{thm:Especres}
For $n \in \mathbb{Z}$, let $F_n$ be the real-valued function
$$F_n(\xi) = \frac{P_{i\xi - 1/2}^n(\cos \alpha)\dot{P}_{i\xi - 1/2}^n(-\cos \alpha) - P_{i\xi - 1/2}^n(-\cos \alpha)\dot{P}_{i\xi - 1/2}^n(\cos \alpha)}{P_{i\xi - 1/2}^n(-\cos \alpha)\dot{P}_{i\xi - 1/2}^n(\cos \alpha)+P_{i\xi - 1/2}^n(\cos \alpha)\dot{P}_{i\xi - 1/2}^n(-\cos \alpha)},$$
and let $M_{F_n} \colon L^2( \mathbb{R}) \to L^2( \mathbb{R})$ denote the operator of multiplication by $F_n$. Then $K^{\Gamma_\alpha} \colon \mathcal{E} \to \mathcal{E}$ is unitarily equivalent to the direct sum of the operators $M_{F_n}$,
$$K^{\Gamma_\alpha} \simeq_{\ue} \bigoplus_{n=-\infty}^{\infty} M_{F_n}.$$
In particular, letting $\Sigma_n$ be the interval
$$\Sigma_n = \{F_n(\xi) \, : \, -\infty \leq \xi \leq \infty\},$$
we have that
$$\sigma(K^{\Gamma_\alpha}, \mathcal{E}) = \bigcup_{n=-\infty}^\infty \Sigma_n.$$
\end{thm}
\begin{rmk}
That $F_n(\xi)$ is real-valued follows abstractly from our considerations, but can also be seen directly from the explicit formula. It follows from Lemma~\ref{lem:specestimate} that the sets $\Sigma_n$ shrink to zero as $|n|\to\infty$ (on the order of $1/|n|$), compare with the proof of Theorem~\ref{thm:infconespec}.
\end{rmk}

\section{Perturbation of a straight cone} \label{sec:perturb}
We consider a surface $\Gamma$ obtained by revolving a curve $\gamma$ in the $xz$-plane around the $z$-axis, see Figure~\ref{fig:surface} for two examples. We suppose that $\gamma$ is a simple $C^5$-curve, $r = \gamma(t) = (\gamma_1(t), \gamma_2(t))$, $0 \leq t \leq 1$, such that $\gamma(0) = 0$, $\gamma_1(t) > 0$ for $0 < t < 1$, $\gamma_1(1) = 0$, $\gamma_2(1) > 0$, and $\gamma_2'(1) = 0$. We normalize the curve so that $|\gamma'(0)|=1$ and assume that $\gamma_1'(0) \neq 0$ and $\gamma_2'(0) \neq 0$. Let $0 < \alpha < \pi$, $\alpha \neq \pi/2$, be the angle which $\gamma'(0)$ makes with the $z$-axis.  Let $\gamma_{c}$ be a curve of the same type as $\gamma$ such that 
$$\gamma_c(t) = (\sin (\alpha) t, \cos (\alpha) t), \qquad 0 \leq t \leq 1/2,$$
and denote its surface of revolution be $\Gamma_c$. The goal in this section is to establish that $K^\Gamma \simeq K^{\Gamma_c}$, on $L^2(\Gamma)$ and on $\mathcal{E}$, so that we may study the essential spectrum of $K^\Gamma$ by considering $K^{\Gamma_c}$.

$\Gamma$ has the parametrization
\begin{equation} \label{eq:gammaparamet2}
\bm{r}(t, \theta) = (\gamma_1(t)\cos\theta, \gamma_1(t) \sin\theta, \gamma_2(t)), \qquad \theta \in [0, 2\pi], \; 0 \leq t \leq 1,
\end{equation}
and therefore
$$d \sigma(t, \theta) = \gamma_1(t)|\gamma'(t)| \, dt \, d\theta.$$
When $t \to 0$ we have by assumption that $\gamma_1(t)|\gamma'(t)| = \sin (\alpha)t + O(t^2)$ and similarly the outward unit normal $\bm{\nu}_{\bm{r}}$ satisfies
\begin{equation} \label{eq:normalperturb}
\bm{\nu}_{\bm{r}} = (\cos \alpha \cos \theta, \cos \alpha \sin \theta, -\sin \alpha) + O(t).
\end{equation}
As we did earlier for the infinite straight cone, we write $K^\Gamma(t, \theta, t', \theta') = K^\Gamma(\bm{r}(t, \theta), \bm{r}(t', \theta'))$, and adopt a similar convention for other functions and kernels on $\Gamma$. 

We will first study the action of $K^\Gamma$ on $L^2(\Gamma)$, and begin with the following simple lemma.
\begin{lem} \label{lem:multcomp}
Let $b(t)$ be a bounded function on $[0,1]$ such that $b(t) = o(1)$ as $t \to 0$. Let $M_b$ denote the operator of multiplication by $b(t)$ on $L^2(\Gamma, d \sigma)$. Then $K^\Gamma M_b$ and $M_b K^\Gamma$ are compact on $L^2(\Gamma, d \sigma)$.
\end{lem}
\begin{proof}
For every $0 < \varepsilon < 1$, let $\rho_\varepsilon \in C_c^\infty([0, 2\varepsilon))$ be such that $\rho_\varepsilon(t) = 1$ for $t \in [0, \varepsilon]$ and $0 \leq \rho_\varepsilon(t) \leq 1$ for all $t$.  $K^\Gamma M_b M_{1 - \rho_\varepsilon}$ has a weakly singular kernel and is therefore a compact operator by Lemma~\ref{lem:cpctweaksing}, cf. \cite[Lemma 1]{PP16}. But for $f \in L^2(\Gamma)$ it holds that
$$\|K^\Gamma M_b f - K^\Gamma M_b M_{1 - \rho_\varepsilon} f\|_{L^2} = \|K^\Gamma M_b M_{\rho_\varepsilon} f\|_{L^2} \leq \|K^\Gamma\|_{L^2(\Gamma) \to L^2(\Gamma)} \|b\rho_\varepsilon f\|_L^2  = o(1)\|f\|_{L^2}.$$
Hence the compact operators $K^\Gamma M_b M_{1 - \rho_\varepsilon}$ converge uniformly to $K^\Gamma M_b$ as $\varepsilon \to 0$, and therefore $K^\Gamma M_b$ is compact. A similar argument shows that $M_b K^\Gamma$ is compact.
\end{proof}
 Having picked a parametrization of $\Gamma$, we may write down the action of $K^\Gamma$ explicitly,
\begin{equation} \label{eq:KGaction}
K^\Gamma f(t, \theta) = \frac{1}{2\pi} \int_0^1 \int_0^{2 \pi} K^\Gamma(t, \theta, t', \theta') f(t', \theta') \, d\sigma(t', \theta'), \quad f \in L^2(\Gamma, d\sigma(t, \theta)).
\end{equation}
Using this formula as the definition, we may also consider $K^\Gamma$ as an operator on $$L^2(\Gamma_c) = L^2([0,1] \times [0, 2\pi], \, d\sigma_c(t, \theta)).$$ The next lemma shows that this renorming only amounts to perturbing $K^\Gamma$ by a compact operator.
\begin{lem} \label{lem:renorm}
$K^\Gamma \colon L^2(\Gamma, d\sigma(t, \theta)) \to L^2(\Gamma, d\sigma(t, \theta))$ is unitarily equivalent to a compact perturbation of $$K^\Gamma \colon L^2(\Gamma_c,  \, d\sigma_c(t,\theta)) \to L^2(\Gamma_c, \, d\sigma_c(t,\theta)).$$ 
\end{lem}
\begin{proof}
Let $U \colon L^2(\Gamma_c, \, d\sigma_c(t,\theta)) \to L^2(\Gamma, d\sigma(t, \theta))$ be the unitary map 
$$U \colon f(t, \theta) \mapsto \sqrt{\frac{\gamma_{c,1}(t)|\gamma_c'(t)|}{\gamma_{1}(t)|\gamma'(t)|}} f(t, \theta).$$ 
Let $I \colon  L^2(\Gamma, \sin(\alpha) t \, dt \, d\theta) \to L^2(\Gamma, d\sigma(t, \theta))$ be the inclusion map, $If = f$. Then $U = I + M_{b_1}$ and $U^{-1} = I^{-1} + M_{b_2}$, for functions $b_1$ and $b_2$ which satisfy $b_j(t) = O(t)$, $j=1,2$. Hence
$$ U K^\Gamma U^{-1} = K^\Gamma + K^\Gamma M_{b_2} + M_{b_1} K^\Gamma + M_{b_1} K^\Gamma M_{b_2},$$
which by Lemma~\ref{lem:multcomp} implies the desired property of $K^\Gamma$. 
\end{proof}

We can now prove our main perturbation result on $L^2(\Gamma)$.
\begin{thm} \label{thm:perturbcurve}
$K^\Gamma \colon L^2(\Gamma) \to L^2(\Gamma)$ and $K^{\Gamma_c} \colon L^2(\Gamma_c) \to L^2(\Gamma_c)$ are equivalent in the sense of Definition~\ref{def:fredholm}. That is,
$$K^\Gamma \simeq K^{\Gamma_c}.$$
\end{thm}
\begin{proof}
Let $\rho_{\varepsilon}$ be as in Lemma~\ref{lem:multcomp} and let $\rho = \rho_{1/4}$. By Lemma~\ref{lem:renorm} we are justified to consider the difference $K^{\Gamma} - K^{\Gamma_c}$ as an operator on $L^2(\Gamma_c)$, and it sufficient to prove that it is compact. The differences $K^\Gamma - M_\rho K^\Gamma M_\rho$ and $K^{\Gamma_c} -  M_\rho K^{\Gamma_{c}} M_\rho$ have weakly singular kernels and therefore define compact operators by Lemma~\ref{lem:cpctweaksing}. It follows that it is sufficient to prove that $T = M_\rho(K^\Gamma - K^{\Gamma_c})M_\rho$ is compact on $L^2(\widetilde{\Gamma}_c, d\sigma_c)$, where $\widetilde{\Gamma}_c$ is the conical surface
$$\widetilde{\Gamma}_c = \{(\sin (\alpha) t \cos \theta, \sin (\alpha) t \sin \theta, \cos (\alpha) t) \, : \, (t, \theta) \in [0, 1/2] \times [0, 2\pi]\}$$
with surface measure $d\sigma_c(t, \theta) = \sin (\alpha) t\, dt \, d\theta$.

Since $T - M_{\rho_\varepsilon} T M_{\rho_\varepsilon}$ is compact, again by Lemma~\ref{lem:cpctweaksing}, it is sufficient to show that $T_\varepsilon = M_{\rho_\varepsilon} (K^\Gamma - K^{\Gamma_c}) M_{\rho_\varepsilon}$ converges to zero in operator norm as $\varepsilon \to 0$. Note that for $t \in [0, 1/2]$
\begin{multline*}
M_{\rho_\varepsilon} K^\Gamma M_{\rho_\varepsilon} f(t, \theta) = \frac{\rho_\varepsilon(t)}{2\pi} \int_0^{2\pi} \int_0^{2\varepsilon} K^\Gamma(t, \theta, t', \theta') \rho_\varepsilon(t')f(t', \theta') \, d\sigma_c(t', \theta' ) \\ + \frac{\rho_\varepsilon(t)}{2\pi} \int_0^{2\pi} \int_0^{2\varepsilon} K^\Gamma(t, \theta, t', \theta') \rho_\varepsilon(t')f(t', \theta') b(t') \, dt' \, d\theta',
\end{multline*}
where $b(t') = O(t')$ as $t' \to 0$. The second integral in this expression has operator norm tending to zero as $\varepsilon \to 0$ (cf. Lemma~\ref{lem:multcomp}), and hence we do not have to consider it.

We will work with $\widetilde{\Gamma}_c$ in its Lipschitz parametrization
$$\widetilde{\Gamma}_c = \{(x, y, \cot(\alpha)\sqrt{x^2+y^2}) \, : \, x^2 + y^2 \leq 1/4 \},$$
for which the surface measure is given by $d\sigma_c(x,y) = (1/\sin \alpha) \, dx \, dy$ and the unit outward normal is
$$\bm{\nu}^c_{x,y} = \sin \alpha \left(\cot \alpha \frac{x}{\sqrt{x^2+y^2}}, \cot \alpha \frac{y}{\sqrt{x^2+y^2}}, -1\right).$$
Close to the origin, $\Gamma$ has by assumption the parametrization 
$$\Gamma \cap \{\bm{r} \, : \, |\bm{r}| < \delta\} = \{(x, y, \cot(\alpha)\varphi(\sqrt{x^2+y^2})) \, : \, x^2 + y^2 < \delta^2 \},$$
for a sufficiently small $\delta > 0$ and a function $\varphi \in C^5([0, \delta])$ such that $\varphi(t) = t + O(t^2)$. For $\varepsilon \leq \delta/\sin(\alpha)$, we are now left to consider the integral kernel 
\begin{equation} \label{eq:Tintker}
\rho_\varepsilon(\csc(\alpha) \sqrt{x^2 + y^2})\widetilde{T}(x,y, x', y')\rho_\varepsilon(\csc(\alpha) \sqrt{x'^2 + y'^2}),
\end{equation}
where for $x^2+y^2 < \delta^2$ and $x'^2 + y'^2 < \delta^2$ it holds that
\begin{align*}
\widetilde{T}(x,y,x',y') &= \frac{\langle (x-x',y-y', \cot(\alpha) (\sqrt{x^2+y^2} - \sqrt{x'^2+y'^2})), \bm{\nu}_{x,y}^c \rangle}{|(x-x',y-y', \cot(\alpha) (\sqrt{x^2+y^2} - \sqrt{x'^2+y'^2}))|^3} 
\\&\qquad- \frac{\langle (x-x',y-y', \cot(\alpha) (\varphi(\sqrt{x^2+y^2}) - \varphi(\sqrt{x'^2+y'^2}))), \bm{\nu}_{x,y} \rangle}{|(x-x',y-y', \cot(\alpha) (\varphi(\sqrt{x^2+y^2})) - \varphi(\sqrt{x'^2+y'^2})))|^3} 
\\ &= 
\underbrace{\frac{\varphi(\sqrt{x^2+y^2}) - \varphi(\sqrt{x'^2+y'^2}) - \sqrt{x^2+y^2} + \sqrt{x'^2+y'^2}}{\tan \alpha |(x-x',y-y', \cot(\alpha) (\sqrt{x^2+y^2} - \sqrt{x'^2+y'^2}))|^3}}_{(I)} 
\\&\qquad + \underbrace{\frac{\langle (x-x',y-y', \cot(\alpha) (\varphi(\sqrt{x^2+y^2}) - \varphi(\sqrt{x'^2+y'^2}))), \bm{\nu}_{x,y}^c - \bm{\nu}_{x,y} \rangle}{|(x-x',y-y', \cot(\alpha) (\sqrt{x^2+y^2} - \sqrt{x'^2+y'^2}))|^3}}_{(II)}
\\&\qquad + \underbrace{\langle (x-x',y-y', \cot(\alpha) (\varphi(\sqrt{x^2+y^2}) - \varphi(\sqrt{x'^2+y'^2}))), \bm{\nu}_{x,y} \rangle D(x,y, x', y')}_{(III)},
\end{align*}
where
\begin{multline*}
D(x,y, x', y') = |(x-x',y-y', \cot(\alpha) (\sqrt{x^2+y^2} - \sqrt{x'^2+y'^2}))|^{-3} \\- |(x-x',y-y', \cot(\alpha) (\varphi(\sqrt{x^2+y^2})) - \varphi(\sqrt{x'^2+y'^2})))|^{-3}.
\end{multline*}

To prove the theorem, we will now show that $\widetilde{T}$ decomposes into a finite sum 
\begin{equation} \label{eq:Tdecomp}
\widetilde{T} = \sum_i \widetilde{T}^i,
\end{equation}
 where each $\widetilde{T}^i$ is of the form $$\widetilde{T}^i(x,y,x',y') = (b_1(x,y) + b_2(x',y'))B(x,y,x',y') R(x,y,x',y').$$
 Here $R$ denotes one of the three Riesz transforms of $\widetilde{\Gamma}_c$, $b_1$ and $b_2$ are bounded functions satisfying $b_1(x,y) = O(\sqrt{x^2+y^2})$ and $b_2(x',y') = O(\sqrt{x'^2+y'^2})$, and $B \in C^1(([-\delta,\delta]\setminus\{0\})^4)$ is Lipschitz. Note that the operator with kernel $B(x,y,x',y') R(x,y,x',y')$ is bounded; in the decomposition
 $$B(x,y,x',y') R(x,y,x',y') = (B(x,y,x',y') - B(x',y',x',y'))R(x,y,x',y') + B(x',y',x',y')R(x,y,x',y')$$
 the first term is weakly singular and therefore defines a compact operator by Lemma~\ref{lem:cpctweaksing}, while the second term defines a bounded integral operator by an application of Lemma~\ref{lem:riesz}. Since $b_1$ and $b_2$ both tend to zero at the origin, the argument of Lemma~\ref{lem:multcomp} hence shows that the integral operator with the kernel \eqref{eq:Tintker} tends to zero in operator norm as $\varepsilon \to 0$. Thus, the proof is finished if we show that there is a decomposition of the type \eqref{eq:Tdecomp}.

In fact, each of the terms $(I)$, $(II)$, and $(III)$, decomposes in the described way. Let us first treat the term $(I)$. Let $\psi(t) = t^{-2}(\varphi(t) - t) \in C^3([0,\delta])$. Then
\begin{multline*}
(I) = \psi(\sqrt{x^2+y^2})(x + x')\frac{x-x'}{\tan \alpha |(x-x',y-y', \cot(\alpha) (\sqrt{x^2+y^2} - \sqrt{x'^2+y'^2}))|^3} 
\\ + \psi(\sqrt{x^2+y^2})(y + y')\frac{y-y'}{\tan \alpha |(x-x',y-y', \cot(\alpha) (\sqrt{x^2+y^2} - \sqrt{x'^2+y'^2}))|^3} 
\\+ \frac{\psi(\sqrt{x^2+y^2}) - \psi(\sqrt{x'^2+y'^2})}{\sqrt{x^2+y^2} - \sqrt{x'^2+y'^2}}(x'^2+y'^2)\frac{\sqrt{x^2+y^2} - \sqrt{x'^2+y'^2}}{\tan \alpha |(x-x',y-y', \cot(\alpha) (\sqrt{x^2+y^2} - \sqrt{x'^2+y'^2}))|^3}.
\end{multline*} 
Each of the terms of this sum are of the described form. For instance, for the final term we have that $b_1 = 0$, $b_2(x',y') = x'^2+y'^2$, $$B(x,y,x',y') = \frac{\psi(\sqrt{x^2+y^2}) - \psi(\sqrt{x'^2+y'^2})}{\sqrt{x^2+y^2} - \sqrt{x'^2+y'^2}},$$
and $R(x,y,x',y') = R^{\widetilde{\Gamma}_c}_3(x,y,x',y')$ is the kernel of the third Riesz transform of $\widetilde{\Gamma}_c$. 

The term $(II)$ is treated very similarly, after using that
$$\bm{\nu}_{x,y} = \frac{\cot \alpha}{\left(1 + \cot^2(\alpha)\left[\varphi'(\sqrt{x^2+y^2})\right]^2\right)^{1/2}}\left( \frac{x}{\sqrt{x^2+y^2}} \varphi(\sqrt{x^2+y^2}),  \frac{y}{\sqrt{x^2+y^2}}\varphi(\sqrt{x^2+y^2}), -\frac{1}{\cot \alpha}\right).$$

To deal with term $(III)$, let 
$$d_1(x,y,x',y') = |(x-x',y-y', \cot(\alpha) (\sqrt{x^2+y^2} - \sqrt{x'^2+y'^2}))|,$$
and let
$$d_2(x,y,x',y') =  |(x-x',y-y', \cot(\alpha) (\varphi(\sqrt{x^2+y^2})) - \varphi(\sqrt{x'^2+y'^2})))|.$$
Then
$$D = \frac{(d_1^2 + d_1d_2 + d_2^2)(d_1 - d_2)}{d_1^3d_2^3}.$$
The quotients $d_1/d_2$ and $d_2/d_1$ are continuously differentiable on $([-\delta,\delta]\setminus\{0\})^4$ and Lipschitz, so it is sufficient to consider the kernel
$$\frac{\langle (x-x',y-y', \cot(\alpha) (\varphi(\sqrt{x^2+y^2}) - \varphi(\sqrt{x'^2+y'^2}))), \bm{\nu}_{x,y} \rangle (d_1(x,y,x',y') - d_2(x,y,x',y'))}{d_1(x,y,x',y')^3}.$$
The remaining details are again very similar to those of the term $(I)$.
\end{proof}
Next we will prove the corresponding theorem for $\mathcal{E}$, the energy space associated with $\Gamma$. We denote by $\mathcal{E}_c$ the energy space associated with $\Gamma_c$. Recall that $\mathcal{E}$ may be isomorphically identified with $H^{-1/2}(\Gamma)$. To treat the mapping properties of $K^{\Gamma}$ on $\mathcal{E}$ we will actually consider $(K^\Gamma)^\ast$ on $H^{1/2}(\Gamma)$. Just as for $K^\Gamma$, we may use the parametrization of $\Gamma$ to consider $(K^\Gamma)^\ast$ as an operator on $H^{1/2}(\Gamma_c)$, cf. the remarks surrounding \eqref{eq:KGaction}. We then have the following analogue of Lemma~\ref{lem:renorm}.
\begin{lem} \label{lem:Erenorm}
$(K^\Gamma)^\ast \colon H^{1/2}(\Gamma) \to H^{1/2}(\Gamma)$ is similar to $(K^\Gamma)^\ast \colon H^{1/2}(\Gamma_c) \to H^{1/2}(\Gamma_c)$.
\end{lem}
\begin{proof}
This is immediate from the fact that the inclusion map $I : H^{1/2}(\Gamma) \to H^{1/2}(\Gamma_c)$ is an isomorphism. In other words, the norms of $H^{1/2}(\Gamma)$ and $H^{1/2}(\Gamma_c)$ are comparable, which is clear from the Gagliardo-Slobodeckij norm expression \eqref{eq:gagliardo}.
\end{proof}
\begin{thm} \label{thm:Eperturbcurve}
        $K^\Gamma \colon \mathcal{E} \to \mathcal{E}$ and $K^{\Gamma_c} \colon \mathcal{E}_c \to \mathcal{E}_c$ are equivalent in the sense of Definition~\ref{def:fredholm}. That is,
        $$K^\Gamma \simeq K^{\Gamma_c}.$$
\end{thm}
\begin{proof}
The proof of Theorem~\ref{thm:perturbcurve} also shows that $(K^\Gamma)^\ast - (K^{\Gamma_c})^\ast$ is compact as an operator on $L^2(\Gamma_c)$. We will soon prove that this difference is also compact as an operator on $H^1(\Gamma_c)$. By real interpolation \cite{CEP90} it follows that
$$(K^\Gamma)^\ast - (K^{\Gamma_c})^\ast \colon H^{1/2}(\Gamma_c) \to H^{1/2}(\Gamma_c)$$
is compact. By Lemma~\ref{lem:Erenorm} it follows that $(K^\Gamma)^\ast \colon H^{1/2}(\Gamma) \to H^{1/2}(\Gamma)$ and $(K^{\Gamma_c})^\ast \colon H^{1/2}(\Gamma_c) \to H^{1/2}(\Gamma_c)$ are equivalent in the sense of Definition~\ref{def:fredholm}. This proves the theorem, by duality and the identification of $H^{-1/2}(\Gamma)$ and $\mathcal{E}$.

As we have done previously for $K^\Gamma$ and $(K^\Gamma)^\ast$, we may consider $S^{\Gamma}$ as an operator on $L^2(\Gamma_c)$. It is an obviously bounded operator, using Lemma~\ref{lem:cpctweaksing}. To prove compactness on $H^1(\Gamma_c)$ we may, by Lemma~\ref{lem:Siso}, equivalently prove that 
$$((K^\Gamma)^\ast - (K^{\Gamma_c})^\ast)S^{\Gamma_c} \colon L^2(\Gamma_c) \to H^1(\Gamma_c)$$
is compact. By the Plemelj formula \eqref{eq:plemelj} we have that
$$((K^\Gamma)^\ast - (K^{\Gamma_c})^\ast)S^{\Gamma_c} = S^{\Gamma_c}(K^{\Gamma}-K^{\Gamma_c}) - (S^{\Gamma_c} - S^{\Gamma})K^{\Gamma} + (K^\Gamma)^\ast(S^{\Gamma_c} - S^{\Gamma}).$$
We already know that $S^{\Gamma_c}(K^{\Gamma}-K^{\Gamma_c}) \colon L^2(\Gamma_c) \to H^1(\Gamma_c)$ is compact, by the proof of Theorem~\ref{thm:perturbcurve} and Lemma~\ref{lem:Siso}. It is hence sufficient to prove that $(S^{\Gamma_c} - S^{\Gamma})K^{\Gamma}$ and $(K^\Gamma)^\ast(S^{\Gamma_c} - S^{\Gamma})$ also are compact terms. They are both similar; we will treat the first one.

Let $\rho_\varepsilon$ be as in Lemma~\ref{lem:multcomp}. Since $M_{\rho_{\varepsilon}}K^\Gamma M_{\rho_{\varepsilon}} - K^\Gamma$ is compact on $L^2(\Gamma_c)$ by Lemma~\ref{lem:cpctweaksing}, it is sufficient to prove that
$$(S^{\Gamma_c} - S^{\Gamma})M_{\rho_{\varepsilon}}K^\Gamma M_{\rho_{\varepsilon}} \colon L^2(\Gamma_c) \to H^1(\Gamma_c)$$
is compact. To accomplish this, we show that $(S^{\Gamma_c} - S^{\Gamma})M_{\rho_{\varepsilon}}$ is compact. For $\bm{r} \in \Gamma_c$, let $\bm{\tau} = \bm{\tau}(\bm{r})$ be a tangent vector on $\Gamma_c$, smooth in $\bm{r}$, except close to the origin where we still assume that $\bm{\tau}$ is bounded above and below in length. We need to show that
$\partial_{\bm{\tau}} (S^{\Gamma_c} - S^{\Gamma})M_{\rho_{\varepsilon}} \colon L^2(\Gamma_c) \to L^2(\Gamma_c)$ is compact, for any choice of $\bm{\tau}$. Clearly $M_{1 - \rho_{3\varepsilon}}\partial_{\bm{\tau}} (S^{\Gamma_c} - S^{\Gamma})M_{\rho_{\varepsilon}}$ is compact on $L^2(\Gamma_c)$ by Lemma~\ref{lem:cpctweaksing}, since its integral kernel is smooth, so it is sufficient to show that the norm of
$$T = M_{\rho_{3\varepsilon}}\partial_{\bm{\tau}} (S^{\Gamma_c} - S^{\Gamma})M_{\rho_{\varepsilon}} \colon L^2(\Gamma_c) \to L^2(\Gamma_c)$$
tends to zero as $\varepsilon \to 0^+$.

For this purpose we may consider $T$ as on operator on $L^2(\widetilde{\Gamma}_c)$, where $\widetilde{\Gamma}_c$ is the conical surface of the proof of Theorem~\ref{thm:perturbcurve}. We consider the same parametrization as in that proof, and we consider the tangential derivative $\partial_{\bm{\tau}} = -\partial_x$. A computation, the interchange of limit and differentiation justified by Lemma~\ref{lem:riesz}, shows that
\begin{multline*}
Tf(x,y,x',y') = \\
\lim_{\eta \to 0^+} \rho_{3\varepsilon}(\csc(\alpha) \sqrt{x^2+y^2})\Bigg(\int_{|(x',y') - (x,y)| > \eta}  K_1(x,y,x',y') \rho_{\varepsilon}(\csc(\alpha) \sqrt{x'^2+y'^2}) f(x',y') \csc(\alpha) \, dx \, dy \\
- \int_{|(x',y') - (x,y)| > \eta}  K_2(x,y,x',y') \rho_{\varepsilon}(\csc(\alpha) \sqrt{x'^2+y'^2}) f(x',y') \left(1 + \cot^2(\alpha)\left[\varphi'(\sqrt{x^2+y^2})\right]^2\right)^{1/2} \, dx \, dy \Bigg), 
\end{multline*}
where 
$$K_1(x,y,x',y') = \frac{ (x-x') + \frac{\cot^2(\alpha)x}{\sqrt{x^2+y^2}} \left(\sqrt{x^2+y^2} - \sqrt{x'^2+y'^2}\right)}{|(x-x',y-y', \cot(\alpha) (\sqrt{x^2+y^2} - \sqrt{x'^2+y'^2}))|^3},$$
and 
$$K_2(x,y,x',y') = \frac{(x-x') + \frac{\cot^2(\alpha)x}{\sqrt{x^2+y^2}}\varphi'(\sqrt{x^2+y^2}) \left(\varphi(\sqrt{x^2+y^2}) - \varphi(\sqrt{x'^2+y'^2})\right)}{|(x-x',y-y', \cot(\alpha) (\varphi(\sqrt{x^2+y^2})) - \varphi(\sqrt{x'^2+y'^2})))|^{3}}.$$
Analogous arguments to those of the proof of Theorem~\ref{thm:perturbcurve} will yield that the integral kernel of $T$ may be decomposed in the same way as in that proof. That is, we may write the integral kernel as a sum of Riesz kernels multiplied by functions which are either sufficiently smooth or small at the origin. Once this is done, we may take $\eta \to 0^+$ using the maximal Riesz transform estimate of Lemma~\ref{lem:riesz}, concluding that $T$ is arbitrarily small in norm as $\varepsilon \to 0^+$. The same can then be done for the tangential derivative $\partial_{\bm{\tau}} = -\partial_y$, finishing the proof. As the details are quite clear after having seen Theorem~\ref{thm:perturbcurve}, but equally lengthy, we choose to not give any further computations.
\end{proof}
\section{The essential spectrum and Fredholm index on $L^2$} \label{sec:L2spec}
Throughout this section we consider $K^{\Gamma} \colon L^2(\Gamma) \to L^2(\Gamma)$ as an operator on $L^2(\Gamma)$ only. The kernel $K^\Gamma$ satisfies the same type of rotational invariance that we considered earlier,
$$K^\Gamma(t, \theta, t', \theta') = K^\Gamma(t, \theta - \theta', t', 0).$$
Hence we let $K_n^\gamma : L^2(\gamma_1(t)|\gamma'(t)| \, dt) \to L^2(\gamma_1(t)|\gamma'(t)| \, dt)$ be the integral operator with kernel
$$K^\gamma_n(t, t') = \frac{1}{2\pi}\int_0^{2\pi} e^{-in\theta} K^\Gamma(t, \theta, t', 0) \, d\theta, \qquad 0 < t, t' \leq 1,$$
so that
\begin{equation} \label{eq:Kgaction}
K_n^\gamma f(t) = \int_0^1 K_n^\gamma(t, t') f(t')  \gamma_1(t')|\gamma'(t')| \, dt'.
\end{equation}
Just as for the infinite cone in Section~\ref{sec:infcone}, we then have that
$$K^\Gamma \simeq_{\ue} \bigoplus_{n=-\infty}^\infty K_n^\gamma.$$

Having established that $K^\Gamma \simeq K^{\Gamma_c}$ in Theorem~\ref{thm:perturbcurve}, we will now investigate the spectrum of
\begin{equation} \label{eq:kgcdecomp}
K^{\Gamma_c} \simeq_{\ue} \bigoplus_{n=-\infty}^\infty K^{\gamma_c}_n
\end{equation}
acting on 
$$L^2(\Gamma_c, \, d\sigma_c) = \bigoplus_{n=-\infty}^\infty L^2([0,1], \gamma_1(t) |\gamma'(t)| \, dt).$$
Let $\rho \in C^\infty_c([0, 1/2))$ be a non-negative cut off function such that $\rho(t) = 1$ for $0 \leq t \leq 1/4$, and $0 \leq \rho(t) \leq 1$ for all $t$. Then $K^{\Gamma_c} - M_\rho K^{\Gamma_c} M_\rho$ is compact, and 
\begin{equation}
\label{eq:kgcdecomp2}
M_\rho K^{\Gamma_c} M_\rho \simeq_{\ue} \bigoplus_{n=-\infty}^\infty M_\rho K^{\gamma_c}_n M_\rho.
\end{equation}
Equations \eqref{eq:kgcdecomp} and \eqref{eq:kgcdecomp2} imply that $K^{\gamma_c} - M_\rho K^{\gamma_c}_n M_\rho$ is compact, for every $n$. Since $\gamma_c$ is a straight line segment for $0 \leq t \leq 1/2$ we consider $M_\rho K^{\Gamma_c} M_\rho$ and $M_\rho K^{\gamma_c}_n M_\rho$ to be operators 
$$M_\rho K^{\Gamma_c} M_\rho \colon L^2([0,1/2] \times [0, 2\pi], \sin(\alpha) t \, dt \, d\theta) \to  L^2([0,1/2] \times [0, 2\pi], \sin(\alpha) t \, dt \, d\theta)$$ 
and 
$$M_\rho K^{\gamma_c}_n M_\rho \colon L^2([0,1/2], \sin(\alpha) t \, dt) \to L^2([0,1/2], \sin(\alpha) t \, dt).$$
This preserves equivalence in the sense of Definition~\ref{def:fredholm}.

Let $U : L^2([0,1/2],  \sin(\alpha) t \, dt) \to L^2([0,1/2], dt)$ be the unitary operator defined by 
$$Uf(t) = \sqrt{t \sin \alpha}f(t), \quad 0 \leq t \leq 1/2.$$
 Observe that $K^{\gamma_c}_n(t, t') = K^\alpha_n(t,t')$ for $0 < t, t' \leq 1/2$, where $K^\alpha_n$ is the kernel associated with the infinite straight cone, defined in Section~\ref{sec:infcone}. By the homogeneity property \eqref{eq:homog}, we have for $g \in L^2([0,1/2], dt)$ that
$$UM_\rho K^{\gamma_c}_n M_\rho U^{-1} g(t) = \sin(\alpha) \rho(t)\int_0^{1/2} \sqrt{t/t'} K_n^\alpha(t/t', 1) \rho(t') g(t') \, \frac{dt'}{t'}.$$
The operator $J_n^\alpha = UM_\rho K^{\gamma_c}_n M_\rho U^{-1}$ is hence a Hardy kernel operator on $L^2([0,1/2], \, dt)$ with kernel 
$$j_n^\alpha(t) = \sin(\alpha) \sqrt{t} K_n^\alpha(t, 1),$$
  multiplied from the left and the right by the cut-off function $\rho$.  Hence, once we show that the Mellin transform of the kernel $j_n^\alpha$ is sufficiently well-behaved, it follows that the operator $J_n^\alpha$ is a pseudodifferential operator of Mellin type. 
There is a well-established symbolic calculus for such Mellin operators, developed in \cite{Els87,Lew91,LP83}, which we will use to compute the spectrum of $K^{\gamma_c}_n$.

 For $-\infty < \alpha < \beta < \infty$ and a number $m$ we follow the notation of \cite{Els87,Lew91} and denote by $\mathcal{O}^{m}_{\alpha,\beta}$ the space of functions $g(\zeta)$, holomorphic in the strip $\{\zeta = \eta + i\xi \, : \, \alpha < \eta < \beta, \; \xi \in \mathbb{R}\}$ and such that for every $ \alpha < c < d < \beta$ and non-negative integer $l$ it holds that
$$\sup_{\zeta = \eta + i\xi \atop c < \eta < d, \; \xi \in \mathbb{R}} \left| (1+|\xi|)^{l-m}\left(\frac{d}{d\zeta}\right)^l\mathcal{M}g(\zeta) \right| < \infty.$$
\begin{lem} \label{lem:mellinsmooth}
The Mellin transform of $j^\alpha_0$ belongs to $\mathcal{O}^{-1}_{-1/2,5/2}$, $\mathcal{M}j^\alpha_0 \in \mathcal{O}^{-1}_{-1/2,5/2}$. For $n\neq0$, it holds that $\mathcal{M}j^\alpha_n\in \mathcal{O}^{-1}_{-|n|+1/2,|n|+3/2}$.
\end{lem}
\begin{proof}
First note that the function $w(t) = \log\left|\frac{1+t}{1-t}\right|$ has the explicit Mellin transform
$$\mathcal{M}w(\zeta) = \frac{\pi}{\zeta}\tan\left(\frac{\pi \zeta}{2}\right), \qquad -1 < \Re\mathrm{e} \, \zeta < 1.$$
By this formula, $\mathcal{M}w \in \mathcal{O}^{-1}_{0,1}$. For fixed $c$ and $d$, $0 < c < d < 1$, we let 
$$C_l = \sup_{\zeta = \eta + i\xi \atop c < \eta < d, \; \xi \in \mathbb{R}} \left| (1+|\xi|)^{l+1}\left(\frac{d}{d\zeta}\right)^l\mathcal{M}w(\zeta) \right|.$$
Let $\tau$ be any smooth compactly supported function in $(0, \infty)$.
Let $T(\xi)$ be the Fourier transform of $\tau(e^x)$,
$$T(\xi) = \int_{-\infty}^{\infty} e^{i x \xi} \tau(e^x) \, dx.$$
Note that for every $l \geq 0$ there is a constant $D_l < \infty$ such that $|T(\xi)| \leq D_l (1+|\xi|)^{-l-1}$, since $\tau(e^x)$ is a smooth compactly supported function on $\mathbb{R}$. For $\zeta = \eta + i \xi$, $c < \eta < d$, we have by the change of variable $r = e^x$ and the usual Fourier convolution theorem that
$$\mathcal{M}(w\tau)(\zeta) = \int_{-\infty}^{\infty} e^{ix\xi} e^{x\eta}w(e^x) \tau(e^x)\, dx = \frac{1}{2\pi}\int_{-\infty}^{\infty} \mathcal{M}w(\eta + i (\xi - x)) T(x) \, dx.$$
Using the definitions of $C_l$ and $D_l$ and the fact that $T$ and $\left(\frac{d}{d\zeta}\right)^l \mathcal{M}w(\eta + i\cdot) $ belong to $L^1(\mathbb{R})$, we find for every $l \geq 1$ that
\begin{align*}
\left|\left(\frac{d}{d\zeta}\right)^l \mathcal{M}(w\tau)(\zeta)\right| &= \left|\int_{-\infty}^{\infty} \left(\left(\frac{d}{d\zeta}\right)^l \mathcal{M}w\right)(\eta + i (\xi-x)) T(x) \, dx.\right| \\ &\leq \left(\int_{|\xi - x| \geq \frac{|\xi|}{2}} + \int_{|x| \geq \frac{|\xi|}{2}}\right) \left|\left(\left(\frac{d}{d\zeta}\right)^l \mathcal{M}w\right)(\eta + i (\xi-x)) T(x)\right| \, dx \\  &\lesssim 2^l(C_l + D_l)(1+|\xi|)^{-l-1}.
\end{align*}
The same inequality follows also for $l = 0$, by integrating the obtained inequality for $l = 1$. It immediately follows that $\mathcal{M}(w\tau) \in \mathcal{O}^{-1}_{0,1}$. However, since $\tau(r) r^\eta$ is also compactly supported smooth function on $(0,1)$, for any $\eta \in \mathbb{R}$, we actually conclude that $\mathcal{M}(w\tau) \in \mathcal{O}^{-1}_{\alpha,\beta}$ for any $-\infty < \alpha < \beta < \infty$.

Now fix $n \neq 0$ and a smooth function $\tau_1$, compactly supported in $(1/2, 3/2)$ and such that $\tau_1(t) = 1$ for $t \in [-3/4, 5/4]$. By Lemma~\ref{lem:specestimate} there is a smooth function $G(t)$ on $(1/2, 3/2)$ such that 
$$d(t) := j^\alpha_n(t) - w(t) G(t) \tau_1(t)$$ is a function in $C^\infty((0,\infty))$.  It is also clear from the easy part of the proof, which appears in the Appendix, that the statement of Lemma~\ref{lem:specestimate} could have been sharpened somewhat: for every $l \geq 0$ it holds that 
$$ \left(t \frac{d}{dt}\right)^l d(t) = O(t^{|n|+3/2}), \; t \to \infty, \quad  \left(t \frac{d}{dt}\right)^l d(t) = O(t^{|n|-1/2}), \; t \to 0.$$
By \cite[Lemma 1.2]{Els87} it follows that $\mathcal{M}d \in \mathcal{O}^{m}_{-|n|+1/2,|n|+3/2}$ for every integer $m$. By combining this with the conclusion of the previous paragraph, letting $\tau = G \tau_1$, we see that $\mathcal{M}j_n^\alpha \in \mathcal{O}^{-1}_{-|n|+1/2,|n|+3/2}$. Similarly, taking into account the different asymptotics for $n=0$, $\mathcal{M}j_0^\alpha \in \mathcal{O}^{-1}_{-1/2,5/2}$. 
\end{proof}
\begin{rmk}
The proof shows that for no $\alpha < \beta$ and $\varepsilon > 0$ does it hold that $\mathcal{M} j_n^\alpha \in \mathcal{O}^{-1-\varepsilon}_{\alpha, \beta}$. Hence $j_n^\alpha$ only barely satisfies the smoothness hypothesis required in \cite{Els87,Lew91}. This is unlike the applications to layer potentials for domains with corners in 2D \cite{Lew91,Mit02}, where the Mellin transforms of the corresponding Hardy kernels belong to $\mathcal{O}^{m}_{0,1}$ for every $m$.
\end{rmk}
Having verified that $\mathcal{M} j_n^\alpha \in \mathcal{O}^{-1}_{0,1}$ for every $n$, we may now apply the symbolic calculus. We refer to \cite[pp. 472--473]{Mit02} for a summary of the theory. We will not recount any of the details here, but only mention that each Mellin pseudodifferential operator $A$ is associated with its symbol $\smbl^{1/2} A$, which can be interpreted as a continuous function on the boundary of the compact rectangle $\mathcal{R}^{1/2}$ (see Lemma~\ref{lem:mellincalculus})
The operator $A : L^2([0,1/2], \, dt) \to L^2([0,1/2], \, dt)$ is Fredholm if and only if $\smbl^{1/2} A$ has no zero on $\partial \mathcal{R}^{1/2}$. In this case, the Fredholm index of $A$ is given by the change in argument of $\smbl^{1/2} A$ as $\partial \mathcal{R}^{1/2}$ is traversed clockwise. As $J_n^\alpha$ is a Hardy kernel operator, multiplied from the left and the right by a smooth cut off function, the symbolic calculus yields the following.
\begin{lem} \label{lem:mellincalculus}
For every $z \in \mathbb{C}$ and $n \in \mathbb{Z}$, the operator $J_n^\alpha - z: L^2([0,1/2], \, dt) \to  L^2([0,1/2], \, dt)$ is a Mellin pseudodifferential operator with symbol $\smbl^{1/2} \left(J_n^\alpha - z \right)$ given by
\[
\begin{tikzpicture}
\draw (0,0) node[below = 0.1cm] {$t=0$}
-- (5, 0) node (larrow) {\leftiearrow} node[below = 0.1cm of larrow] {$-z$}
-- (10,0) node[below = 0.1cm] {$t=1/2$} node[right = 0.1cm] {$1/2+i\infty$}
-- (10,2) node (darrow) {\downnarrow}  node[right = 0.1cm of darrow] {$-z$}
-- (10,4) node[right= 0.1cm] {$1/2 - i\infty$} node[above= 0.1cm] {$t=1/2$}
-- (5,4) node (rarrow) {\rightiearrow} node[above = 0.1cm of rarrow] {$-z$}
-- (0,4) node[above= 0.1cm] {$t=0$} node[left= 0.1cm] {$1/2+i\infty$}
-- (0,2) node (uarrow) {\upparrow} node[left = 0.1cm of uarrow] {$\mathcal{M}j_n^{\alpha}(1/2 + i\xi) - z$}
-- cycle node[left= 0.1cm] {$1/2-i\infty$};
\draw node at (5,2) {$\mathcal{R}^{1/2}$};
\end{tikzpicture}
\]
Hence, the essential spectrum of $J_n^\alpha$ is given by the closed curve
$$\sigma_{\ess}(J_n^\alpha) = \{\mathcal{M}j_n^{\alpha}(1/2 + i\xi) \, : \, -\infty \leq \xi \leq \infty\},$$
where it is understood that $\mathcal{M}j_n^{\alpha}(1/2 \pm i\infty) = 0$. For $z \notin \sigma_{\ess}(J_n^\alpha)$, $J_n^\alpha - z$ is Fredholm with index given by the winding number of $z$ with respect to the curve $\sigma_{\ess}(J_n^\alpha)$,
$$\ind(J_n^\alpha - z) = W(z, \sigma_{\ess}(J_n^\alpha)).$$
\end{lem}
Since $\mathcal{M}j_n^{\alpha}(1/2 + i\xi) = \sin(\alpha)\mathcal{M}[K_n^\alpha(\cdot, 1)](1+i\xi)$, the curve in Lemma~\ref{lem:mellincalculus} is the same curve that appears in Theorem~\ref{thm:infconespec},
$$\sigma_{\ess}(J_n^\alpha) = \{\mathcal{M}j_n^{\alpha}(1/2 + i\xi) \, : \, -\infty \leq \xi \leq \infty\} = \Pi_n.$$
We need one more lemma.
\begin{lem} \label{lem:specestimate2}
It holds that
$$\|J_n^\alpha\|_{L^2([0,1/2], dt) \to L^2([0,1/2], dt)} \lesssim \frac{1}{|n|+1}.$$
\end{lem}
\begin{proof}
Let $K_n^\alpha \colon L^2([0, \infty), \sin(\alpha) \, t dt) \to L^2([0, \infty), \sin(\alpha) \, t dt)$ be the operator defined in Section~\ref{sec:infcone} and recall that we showed that $\|K_n\| \lesssim (1+|n|)^{-1}$ in the proof of Theorem~\ref{thm:infconespec}. By definition $J_n^\alpha \simeq_{\ue} M_\rho K^{\gamma_c}_n M_\rho.$ Then clearly
$$\|J_n^\alpha\| = \|M_\rho K^{\gamma_c}_n M_\rho\| \leq \|K_n^{\gamma_\alpha}\|\lesssim \frac{1}{|n|+1},$$
which is what we wanted to prove.
\end{proof}
We can now piece together all of our results to obtain the main theorem of this section.
\begin{thm} \label{thm:winding}
Let $\Gamma$ be a closed surface of revolution with a conical point of opening angle $2\alpha$, obtained by revolving a $C^5$-curve $\gamma$. For $n \in \mathbb{Z}$, denote by $\Pi_n$ the closed curve
\begin{align*}
\Pi_n &= \{\sin(\alpha)\mathcal{M}[K_n^\alpha(\cdot, 1)](1+i\xi) \, : \, -\infty \leq \xi \leq \infty\} \\ &=\left\{ \frac{P_{i\xi}^n(\cos \alpha)\dot{P}_{i\xi}^n(-\cos \alpha)-P_{i\xi}^n(-\cos \alpha)\dot{P}_{i\xi}^n(\cos \alpha)}{P_{i\xi}^n(-\cos \alpha)\dot{P}_{i\xi}^n(\cos \alpha)+P_{i\xi}^n(\cos \alpha)\dot{P}_{i\xi}^n(-\cos \alpha)} \, : \, -\infty \leq \xi \leq \infty \right\},
\end{align*}
with orientation given by the $\xi$-variable.
Then the operator $K^\Gamma \colon L^2(\Gamma, \, d\sigma) \to L^2(\Gamma, \, d\sigma)$ has essential spectrum
\begin{equation} \label{eq:esspecformula}
\sigma_{\ess}(K^\Gamma, L^2) = \bigcup_{n=-\infty}^\infty \Pi_n.
\end{equation}
If $z \notin \sigma_{\ess}(K^\Gamma, L^2)$, then $K^\Gamma-z$ has Fredholm index
$$\ind(K^\Gamma-z) = \sum_{n=-\infty}^\infty W(z, \Pi_n) = W(z, \Pi_0) + 2\sum_{n=1}^\infty W(z, \Pi_n)$$
where $W(z, \Pi_n) \leq 0$ denotes the winding number of $z$ with respect to $\Pi_n$ and the right-hand side is always a finite sum. In particular, every point $z$ lying inside one of the curves $\Pi_n$ belongs to the spectrum $\sigma(K^\Gamma, L^2)$.

Whenever $z$ is not a real number, it holds that $\dim \ker (K^\Gamma - z) = 0$, so that
$$\ind(K^\Gamma - z) = -\codim \ran K^\Gamma, \quad z \in \mathbb{C}\setminus\mathbb{R}.$$
 In particular,  if $\ind(K^\Gamma-z) = 0$ (so that $z$ lies outside every curve $\Pi_n$), then either $K^\Gamma-z$ is invertible or $z = x$ is real and an eigenvalue of $K^\Gamma$.
\end{thm}
\begin{rmk}
Figure~\ref{fig:leaves} illustrates the set $\bigcup_{n=-\infty}^\infty \Pi_n$. Note that for each $z$ it holds that $W(z, \Pi_0)$ is $0$ or $-1$ and for $n\neq0$ that $W(z, \Pi_n)$ is $0$, $-1$, or $-2$. Strictly speaking we never prove this statement.
\end{rmk}
\begin{rmk}
We could define a notion of more general surfaces $\Gamma$ with a finite number of axially symmetric conical points. Theorem~\ref{thm:winding} extends effortlessly to such surfaces, each conical point contributing a set of the type \eqref{eq:esspecformula} to the essential spectrum. The index formula then extends additively with the corners. We refrain from giving an explicit statement, to avoid introducing further notation.
\end{rmk}
\begin{proof}
 By Theorem~\ref{thm:perturbcurve}, 
$$K^\Gamma \simeq \bigoplus_{n=-\infty}^\infty J_n^\alpha.$$
By Lemma~\ref{lem:specestimate2}, there is a finite number $m$, depending only on $z$, such that  $J_n^\alpha - z$ is invertible for $|n| > m$. Hence, by Lemma~\ref{lem:mellincalculus}, there are two possibilities. Either $z \in \bigcup_{|n| \leq m} \Pi_n$, and in this case $z \in \sigma_{\ess}(K^\Gamma, L^2)$, or $K^\Gamma - z$ is Fredholm with index 
$$\ind(K^\Gamma-z) = \sum_{|n| \leq m} W(z, \Pi_n) = \sum_{n=-\infty}^\infty W(z, \Pi_n).$$
The explicit formula for $\Pi_n$ was proven in Theorem~\ref{thm:curveformula}. 

Recall next that $K^\Gamma \colon \mathcal{E} \to \mathcal{E}$ is a self-adjoint operator. Hence, unless $z = x$ is real, $K^\Gamma-z \colon \mathcal{E} \to \mathcal{E}$ is always invertible. Since $L^2(\Gamma) \subset \mathcal{E}$ it follows that $z$ cannot be an eigenvalue for $K^\Gamma : L^2(\Gamma) \to L^2(\Gamma)$ unless $z$ is real. In particular, $K^\Gamma-z$ is invertible if $z$ is non-real and $\ind(K^\Gamma-z) = 0$. It also follows that $W(z, \Pi_n) \leq 0$ for every $z \notin \Pi_n$ and $n$. If $z = x$ is real and $K-x$ is not invertible, then clearly it must be an eigenvalue if $\ind(K^\Gamma-z) = 0$.
\end{proof}

\section{The essential spectrum on $\mathcal{E}$} \label{sec:Espec}
Recall that we characterized the spectrum of $K^{\Gamma_\alpha} \colon \mathcal{E} \to \mathcal{E}$ in Theorem~\ref{thm:Especres}, when $\Gamma_\alpha$ is an infinite straight cone. To begin this section we will show that the essential spectrum remains the same if we localize $K^{\Gamma_\alpha}$ to the origin. Informally speaking, we will show that the singularities of $K^{\Gamma_\alpha}$ at the origin and at infinity contribute equally to the essential spectrum. 
\begin{lem}
        For $f \in L^2(\Gamma_\alpha)$, let $Vf(t, \theta) = \frac{1}{t^3}f\left(\frac{1}{t}, -\theta \right)$. Then $V$ extends to a unitary involution 
        $$V \colon \mathcal{E} \to \mathcal{E}.$$
        Furthermore, $V$ commutes with $K^{\Gamma_\alpha}$,
        $$VK^{\Gamma_\alpha} = K^{\Gamma_\alpha}V.$$
\end{lem}
\begin{proof}
        It is clear that $V$ has dense range in $\mathcal{E}$, so we only need to verify that it is an isometry. Note that $Uf(t, \theta) = \frac{1}{t^2}f\left(\frac{1}{t}, -\theta \right)$ is a unitary map of $L^2(\Gamma_\alpha)$ such that $U = U^*$. A computation relying on the homogeneity and symmetry of $S^{\Gamma_\alpha}$ shows that
        \begin{align*}
        US^{\Gamma_\alpha}f(t) &= \frac{1}{t^2} \int_0^{2\pi} \int_0^\infty S^{\Gamma_\alpha}\left(\frac{1}{t}, -\theta, t', \theta' \right) f(t', \theta') t' \, dt'  \sin(\alpha) \, d\theta' \\ 
        &= \frac{1}{t^2} \int_0^{2\pi} \int_0^\infty S^{\Gamma_\alpha} \left(\frac{1}{t}, -\theta, \frac{1}{s'}, -\theta' \right) Vf(s', \theta') \, ds' \sin(\alpha) \, d\theta' \\
        &= \frac{1}{t} \int_0^{2\pi} \int_0^\infty S^{\Gamma_\alpha} \left(t, \theta, s', \theta' \right) Vf(s') s' \, ds' \sin(\alpha) \, d\theta' \\
        &= \frac{1}{t} S^{\Gamma_\alpha} Vf(t).
        \end{align*} 
        Hence
        \begin{equation*}
        \langle f, g \rangle_{\mathcal{E}} = \langle US^{\Gamma_\alpha} f, Ug \rangle_{L^2(\Gamma_\alpha)} = \langle S^{\Gamma_\alpha} V f, Vg \rangle_{L^2(\Gamma_\alpha)} = \langle Vf, Vg \rangle_{\mathcal{E}}. 
        \end{equation*}
        
        For $\bm{r}, \bm{r'} \in \Gamma_\alpha$ we have that
        $$\langle \bm{r'} - \bm{r}, \bm{\nu_{r}} \rangle  = \cos (\alpha) \sin(\alpha) t' (1 - \cos(\theta' - \theta)),$$
        and thus also that
        $$\langle \bm{r} - \bm{r'}, \bm{\nu_{r'}} \rangle  = \cos (\alpha) \sin(\alpha) t (1 - \cos(\theta' - \theta))$$
        It follows that
        $$K^{\Gamma_\alpha}(t, \theta, t', \theta') = \frac{t'}{t}(K^{\Gamma_\alpha})^{\ast}(t, \theta, t', \theta'),$$
        cf. Lemma~\ref{lem:specestimate}.
        Hence a computation similar to the above one yields that
        \begin{align*}
        VK^{\Gamma_\alpha}f(t) &= \frac{1}{t^3} \int_0^{2\pi} \int_0^\infty K^{\Gamma_\alpha} \left(\frac{1}{t}, -\theta, \frac{1}{s'}, -\theta' \right) Vf(s', \theta') \, ds' \sin(\alpha) \, d\theta' \\
        &= \int_0^{2\pi} \int_0^\infty \frac{1}{t} (K^{\Gamma_\alpha})^\ast \left(t, \theta, s', \theta' \right) s' Vf(s') s' \, ds' \sin(\alpha) \, d\theta' \\
        &= K^{\Gamma_\alpha} V f(t). \qedhere
        \end{align*}
\end{proof}
The localization result we are after is the following.
\begin{lem} \label{lem:essspecloc}
        Let $\rho \in C_c^\infty([0,1/2))$ be a smooth compactly supported function such that $\rho(t) = 1$ for $t \in [0,1/4]$, and denote by $M_\rho : \mathcal{E} \to \mathcal{E}$ the operator of multiplication by $\rho$ in $(t,\theta)$-coordinates. Then
        $$\sigma_{\ea}(M_\rho K^{\Gamma_\alpha} M_\rho, \mathcal{E}) = \bigcup_{n=-\infty}^\infty \Sigma_n,$$
        where $\Sigma_n$ are the same intervals as in Theorem~\ref{thm:Especres}. 
\end{lem}
\begin{proof}
        Let $\rho_1 = \rho$, let $\rho_2(t) = \rho_1(1/t)$, $0 < t < \infty$, and let $\rho_3 = 1 - \rho_1 - \rho_2$. Then 
        $$K^{\Gamma_\alpha} = \sum_{j,k=1}^3 M_{\rho_j} K^{\Gamma_\alpha} M_{\rho_k}.$$
        If $(j,k)$ is any of the indices $(1,3)$, $(3,1)$, or $(3,3)$, then  $M_{\rho_j} K^{\Gamma_\alpha} M_{\rho_k}$ has a weakly singular kernel and it follows from Lemma~\ref{lem:cpctweaksing} and \eqref{eq:energynormcutoff} that it is a compact operator. To treat the remaining non-diagonal terms we use the operator $V$ to move neighborhoods of $\infty$ to the origin, so that we may apply $\eqref{eq:energynormcutoff}$. Note that $VM_{\rho_1} = M_{\rho_2}V$,  $VM_{\rho_3} = M_{\rho_3}V$ and that the adjoint of $V$ with respect to the $L^2(\Gamma_\alpha)$-pairing is given by
        $$V^\ast f(t) = \frac{1}{t}f\left(\frac{1}{t}, -\theta\right).$$ 
        By $\eqref{eq:energynormcutoff}$, $M_{\rho_1} K^{\Gamma_\alpha} M_{\rho_2} \colon \mathcal{E} \to \mathcal{E}$ is compact if and only if 
        $$M_{\rho_1} K^{\Gamma_\alpha} M_{\rho_2} V = M_{\rho_1} K^{\Gamma_\alpha} V M_{\rho_1}  \colon H^{-1/2}(\Gamma_\alpha) \to H^{-1/2}(\Gamma_\alpha)$$
        is compact. It is easy to check that the $L^2(\Gamma_\alpha)$-adjoint of this latter operator defines a bounded operator
        $$M_{\rho_1} V^\ast (K^{\Gamma_\alpha})^\ast  M_{\rho_1} \colon L^2(\Gamma_\alpha) \to H^{1/2}(\Gamma_\alpha)$$
        As in the proof of Lemma~\ref{lem:cpctweaksing}, it follows that $M_{\rho_1} K^{\Gamma_\alpha} M_{\rho_2} \colon \mathcal{E} \to \mathcal{E}$ is compact. Since
        $$M_{\rho_1} K^{\Gamma_\alpha} M_{\rho_2} V = V M_{\rho_2} K^{\Gamma_\alpha} M_{\rho_1}$$
        it also follows that $M_{\rho_2} K^{\Gamma_\alpha} M_{\rho_1} \colon \mathcal{E} \to \mathcal{E}$ is compact. Similarly, 
        $$M_{\rho_2} K^{\Gamma_\alpha} M_{\rho_3} V = VM_{\rho_1} K^{\Gamma_\alpha} M_{\rho_3}$$
        so we see that $M_{\rho_2} K^{\Gamma_\alpha} M_{\rho_3}$ is compact. In the same way we conclude that $M_{\rho_3} K^{\Gamma_\alpha} M_{\rho_2}$ is compact.
        
        In total, we have shown that
        $$K^{\Gamma_\alpha} = M_{\rho_1} K^{\Gamma_\alpha} M_{\rho_1} + M_{\rho_2} K^{\Gamma_\alpha} M_{\rho_2} \textrm{+ compact.}$$
        Note that $M_{\rho_1} K^{\Gamma_\alpha} M_{\rho_1}$ and $M_{\rho_2} K^{\Gamma_\alpha} M_{\rho_2}$ are orthogonal (in the sense that their composition is $0$) and unitarily equivalent,
        $$VM_{\rho_1} K^{\Gamma_\alpha} M_{\rho_1}V^{-1} = M_{\rho_2} K^{\Gamma_\alpha} M_{\rho_2}.$$
        To conclude using these two facts we apply Weyl's criterion to the self-adjoint operator $K^{\Gamma_\alpha} \colon \mathcal{E} \to \mathcal{E}$. Let 
        $$\lambda \in \sigma_{\ea}(M_{\rho_1} K^{\Gamma_\alpha} M_{\rho_1} + M_{\rho_2} K^{\Gamma_\alpha} M_{\rho_2}, \mathcal{E}) = \sigma_{\ess}(K^{\Gamma_\alpha}, \mathcal{E}),$$ and let $(f_n) \subset \mathcal{E}$ be a corresponding singular sequence. Let $\rho_4$ be a smooth function such that $\rho_4(t) = 1$ on the support of $\rho_1$, $\rho_4(t) = 0$ on the support of $\rho_2$. Then either $(M_{\rho_4}f_n)$ is a singular sequence for $M_{\rho_1} K^{\Gamma_\alpha} M_{\rho_1}$ or $(M_{1-\rho_4}f_n)$ is a singular sequence for $M_{\rho_2} K^{\Gamma_\alpha} M_{\rho_2}$. In either case, we find that
        $$\lambda \in \sigma_{\ea}(M_{\rho_1} K^{\Gamma_\alpha} M_{\rho_1}, \mathcal{E}) = \sigma_{\ea}(M_{\rho_2} K^{\Gamma_\alpha} M_{\rho_2}, \mathcal{E}).$$
        Conversely, if $\lambda \in \sigma_{\ea}(M_{\rho_1} K^{\Gamma_\alpha} M_{\rho_1}, \mathcal{E})$ with singular sequence $(f_n)$, then $M_{\rho_1} f_n$ is a singular sequence for the sum $M_{\rho_1} K^{\Gamma_\alpha} M_{\rho_1} + M_{\rho_2} K^{\Gamma_\alpha} M_{\rho_2}$. Hence
        $$\sigma_{\ea}(M_{\rho_1} K^{\Gamma_\alpha} M_{\rho_1}, \mathcal{E}) = \sigma_{\ess}(K^{\Gamma_\alpha}, \mathcal{E}) = \bigcup_{n=-\infty}^\infty \Sigma_n,$$
        by Theorem~\ref{thm:Especres}.
\end{proof}
Combined with Theorem~\ref{thm:Eperturbcurve} this allows us to prove the main theorem of this section.
\begin{thm} \label{thm:Eessspec}
        Let $\Gamma$ be a closed surface of revolution with a conical point of opening angle $2\alpha$, obtained by revolving a $C^5$-curve $\gamma$. For $n \in \mathbb{Z}$, denote by $\Sigma_n$ the closed interval
        \begin{equation*}
        \Sigma_n = \left\{\frac{P_{i\xi - 1/2}^n(\cos \alpha)\dot{P}_{i\xi - 1/2}^n(-\cos \alpha) - P_{i\xi - 1/2}^n(-\cos \alpha)\dot{P}_{i\xi - 1/2}^n(\cos \alpha)}{P_{i\xi - 1/2}^n(-\cos \alpha)\dot{P}_{i\xi - 1/2}^n(\cos \alpha)+P_{i\xi - 1/2}^n(\cos \alpha)\dot{P}_{i\xi - 1/2}^n(-\cos \alpha)} \, : \, -\infty \leq \xi \leq \infty \right\}.
        \end{equation*}
        Then the self-adjoint operator $K^\Gamma \colon \mathcal{E} \to \mathcal{E}$, where $\mathcal{E}$ is the energy space of $\Gamma$, has essential spectrum
        \begin{equation} \label{eq:essspecE}
        \sigma_{\ess}(K^\Gamma, \mathcal{E}) = \bigcup_{n=-\infty}^\infty \Sigma_n.
        \end{equation}
        Hence, the spectrum of $K^\Gamma$ consists of this interval and a sequence of real eigenvalues $\{\lambda_k\}$ with no limit point outside of it,
                \begin{equation*}
                \sigma(K^\Gamma, \mathcal{E}) = \{\lambda_k\} \cup \sigma_{\ess}(K^\Gamma, \mathcal{E}).
                \end{equation*}
\end{thm}
\begin{rmk}
        Again, this theorem extends to surfaces $\Gamma$ with a finite number of axially symmetric conical points, each conical point contributing an interval of the type \eqref{eq:essspecE} to the essential spectrum.
\end{rmk}
\begin{proof}
Let $\rho$ be as in Lemma~\ref{lem:essspecloc}. In the notation of Theorem~\ref{thm:Eperturbcurve} we already know that
$$\sigma_{\ess}(K^\Gamma, \mathcal{E}) = \sigma_{\ess}(K^{\Gamma_c}, \mathcal{E}_c) = \sigma_{\ea}(K^{\Gamma_c}, \mathcal{E}_c) = \sigma_{\ea}(M_\rho K^{\Gamma_c} M_\rho, \mathcal{E}_c).$$
For an element $f(t, \theta) \in \mathcal{E}_c$ supported in $[0,1/2] \times [0,2\pi]$ we have by \eqref{eq:energynormcutoff} that
$$\|f\|_{\mathcal{E}_c} \simeq \|f\|_{H^{-1/2}(\Gamma_c)} = \|f\|_{H^{-1/2}(\Gamma_\alpha)} \simeq \|f\|_{\mathcal{E}(\Gamma_\alpha)}.$$
In other words, the energy norms on $\Gamma_c$ and $\Gamma_\alpha$ are comparable for such elements. If $(g_n) \subset \mathcal{E}(\Gamma_\alpha)$ is a singular sequence for $M_\rho K^{\Gamma_\alpha} M_\rho$ it is clear that $(M_\rho g_n)$ also is a singular sequence. The same statement is also true with $\Gamma_c$ in place of $\Gamma_{\alpha}$. Combined with Lemma~\ref{lem:essspecloc} it follows that 
$$\sigma_{\ess}(K^\Gamma, \mathcal{E}) = \sigma_{\ea}(M_\rho K^{\Gamma_c} M_\rho, \mathcal{E}_c) =  \sigma_{\ea}(M_\rho K^{\Gamma_\alpha} M_\rho, \mathcal{E}(\Gamma_{\alpha})) = \bigcup_{n=-\infty}^\infty \Sigma_n. \qedhere$$
\end{proof}

\section{Numerical experiments} 
\label{sec:numerical}

This section is organized as follows: Section~\ref{sec:indicator}
introduces an indicator function $\kappa_n^{\mathrm{ind}}$ which highlights the spectral properties of $K^\gamma_n$. The
function $\kappa_n^{\mathrm{ind}}$ is based on a generalization of
the polarizability $\omega_{jj}$ in~(\ref{eq:polariz}) and
bears some resemblance to the function $\alpha_\sharp$
of~\cite[Eq.~(4.8)]{HKL16}. Section~\ref{sec:nummet} reviews an
efficient strategy for the numerical solution of integral equations of
the type
\begin{equation}
\left(K_n^\gamma-z\right)\rho_n(t)=g_n(t),
\label{eq:inteq2F}
\end{equation}
needed to compute $\omega_{jj}$ and $\kappa_n^{\mathrm{ind}}$ to
high precision in the numerical examples of Section~\ref{sec:numres}.

\subsection{Polarizability and the indicator function} \label{sec:indicator}
Let $\Gamma$ be a closed surface of revolution with a conical point of
opening angle $2\alpha$, obtained by revolving a $C^5$-curve $\gamma$,
as described in Section~\ref{sec:perturb}. By rotational invariance,
following Section~\ref{sec:Especres}, the energy space on $\Gamma$ has
the orthogonal decomposition
$$\mathcal{E} = \bigoplus_{n=-\infty}^\infty \mathcal{E}_n,$$
where the norm on $\mathcal{E}_n$ is given by
$$\|f_n\|_{\mathcal{E}_n} = \langle S_n^{\gamma} f_n, f_n \rangle_{L^2([0,1], \gamma_1(t)|\gamma'(t)| \, dt)}.$$
The double layer potential $K^\Gamma \colon \mathcal{E} \to \mathcal{E}$ acts diagonally in this decomposition,
$$K^\Gamma \simeq_{\ue} \bigoplus_{n=-\infty}^\infty K^{\gamma}_n.$$
The operators $K^{\gamma}_n \colon \mathcal{E}_n \to \mathcal{E}_n$ were defined in Section~\ref{sec:L2spec} (but their action was considered on a different space), and the operators $S_n^{\gamma}$ are defined analogously (cf. Section~\ref{sec:transmission}). Hence, to numerically study the spectrum of $K^\Gamma$, we may consider the modal operators $K_n^{\gamma}$ separately. To accomplish this, we will now follow the symmetrization scheme of \cite[Section 5]{HP13}.

Since $S^\Gamma \colon L^2(\Gamma) \to L^2(\Gamma)$ is a positive operator, the same is true of 
$$S_n^{\gamma} \colon L^2([0,1], \gamma_1(t)|\gamma'(t)| \, dt) \to L^2([0,1], \gamma_1(t)|\gamma'(t)| \, dt).$$ By definition, the square root of $S_n^{\gamma}$ extends to a unitary operator
$$(S_n^\gamma)^{1/2} = \sqrt{S_n^{\gamma}} \colon \mathcal{E}_n \to L^2([0,1], \gamma_1(t)|\gamma'(t)| \, dt).$$
We denote its inverse by $(S_n^\gamma)^{-1/2}$.
Let $\mathcal{E}_n^*$ be the dual space of $\mathcal{E}_n$, with respect to the $L^2$-pairing. It is straightforward to verify that $\mathcal{E}_n^* \subset L^2([0,1], \gamma_1(t)|\gamma'(t)| \, dt)$. Then, by duality, $\sqrt{S_n^{\gamma}}$ is also unitary as an operator
$$\sqrt{S_n^{\gamma}} \colon L^2([0,1], \gamma_1(t)|\gamma'(t)| \, dt) \to \mathcal{E}_n^*.$$
Let $A_n^\gamma \colon L^2([0,1], \gamma_1(t)|\gamma'(t)| \, dt) \to L^2([0,1], \gamma_1(t)|\gamma'(t)| \, dt)$ be the self-adjoint operator
$$A_n^\gamma = (S_n^\gamma)^{1/2}K^{\gamma}_n(S_n^\gamma)^{-1/2},$$
unitarily equivalent to $K^\gamma_n$.
Since $A_n^\gamma$ is self-adjoint, it is by the spectral theorem associated with a spectral resolution $dE_n$, obviously equivalent to the spectral resolution of $K_n^\gamma$. Suppose that $u \in \mathcal{E}_n$ and $v \in \mathcal{E}_n^*$ are real-valued. Both $S^\gamma_n$ that $K^{\gamma}_n$ map real-valued distributions to real-valued distributions, since their kernels are real (see the Appendix). It follows that the measure 
$$d\langle E_n (S_n^\gamma)^{1/2}u, (S_n^\gamma)^{-1/2}v \rangle = d\langle E_n(\cdot) (S_n^\gamma)^{1/2}u, (S_n^\gamma)^{-1/2}v\rangle_{L^2([0,1], \gamma_1(t)|\gamma'(t)| \, dt)}$$ 
is real. 

For $z = x + iy$, $y \neq 0$, we now let
\begin{equation} \label{eq:kappadef}
\kappa_n(u,v, z) = \Im \mathrm{m} \, \langle (K^{\gamma}_n-z)^{-1} u, v \rangle_{L^2([0,1], \gamma_1(t)|\gamma'(t)| \, dt)}.
\end{equation}
Since the spectral measure $dE_n$ is supported on $[-1,1]$ by \eqref{eq:Kspeccontain}, we have that
\begin{equation} \label{eq:kappapois}
\kappa_n(u,v, z) = \Im \mathrm{m} \, \int_{-1}^{1} \frac{d\langle  E_n(s) (S_n^\gamma)^{1/2}u, (S_n^\gamma)^{-1/2} v \rangle}{s-z}  = y\int_{-1}^1 \frac{d\langle  E_n(s) (S_n^\gamma)^{1/2}u, (S_n^\gamma)^{-1/2}v \rangle}{(s-x)^2 + y^2}.
\end{equation}

The imaginary part of the polarizability $\omega_{33}(z)$ in the $\bm{r}_3$-direction (see Section~\ref{sec:introtrans}) can in the current notation be expressed as
$$\Im \mathrm{m} \, \omega_{33}(z) = \kappa_{0}((g_{\bm{e}_3})_0, (h_{\bm{e}_3})_0, z).$$
This is because $g_{\bm{e}_3}$ and $h_{\bm{e}_3}$ are independent of $\theta$, $\Gamma$ parametrized by \eqref{eq:gammaparamet2}, so that
\begin{equation} \label{eq:polarmode0}
g_{\bm{e}_3}(t, \theta) = \frac{1}{\sqrt{2\pi}}(g_{\bm{e}_3})_0(t), \quad h_{\bm{e}_3}(t, \theta) = \frac{1}{\sqrt{2\pi}}(h_{\bm{e}_3})_0(t), \qquad \theta \in [0, 2\pi], \; 0 \leq t \leq 1.
\end{equation}
In the $\bm{r}_1$-direction, $g_{\bm{e}_1}$ and $h_{\bm{e}_1}$ have non-zero Fourier coefficients for $n = \pm 1$, and
\begin{equation} \label{eq:polarmode1}
g_{\bm{e}_1}(t, \theta) = \frac{1}{\sqrt{2\pi}}\left((g_{\bm{e}_1})_{-1}(t)e^{-in\theta} + (g_{\bm{e}_1})_{1}(t)e^{in\theta}\right), \quad h_{\bm{e}_1}(t, \theta) = \frac{1}{\sqrt{2\pi}}\left((h_{\bm{e}_1})_{-1}(t)e^{-in\theta} + (h_{\bm{e}_1})_{1}(t)e^{in\theta}\right),
\end{equation}
for $\theta \in [0, 2\pi]$, $0 \leq t \leq 1$. By symmetry it follows that
$$\Im \mathrm{m} \, \omega_{11}(z) = 2\kappa_{1}((g_{\bm{e}_1})_1, (h_{\bm{e}_1})_1, z),$$
and a similar formula holds for $\omega_{22}(z)$. Hence, \eqref{eq:kappapois} shows that there indeed is a spectral measure $\mu_j$, $j=1,2,3$, such that $\omega_{jj}(z)$ can be represented as the Cauchy integral of $\mu_j$, as claimed in \eqref{eq:polarcauchy}.

We can draw several conclusions from the representation \eqref{eq:kappapois} of $\kappa_n$ as a Poisson integral. If 
$$d\mu_{u,v}(s) = d\langle  E_a(s)(S_n^\gamma)^{1/2}u, (S_n^\gamma)^{-1/2}v \rangle$$
 has absolutely continuous support around the point $x$, then almost surely
$$\lim_{y \to 0^-}  \kappa_n(u,v, x+iy) = -\pi \mu'_{u,v}(x) = -\pi \lim_{h \to 0} \frac{\mu_{u,v}(x-h, x+h)}{2h}\neq 0.$$
For any $x$, it holds that
$$\lim_{y \to 0^-}  y\kappa_n(u,v, x+iy) = \pi \mu_{u,v}(\{x\}).$$
In particular, if $\mu_{u,v}(\{x\}) \neq 0$, i.e. if $x$ is an $\mathcal{E}_n$-eigenvalue of $K_n^\gamma$ excited by $u$ and $v$, then $|\kappa_n(u,v, x+iy)| \simeq 1/y$ as $y \to 0^-$.
 On the other hand, if $x$ lies outside the support of $\mu_{u,v}$ entirely, then $|\kappa_n(u,v, x+iy)| \simeq y$ as $y \to 0^-$. In fact, if $x$ lies outside the spectrum of $K^{\gamma}_n$, then $|\kappa_n(u,v, x+iy)|/y$ is uniformly bounded in $y < 0$ and real-valued $u \in \mathcal{E}_n$, $v \in \mathcal{E}_n^\ast$ of norm less or equal to $1$. In general, it is not clear how to recover the singular continuous spectrum from $\kappa_n$.

We introduce the indicator function
$$\kappa_n^{\mathrm{ind}} (u, v, x + iy) = \frac{\kappa_n^\Delta (u, v, x + iy) + 1}{2},$$
where
$$\kappa_n^\Delta (u,v, x + iy) = \log_{10}|\kappa_n(u,v, x+iy)| - \log_{10}|\kappa_n(u, v, x+10iy)|.$$
Our numerical experiments rely on the following properties, evident from the preceding discussion.
\begin{itemize}
\item If $x$ belongs to the absolutely continuous support of $\mu_{u,v}$, then almost surely, we have that
$$\lim_{y \to 0^-} \kappa_n^{\mathrm{ind}} (u, v, x + iy) = \frac{1}{2}.$$
\item If $\mu_{u,v}(\{x\}) \neq 0$, then
$$\lim_{y \to 0^-} \kappa_n^{\mathrm{ind}} (u, v, x + iy) = 1.$$
\item If $x$ lies outside the support of $\mu_{u,v}$, then
$$\lim_{y \to 0^-} \kappa_n^{\mathrm{ind}} (u, v, x + iy) = 0.$$
\end{itemize}
We finish this discussion by describing how $u$ and $v$ are chosen, for
a given $x \in \mathbb{R}$, in the framework of our numerical method
presented in Section~\ref{sec:nummet}. To each level of discretization
$\mathfrak{d}$ we associate a finite-dimensional space
$\mathcal{F}_{\mathfrak{d}} \subset \mathcal{E}_n^*$ of piecewise
polynomial functions. The spaces $\mathcal{F}_{\mathfrak{d}}$ increase
as $\mathfrak{d}$ gets finer, and their union is dense in
$\mathcal{E}_n^*$. Note that even when the discretization is rough,
our numerical method still computes $(K^\gamma_n-z)^{-1}u(t)$ to a
very high accuracy for every $t\neq 0$. The level of discretization
only limits the choice of functions $u$.

For a given $x$, we choose a small number $y_0 < 0$, and let $u = \breve{u}$ and $v = \breve{v}$ be the maximizers of the supremum
\begin{equation} \label{eq:uvchoice}
\sup_{\substack{u,v \in \mathcal{F}_{\mathfrak{d}}, \\\|u\|_{L^2}=\|v\|_{L^2} = 1}} \kappa_n(u,v, x+iy_0) < \infty.
\end{equation}
Before continuing, a subtle remark is required: we know from Theorem~\ref{thm:winding} that it often happens that
$$
\sup_{\|u\|_{L^2}=\|v\|_{L^2}=1} \kappa_n(u,v, x+iy_0) = \infty.
$$
already for $y_0 < 0$. In particular, \eqref{eq:uvchoice} may become arbitrarily large as $\mathfrak{d}$ gets finer. However, the $L^2$-, $\mathcal{E}_n$-, and $\mathcal{E}_n^*$-norms are all equivalent on $\mathcal{F}_{\mathfrak{d}}$, since it is a finite-dimensional space. Therefore, \eqref{eq:uvchoice} is certainly finite for a given $\mathfrak{d}$, since
$$
\sup_{\|u\|_{\mathcal{E}_n}=\|v\|_{\mathcal{E}_n^*}=1}
\kappa_n(u,v, x+iy_0) = \|(K_n^\gamma-x-iy_0)^{-1}\| < \infty, \quad
y_0 < 0.
$$
Based on the equivalence of these norms, we will soon see that it is
sound to maximize \eqref{eq:uvchoice} in the $L^2$-norm. In fact,
since the $\mathcal{E}_n$- and $\mathcal{E}_n^*$-norms are more
expensive to compute and also difficult to apply to the numerical
maximization of \eqref{eq:uvchoice}, cf. \cite[Section~4]{HKL16}, it
turns out that our numerical approach yields much better results when
we maximize \eqref{eq:uvchoice} in the $L^2$-norm. The functions $u =
\breve{u}$ and $v = \breve{v}$ are best interpreted as moderately
aggressive test functions for which we are guaranteed that
$\kappa_n(\breve{u},\breve{v}, x+iy_0) > 0$ is relatively large.

By the spectral theorem, there is a unitary operator $V$ that carries
$A^\gamma_n$ onto a multiplication operator $M_\varphi$ on $L^2(X,
\nu)$, for some positive measure $\nu$, $A^\gamma_n = V^{-1}M_\varphi
V.$ For subsets $F \subset [-1,1]$ we then have that
$$\mu_{u,v}(F) = \int_{\varphi^{-1}(F)} V(S^\gamma_n)^{1/2}u \overline{V(S^\gamma_n)^{-1/2}v} \, d\nu.$$

Suppose first that $x$ is an eigenvalue of $K^\gamma_n$, i.e. $\nu(\varphi^{-1}(\{x\})) \neq 0$, and that it is isolated. Suppose that $\mathfrak{d}$ is so fine that some eigenvector of $A^\gamma_n$, to the eigenvalue $x$, is not orthogonal to either $(S^\gamma_n)^{1/2}\mathcal{F}_{\mathfrak{d}}$ or $(S^\gamma_n)^{-1/2}\mathcal{F}_{\mathfrak{d}}$. For any $\delta > 0$, the Poisson integral at $x$ of $\chi_{\{|s - x| > \delta\}} \mu_{u,v}$ tends to zero as $y \to 0$, uniformly in vectors $u\in \mathcal{E}_n$, $v\in\mathcal{E}_n^*$ of norm $1$. Here $\chi$ denotes the characteristic function of the indicated set.
 In view of \eqref{eq:kappapois} and the equivalence of all norms on $\mathcal{F}_\mathfrak{d}$, it follows that if $y_0$ is sufficiently small then it must be that the maximizers $\breve{u}$ and $\breve{v}$ of \eqref{eq:uvchoice} satisfy that $V(S^\gamma_n)^{1/2}\breve{u}(\gamma) \neq 0$ and $V(S^\gamma_n)^{-1/2}\breve{v}(\gamma) \neq 0$ for some $\gamma \in \varphi^{-1}(\{x\})$. Hence 
$$\lim_{y \to 0^-} \kappa_n^{\mathrm{ind}} (\breve{u}, \breve{v}, x + iy) = 1,$$
as soon as $y_0 < 0$ is sufficiently small and the discretization $\mathfrak{d}$ is sufficiently fine.

Our conclusion for the non-discrete spectrum appears to be a little less satisfying without further a priori knowledge about the spectral measures of $K_n^\gamma$.
Let $\delta > 0$, and suppose that $K^\gamma_n$ has absolutely continuous spectrum in the interval $(x-\delta, x+\delta)$. Then the same argument yields, if $y_0 < 0$ is sufficiently small and $\mathfrak{d}$ sufficiently fine, that $V(S^\gamma_n)^{1/2}\breve{u}$ and $V(S^\gamma_n)^{-1/2}\breve{v}$ could not be identically zero in $\varphi^{-1}((x-\delta, x+\delta))$ with respect to $\nu$-measure. Hence there are points $x' \in (x-\delta, x+\delta)$ for which
$$\lim_{y \to 0^-} \kappa_n^{\mathrm{ind}} (\breve{u}, \breve{v}, x' + iy) = \frac{1}{2}.$$
We do expect that the spectral resolution of $K_n^\gamma$ is actually well behaved, allowing for stronger conclusions. In particular, we believe that $K_n^\gamma$ never has a singularly continuous spectrum. However, a rigorous study of the spectral measures is beyond the scope of this article.

\subsection{Numerical method}
\label{sec:nummet}

We rely on high-order panel-based Nystr{\"o}m
discretization~\cite{HK14} to solve~(\ref{eq:inteq2F}). The resolution
requirements for the layer density $\rho_n(t)$ may lead to a giant
linear system that we never form explicitly, but instead solve using a
technique called recursively compressed inverse preconditioning
(RCIP)~\cite{HO08}. A homotopy method~\cite[Section~6.3]{H11} to capture
limits $y\to 0^-$ is another key ingredient in the numerical scheme.

RCIP-accelerated Nystr{\"o}m schemes have previously been used to
solve Fredholm integral equations of the second kind related to the  Neumann--Poincar\'e operator on Lipschitz
surfaces~\cite{H11,HKL16,HP13} and electromagnetic resonances in
axially symmetric domains with sharp edges~\cite{HK16}. It would bring us
too far from the scope of this article to give a complete account of these schemes. We
refer instead to the original papers~\cite{H11,HK14,HK16,HP13} and to
the compendium~\cite{H16}. Below follows a short summary, compiled
from~\cite{HK16,HP13}, giving special attention to differences in the present
implementation from that of~\cite{HK16}.

\subsubsection{Main features of RCIP-acceleration}

In the Nystr{\"o}m discretization of an integral equation, the integral operator is approximated by numerical
quadrature and the resulting semi-discrete equation is enforced
at the quadrature nodes, leading to a linear system for the unknown
layer density at the nodes. The idea behind RCIP-acceleration is to
transform the integral equation into a preconditioned form where the
layer density has better regularity and can be resolved with fewer
nodes. In the present context, the operator in~(\ref{eq:inteq2F}) is split into two parts,
\begin{displaymath}
K_n^\gamma=K_n^{\gamma\star}+K_n^{\gamma\circ},
\end{displaymath}
where $K_n^{\gamma\star}$ describes the kernel interaction close to the
origin and $K_n^{\gamma\circ}$ is a compact operator. The change
of variables
\begin{displaymath}
\rho_n(t)=\left(I-\tfrac{1}{z}K_n^{\gamma\star}\right)^{-1}\tilde{\rho}_n(t)
\end{displaymath}
lead to the right preconditioned equation
\begin{equation}
\left(K_n^{\gamma\circ}\left(I-\tfrac{1}{z}K_n^{\gamma\star}\right)^{-1}
         -z\right)\tilde{\rho}_n(t)=g_n(t).
\label{eq:RCIP}
\end{equation}

The functions $\tilde{\rho}_n(t)$ and $g_n(t)$ in~(\ref{eq:RCIP}) share the same
regularity and they, along with $K_n^{\gamma\circ}$, are discretized on a coarse uniform mesh of
panels on $\gamma$, as in
standard Nystr{\"o}m discretization. The resulting grid $\mathfrak{d}$ is assumed to
resolve these quantities so that their values at arbitrary points on
$\gamma$ can be recovered by piecewise interpolation with polynomials
of no higher degree than that of the underlying quadrature. The local
resolvent $\left(I-\tfrac{1}{z}K_n^{\gamma\star}\right)^{-1}$, on the
other hand, is discretized on a mesh of panels that is almost
infinitely dyadically refined towards the origin. The
corresponding grid $\mathfrak{d}_{\rm fin}$ is assumed to resolve
$\rho_n(t)$. Then $\mathfrak{d}_{\rm fin}$ is coarsened via a
recursive procedure where, in each step, the two smallest panels on
the fine mesh are merged and the part of the local resolvent that
needs those panels for resolution is locally projected. Upon
completion, this process results in a compressed discrete local
resolvent on $\mathfrak{d}$. The computational cost for the
compression grows linearly with the number of refinement levels.
However, by the proof of Theorem~\ref{thm:perturbcurve} it is, for a given grid $\mathfrak{d}$, justified to replace $\gamma$ with a line segment in a very small neighborhood of the origin.  Then $K_n^\gamma(t,t')t'\,dt'$ becomes scale
invariant close to the origin, see~(\ref{eq:homog}), and therefore compression on the finer levels
can be carried out using Newton-accelerated fixed-point
iteration~\cite[Section~6.2]{H11}. The entire compression is then
performed in {\it sublinear} time and the memory requirements are
modest.

It is a very important feature of the RCIP compression scheme that solving the discrete version
of~(\ref{eq:RCIP}) does not lead to any loss of information whatsoever, compared
to solving the discrete version of~(\ref{eq:inteq2F}) entirely on
$\mathfrak{d}_{\rm fin}$, provided that the assumptions on resolution
mentioned in the preceding paragraph are met. When the discrete
version of~(\ref{eq:RCIP}) has been solved for $\tilde{\rho}_n(t)$ on
$\mathfrak{d}$, the original density $\rho_n(t)$ can easily be
reconstructed on $\mathfrak{d}_{\rm fin}$ given that certain
information about the compression has been saved~\cite[Section
9]{H16}. In applications $\rho_n(t)$ is often not needed, as it may be possible to compute quantities of interest directly from
$\tilde{\rho}_n(t)$ and the compressed local resolvent.

\subsubsection{Details particular to the present implementation}

There are several possible ways to compute the Fourier coefficients
$K_n^\gamma(t,t')$. We rely solely on the explicit formula~(\ref{eq:Knexp}) from the Appendix. The associated
Legendre functions are evaluated as outlined in~\cite[Section 5]{HK16}, with some minor improvements to the {\sc Matlab} code.

Numerical experiments strongly suggest that the general asymptotic
behavior of the density $\rho_n(t)$ in~(\ref{eq:inteq2F}) close to the
conical point, and with $z=x+i0^{\pm}$ in the absolutely continuous
spectrum, is
\begin{displaymath}
\rho(t)_n\propto\lvert\gamma(t)\rvert^{-1.5}
    \left(\cos(\xi(x)\log{\lvert\gamma(t)\rvert})
     \pm i\sin(\xi(x)\log{\lvert\gamma(t)\rvert})\right),\qquad t\to 0.
\end{displaymath}
Here $\xi(x)$ is a left inverse of the function $F_n(\xi)$ in
Theorem~\ref{thm:Especres}. Since $F_n(\xi)$, by the proof of
Lemma~\ref{lem:mellinsmooth}, behaves like $1/|\xi|$,
$|\xi|\to\infty$, and has no zero other than at $\pm\infty$, we have
that $|\xi| \simeq 1/|x|$. This means that the number of nodes on each
panel on the fine mesh, in the framework of dyadic mesh refinement,
must increase as $1/|x|$ in order for $\mathfrak{d}_{\rm fin}$ to
resolve $\rho_n(t)$ as $x\to 0$ in the absolutely continuous spectrum.
This stands in stark contrast to the situation of a sharp edge
\cite{HK16} or a corner in 2D \cite{HKL16}, where the need for
resolution of the corresponding density only appears to grow on the
order of $-\log{|x|}$. To cope with this need for high resolution we
use 32-point composite Gauss--Legendre quadrature as the underlying
quadrature in the Nystr{\"o}m scheme, and place up to 1024 nodes on
each panel of the fine mesh. This way, we can accurately resolve
$\rho_n(t)$ for $|x|$ as small as $0.001$.

As a final remark, our scheme does not directly compute the action of the operator
$K_n^\gamma$ in~(\ref{eq:inteq2F}). Instead we do a change of
variables and work with the transformed kernel
\begin{displaymath}
\widetilde{K}_n^\gamma(t,t') = tK_n^\gamma(t,t') (t')^{-1}.
\end{displaymath}
Nystr{\"o}m discretization of the transformed equation then resembles
a norm-preserving discretization on $L^1$ in the terminology
of~\cite{AG14}, and RCIP-acceleration is still applicable. We have
observed slightly better stability using the transformed equation, in
line with the discussion in~\cite{B12}. In particular, for $n=0$ the
transform avoids the formation of a false near-resonance at the
rightmost point of $\sigma_{\ess}(K^\Gamma, L^2)$, a point which in
general is not part of the $\mathcal{E}_0$-spectrum of $K_0^\gamma$,
see Theorems~\ref{thm:winding} and \ref{thm:Eessspec}.

Let $Uf(t) = tf(t)$ and denote by $\widetilde{\mathcal{E}}_n$, $n \in
\mathbb{N}$, the Hilbert space defined by the requirement that $U
\colon \mathcal{E}_n \to \widetilde{\mathcal{E}}_n$ be unitary, cf.
Section~\ref{sec:Especres}. We consider $\widetilde{K}_n^\gamma$ as an
operator on $\widetilde{\mathcal{E}}_n$, and then
\begin{displaymath}
\widetilde{K}_n^\gamma = U K_n^\gamma U^{-1},
\end{displaymath}
With the unitary $U$ in hand, it is now straightforward to implement
the minor modifications to the framework of
Section~\ref{sec:indicator}, needed to directly consider the indicator
function for the transformed operator. We leave the precise details to
the reader.

\subsection{Numerical results}
\label{sec:numres}

Our Nystr{\"o}m scheme for~(\ref{eq:inteq2F}), needed for computing~(\ref{eq:polariz}) and~(\ref{eq:kappadef}), is implemented in {\sc
  Matlab} and executed on a workstation equipped with an Intel Core
i7-3930K CPU and 64 GB of memory. In all numerical experiments
$\gamma$ is the curve
\begin{displaymath}
\gamma(t) = \sin\left(\frac{\pi}{2}t\right)\left( \sin((1-t)\alpha),
  \cos((1-t)\alpha) \right), \quad t\in [0,1],
\end{displaymath}
generating a surface $\Gamma$ with a conical point of opening angle
$2\alpha$, $0<\alpha<\pi$, $\alpha\neq\pi/2$.

The excluded case $\alpha=\pi/2$ corresponds to a sphere. For the
sphere $K^\Gamma\colon\mathcal{E}\to\mathcal{E}$ has the eigenvalues
$x_i=1/(2i-1)$, $i=1,2,\ldots$, but the spectral measure $\mu$ of the
polarizability only has a single atom, at $x=1/3$. Note that the
sphere has the same polarizability in all directions, and hence the
spectral measure is independent of direction. We say that $x=1/3$ is a
{\it bright plasmon}, while the remaining eigenvalues of $K^\Gamma$
are {\it dark plasmons}.

Consider now the polarizability in the $\bm{r}_3$-direction for
$\alpha<\pi/2$, $\bm{e}=\bm{e}_3 = (0,0,1)$. When $\alpha$ shrinks the
non-discrete support of $\mu_3$ becomes wider and the bright plasmon
moves to the right until it disappears into the non-discrete support
at the angle $2\alpha\approx 0.91895945$. We choose
$2\alpha=5\pi/18<0.91895945$ for our experiments in
Sections~\ref{sec:limpol} and \ref{sec:specindic}, since in this case
we may test our results for $\mu_3$ against \eqref{eq:murule2} without
involving any bright plasmons. In Section~\ref{sec:reflex} we consider
the reflex angle $2\alpha=31\pi/18$. The surfaces $\Gamma$, for these
two opening angles, are illustrated in Figure~\ref{fig:surface}.

\subsubsection{Limit polarizability} 
\label{sec:limpol}

We first compute $\omega_{33}^-(x)$ and use \eqref{eq:rhorule} and
\eqref{eq:murule2} as indirect error estimates. The results are shown
in Figure~\ref{fig:polaracute}(a,b). Equation \eqref{eq:murule2},
discretized with adaptive 16-point composite quadrature and a total of
3136 nodes, holds with an estimated relative accuracy of
$5\cdot10^{-16}$. The absolute error in \eqref{eq:rhorule}, called
\textit{charge error} in Figure~\ref{fig:polaracute}(b), depends on
$x$ and varies from no measurable error to an error on the order of
$10^{-12}$.
\begin{figure}[ht]
\centering
\includegraphics[width =0.32\linewidth]{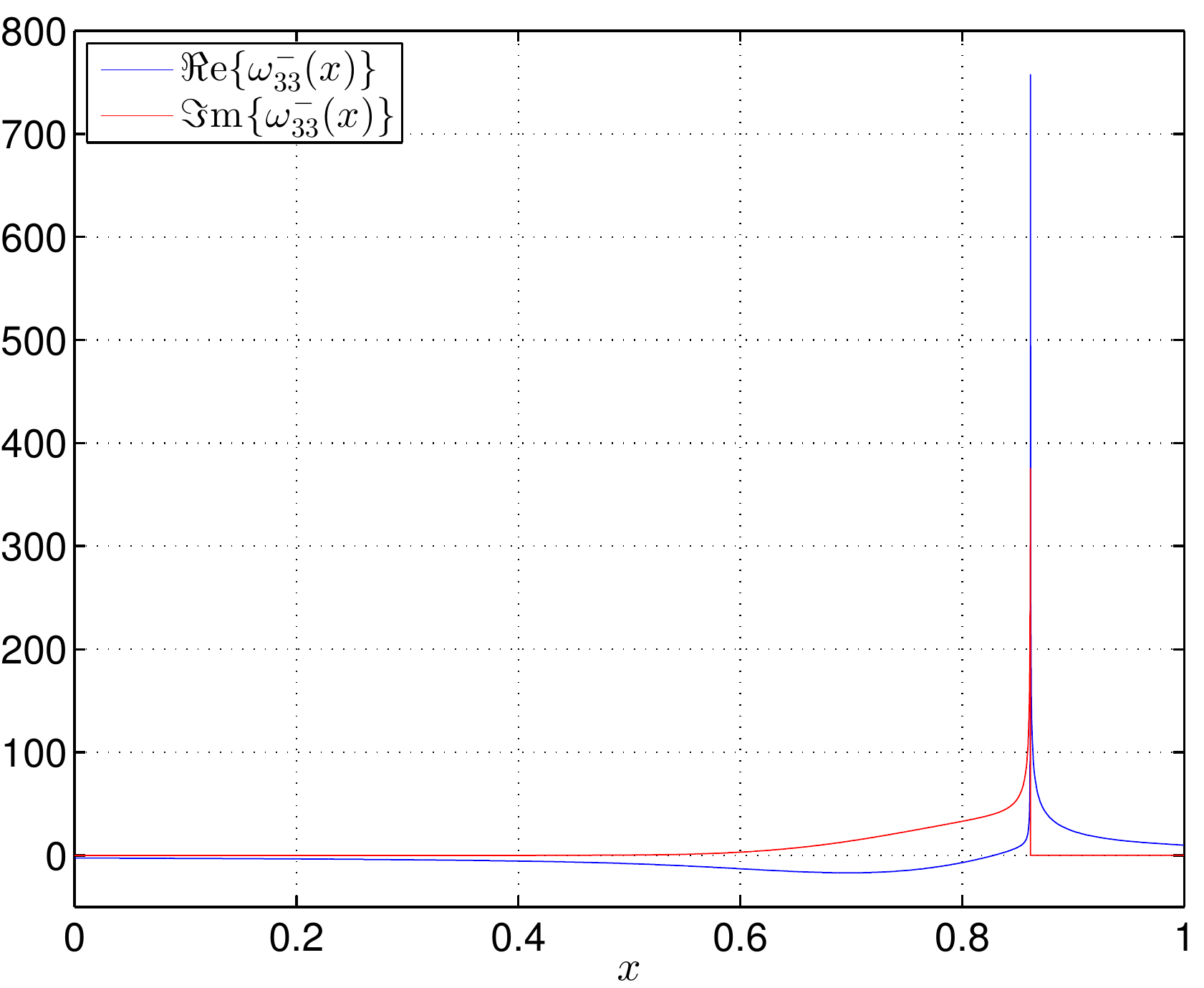}
\quad
\includegraphics[width =0.32\linewidth]{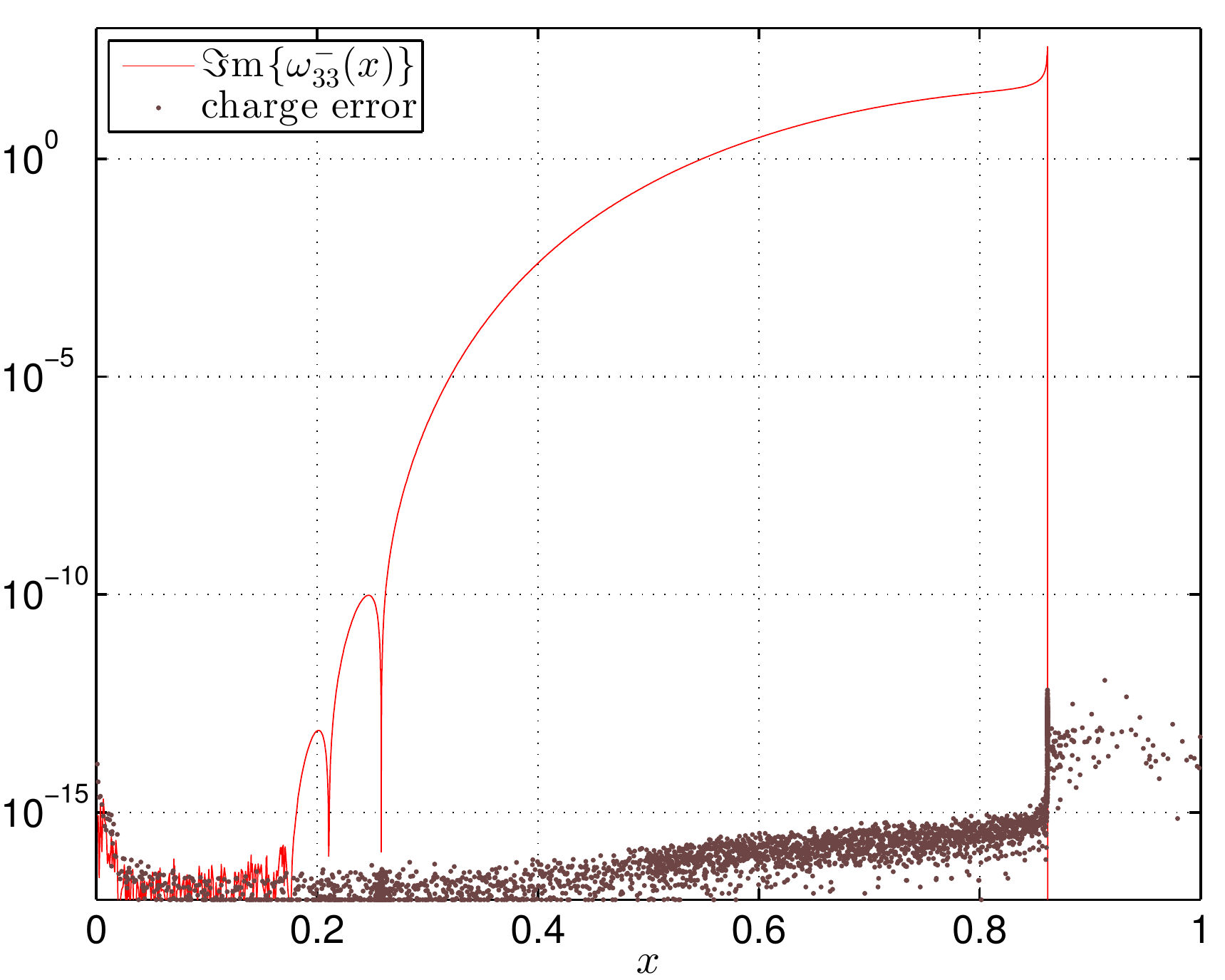} \\
\includegraphics[width =0.32\linewidth]{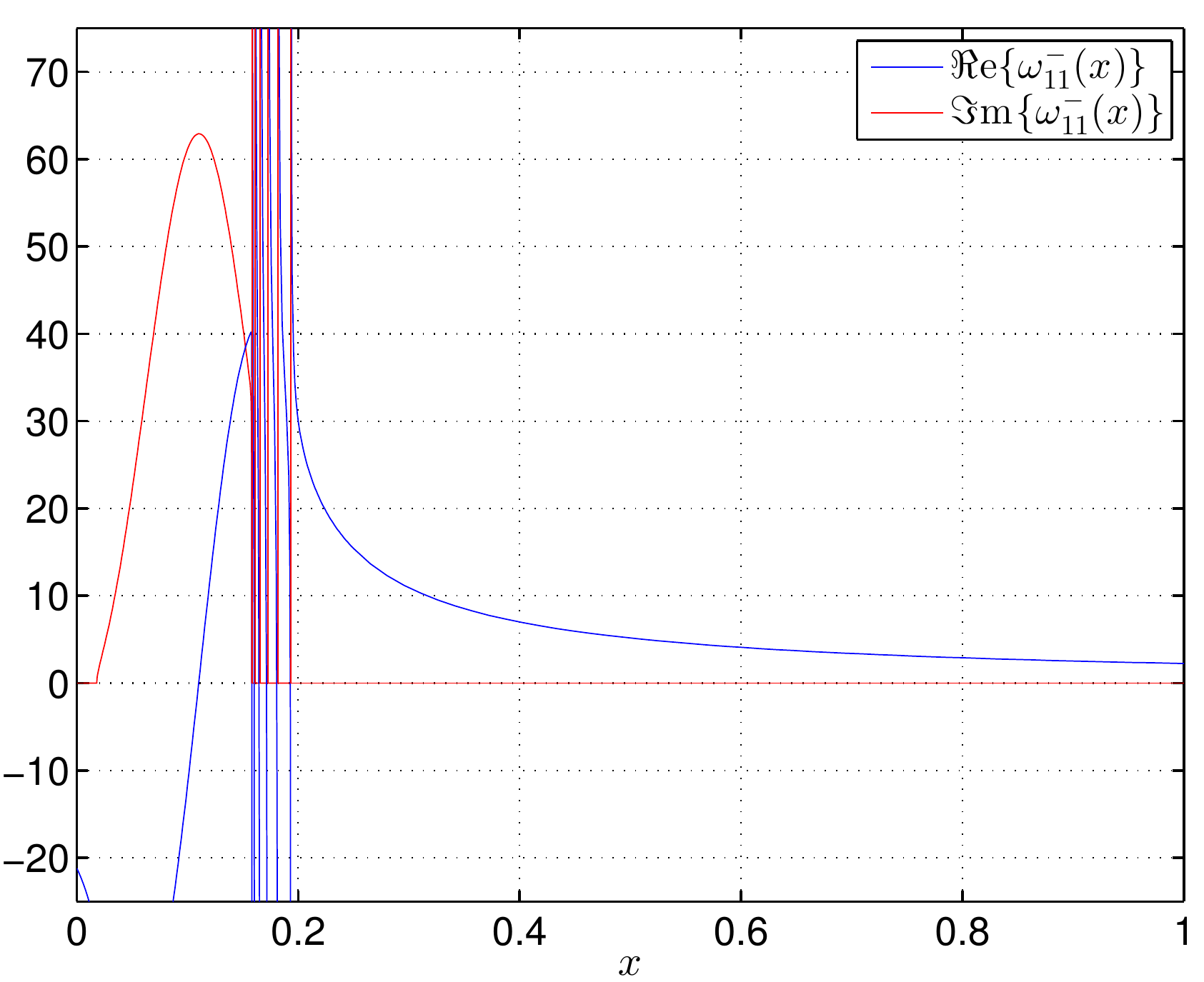}
\quad
\includegraphics[width =0.32\linewidth]{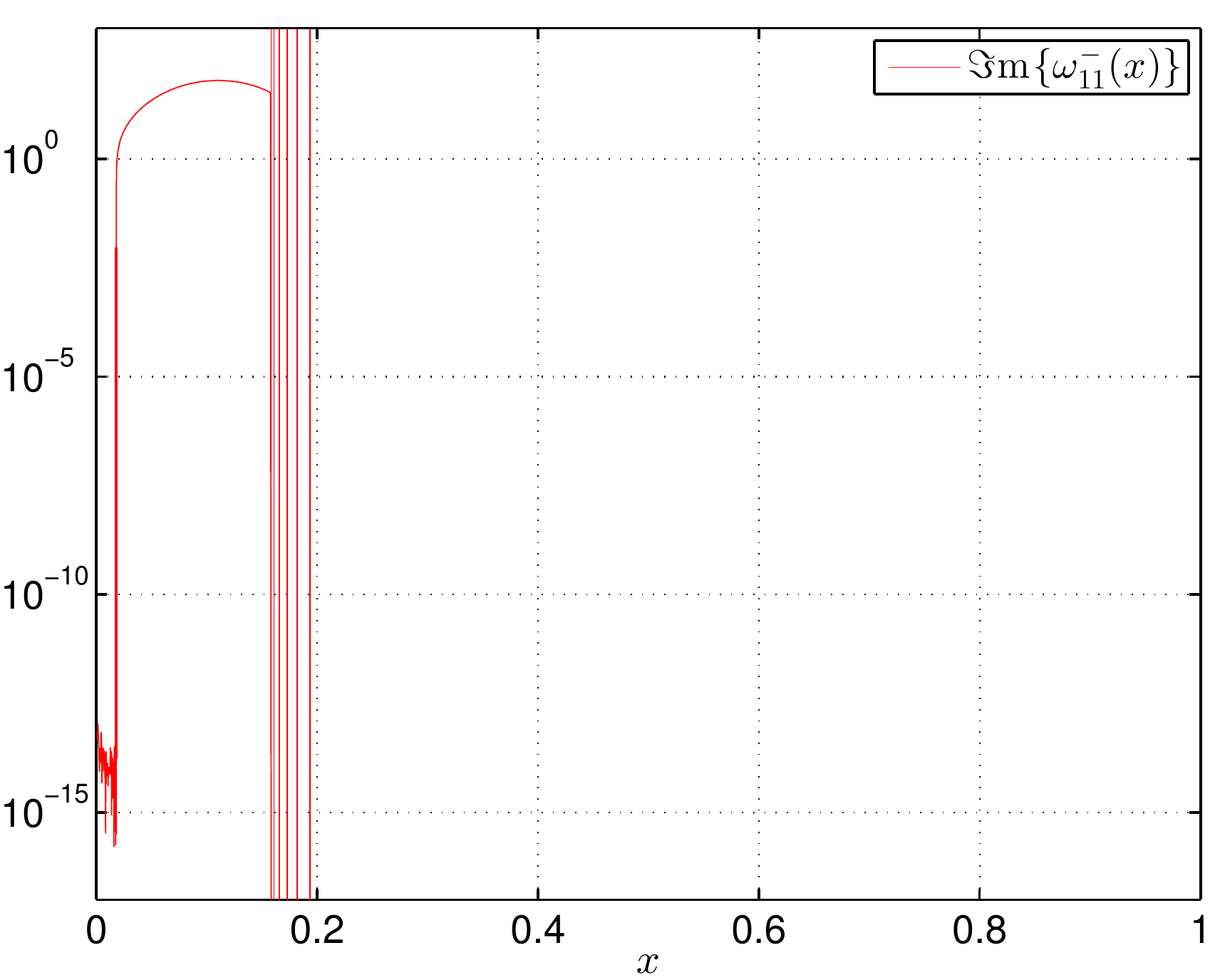}
\caption{(a,b): Limit polarizability in the $\bm{r}_3$-direction for 
  $2\alpha = 5\pi/18$. The right image displays the imaginary part,
  with a logarithmic scale on the vertical axis. (c,d): Limit
  polarizability in the $\bm{r}_1$-direction, featuring 6 bright
  plasmons.}
\label{fig:polaracute}
\end{figure}

The underlying data used to produce Figure~\ref{fig:polaracute}(a,b) shows that
$\Im\mathrm{m}\,\omega_{33}^-(x)=-\pi\mu_3'(x)$ is non-zero to the
left of $x\approx 0.861463506648456$. Note that this corresponds to
the right endpoint of the interval $\Sigma_0$ from
Theorems~\ref{thm:Especres} and \ref{thm:Eessspec}, $F_0(0) \approx
0.861463506648456$, and provides yet another piece of indirect
evidence that our numerical scheme is accurate. See
Figure~\ref{fig:Eessspec}(a) for an illustration of $F_0(\xi)$.
$\Im\mathrm{m}\,\omega_{33}^-(x)$ appears to have zeroes in $x\approx
0.258175$ and $x\approx 0.210575$. To the left of $x=0.18$ it is not
possible to determine whether $\Im\mathrm{m}\,\omega_{33}^-(x)$ is
non-zero, since the numerical results there are of the same order as
the numerical error.

Figure~\ref{fig:polaracute}(c,d) depicts $\omega_{11}^-(x)$. Six
bright plasmons are visible. Their locations $x_i$ and amplitudes $u_i
v_i$ are given in Table~\ref{tab:plasmamp}. Equation
\eqref{eq:murule2} holds with an estimated relative accuracy of
$3\cdot 10^{-16}$. The numerically visible support of $\Im\mathrm{m}\,
\omega_{11}^-(x)=-\pi\mu_1'(x)$ is $(0.018216722,0.15813952053635)$.
The right endpoint again corresponds to the right endpoint of the
interval $\Sigma_1$, $F_1(\pm 3.8202309)\approx 0.158139520536354$.
See Figure~\ref{fig:Eessspec}(b) for an illustration of $F_1(\xi)$.
The left endpoint corresponds to the local minimum of $F_1(\xi)$ at
$\xi=0$, $F_1(0)\approx 0.018216721972542$. Recall from
Section~\ref{sec:transmission} that, on the infinite straight cone,
$F_1(\xi)$ is an eigenvalue of $K_1^\alpha$ to the generalized
eigenvector $t^{-i\xi-3/2}$. Hence, for the infinite straight cone
there is a kind of singularity in the spectrum at $F_1(0)$: as $x\to
F_1(0)^+$ there are generalized eigenvectors with $\xi\to
0$, but as $x<F_1(0)$ all generalized eigenvectors to $x$ have large $\xi$ and therefore exhibit wild oscillations. It seems
likely that a similar phenomenon is responsible for the drastic change
in $\Im\mathrm{m}\,\omega_{11}^-(x)$ at $x=F_1(0)$.
\begin{figure}[ht]
\includegraphics[width =0.32\linewidth]{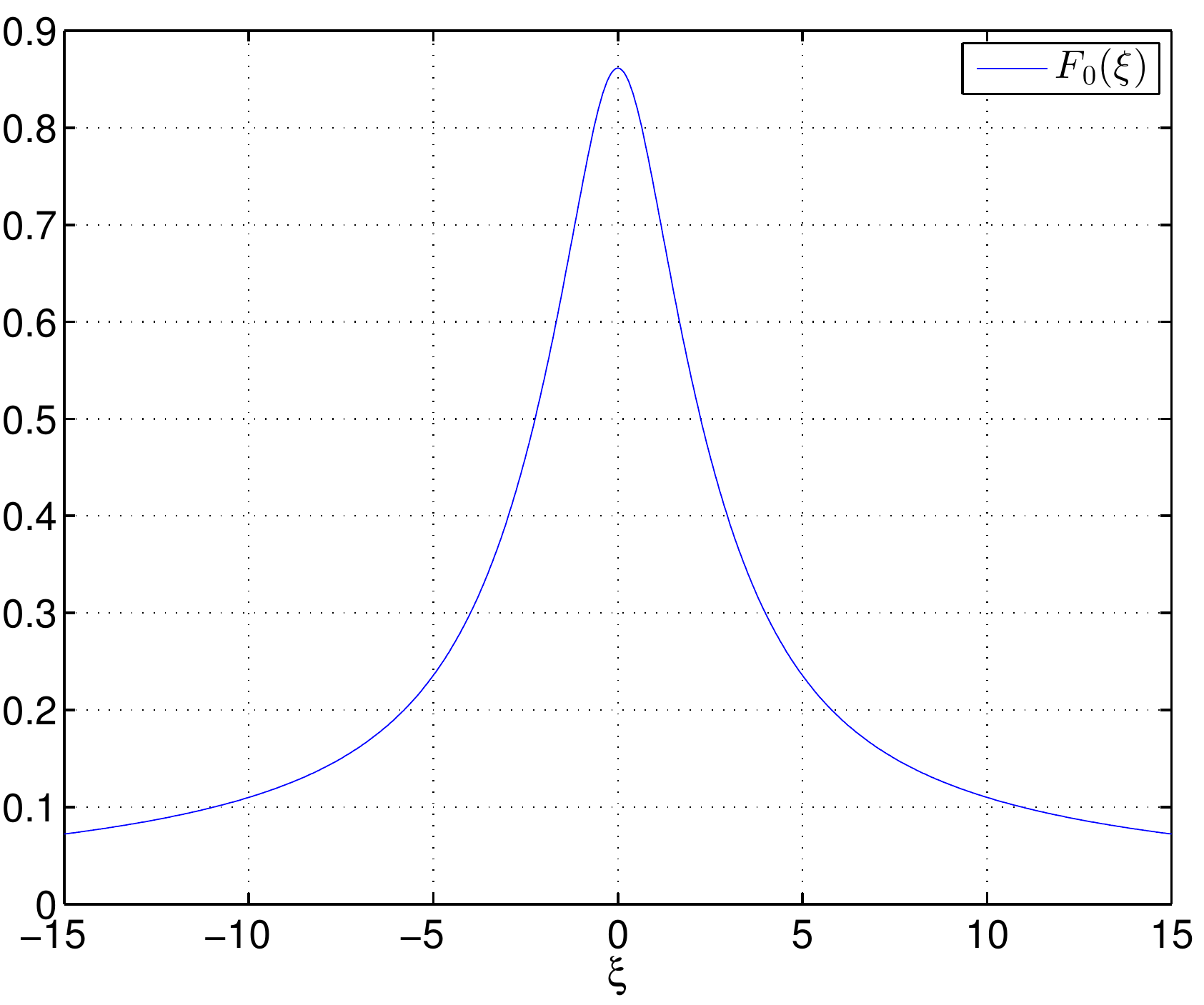}
\hfill
\includegraphics[width =0.32\linewidth]{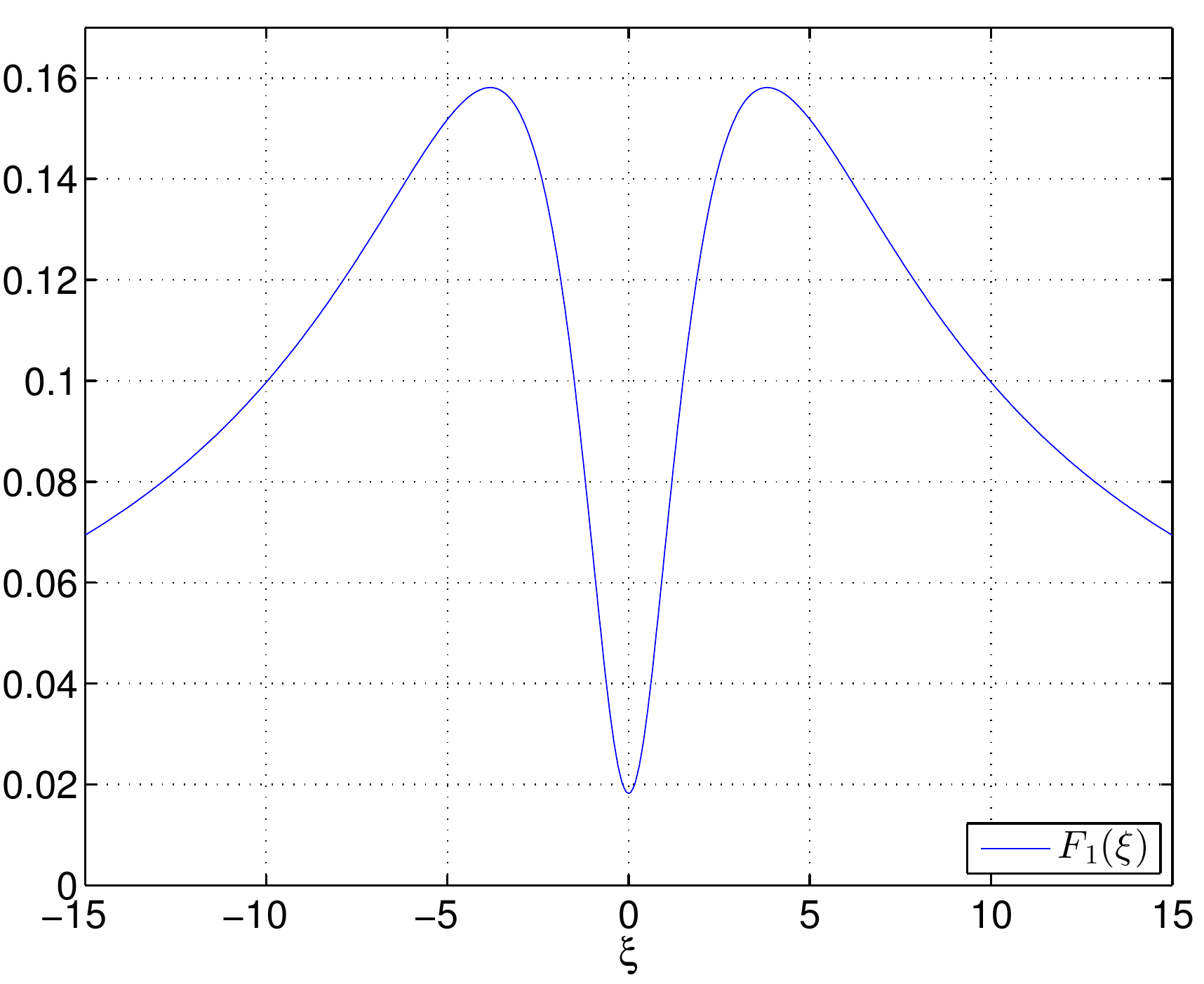}
\hfill
\includegraphics[width =0.32\linewidth]{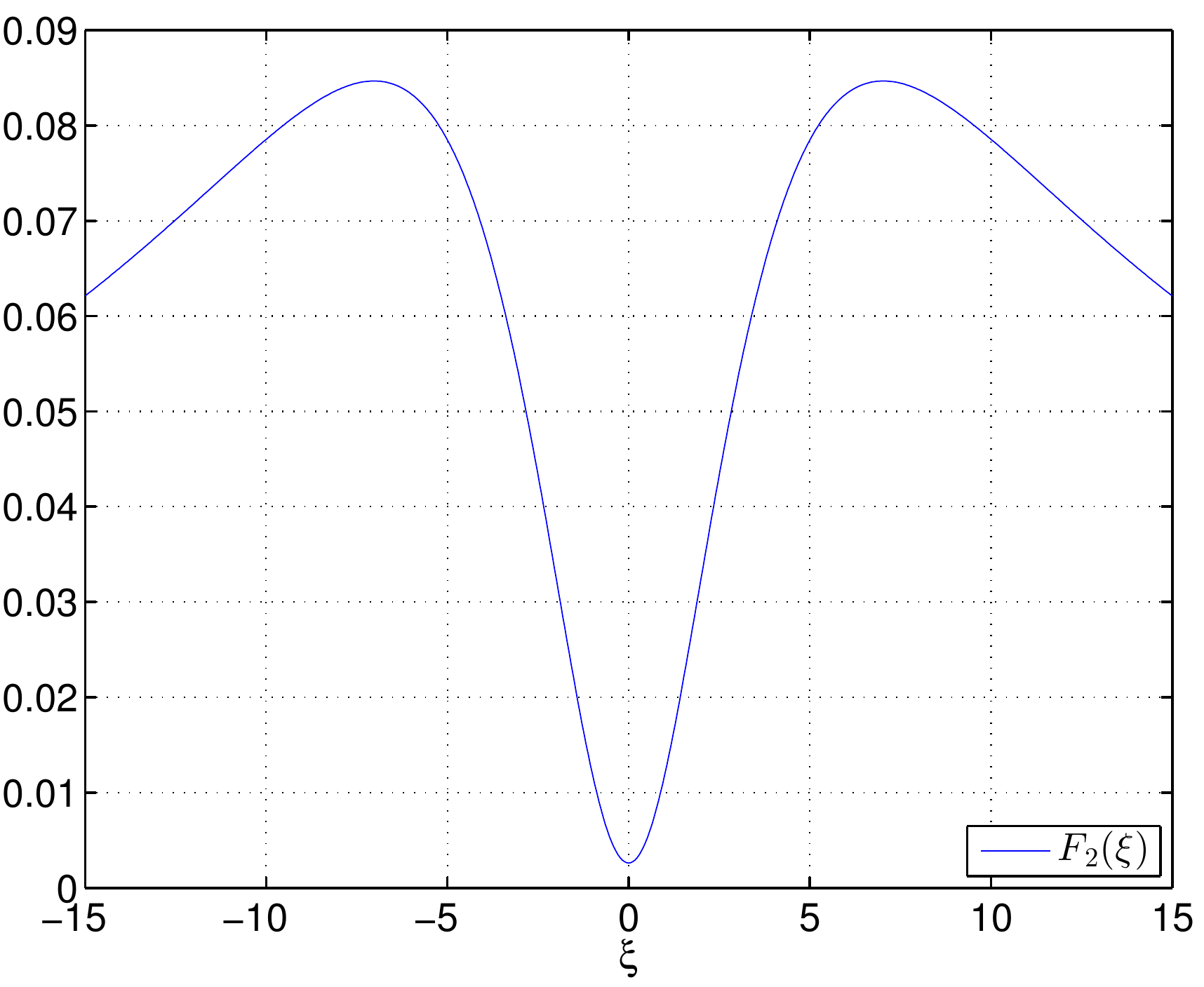}
\caption{(a,b,c): Depictions of $F_0(\xi)$, $F_1(\xi)$, $F_2(\xi)$ for
  $2\alpha=5\pi/18$ and $-15\leq\xi\leq 15$.}
\label{fig:Eessspec}
\end{figure}

\begin{table}
\caption{Positions and amplitudes of the bright plasmons in the $\bm{r}_1$-direction. }
\vspace{0.2cm}
\centering
\begin{tabular}{|c|l|l|}
\hline
$i$ & $x_i\approx$   & $u_iv_i\approx$  \\
\hline
 1 & 0.1935609900496035 & -0.0187559469606535\\
 2 & 0.1818566189413259 & -0.0382018970029643\\
 3 & 0.1727245662280549 & -0.048328578377405\\ 
 4 & 0.1658366392086451 & -0.047191751961888\\
 5 & 0.1610557751155232 & -0.035638113206132\\
 6 & 0.1584765425577683 & -0.014176894941617\\ 
\hline
\end{tabular}
\label{tab:plasmamp}
\end{table}

\subsubsection{The spectrum via the indicator function} 
\label{sec:specindic}

To locate the spectrum of
$K_n^\gamma\colon\mathcal{E}_n\to\mathcal{E}_n$, we want to compute
the limit $\lim_{y\to
  0^-}\kappa_n^{\mathrm{ind}}(\breve{u},\breve{v},x+iy)$ of the
indicator function, as described in Section~\ref{sec:indicator}.
Finite precision arithmetic constrains how small $y$ can be
before losing singular features of the spectrum such as eigenvalues.
The experiments in this section are carried out with $y_0=-10^{-70}$
and three different values of $y$: $y=-10^{-5}$, $y=-10^{-7}$,
$y=-10^{-11}$.

Figure~\ref{fig:indicatorzeromode}(a) shows the indicator function for
mode $n=0$. The only eigenvalue present is at $x=1$, which is an
eigenvalue of $K^\Gamma$ for every closed Lipschitz surface $\Gamma$
\cite[Lemma~3.1]{PP14}. Figure~\ref{fig:indicatorzeromode}(b) shows
the absolute deviation of the indicator function from the step
function
\begin{displaymath}
\mathrm{round}(4\kappa_0^{\mathrm{ind}}(\breve{u},\breve{v},x+iy))/4.
\end{displaymath}
The indicator function is very close to $1/2$ on the interval $(0,
0.861463506648456)$, which agrees with the interval $\Sigma_0$ of
Theorem~\ref{thm:Eessspec}. Note that
$\kappa_0^{\mathrm{ind}}(\breve{u},\breve{v},1+iy)=1$ to almost 13
digits for $y=-10^{-11}$.
\begin{figure}[ht]
\includegraphics[width =0.32\linewidth]{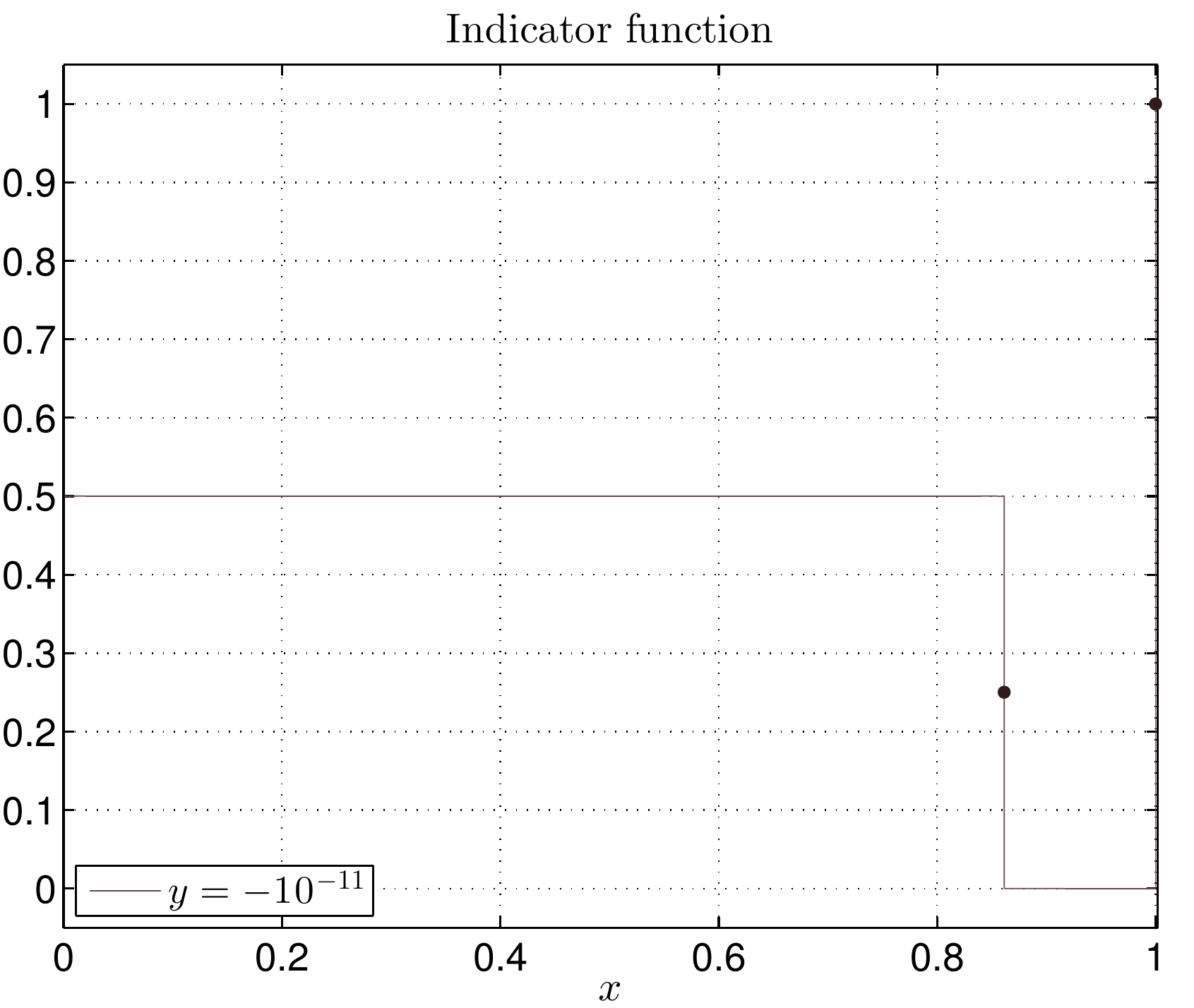}
\hfill
\includegraphics[width =0.32\linewidth]{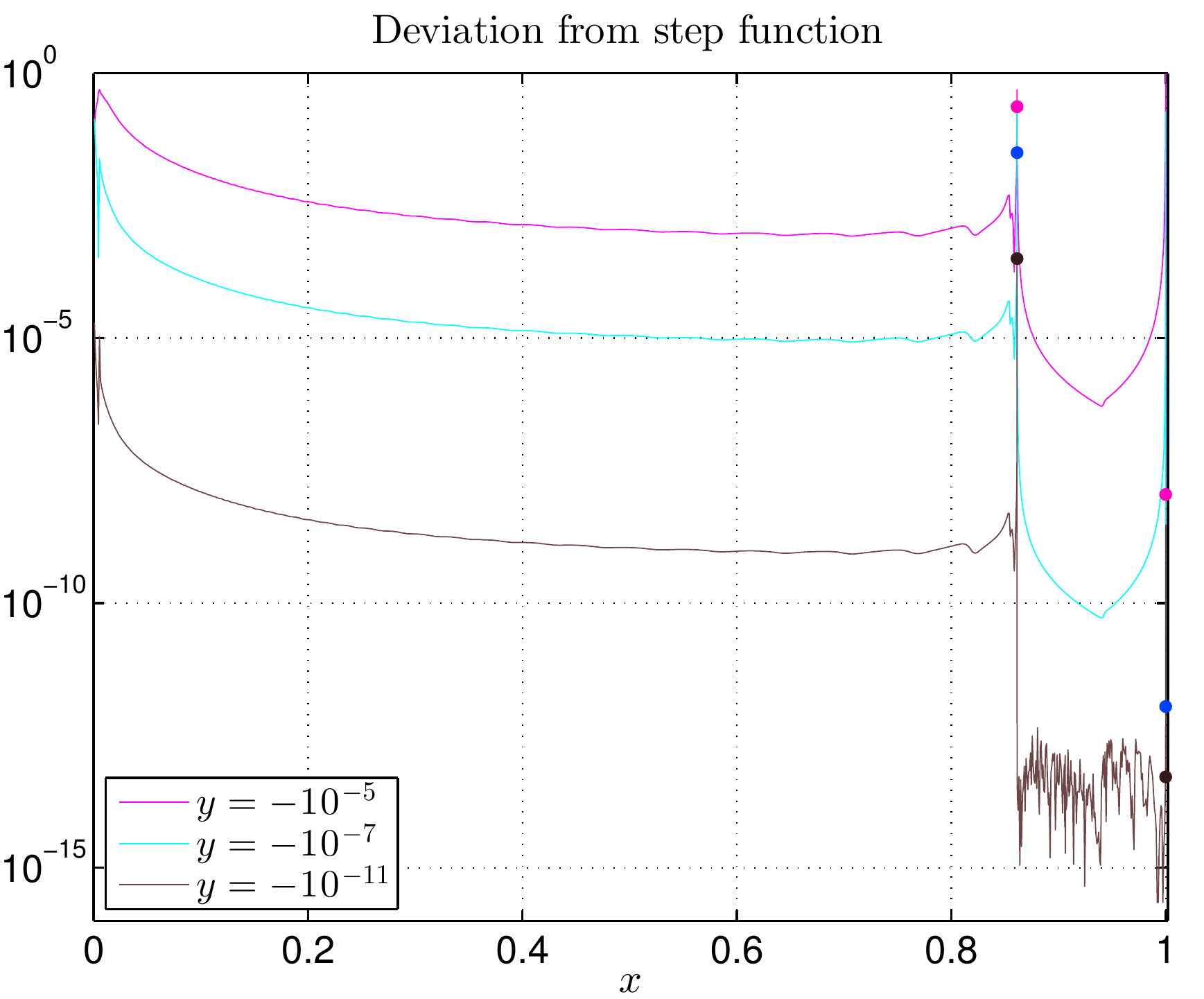} 
\hfill
\includegraphics[width =0.32\linewidth]{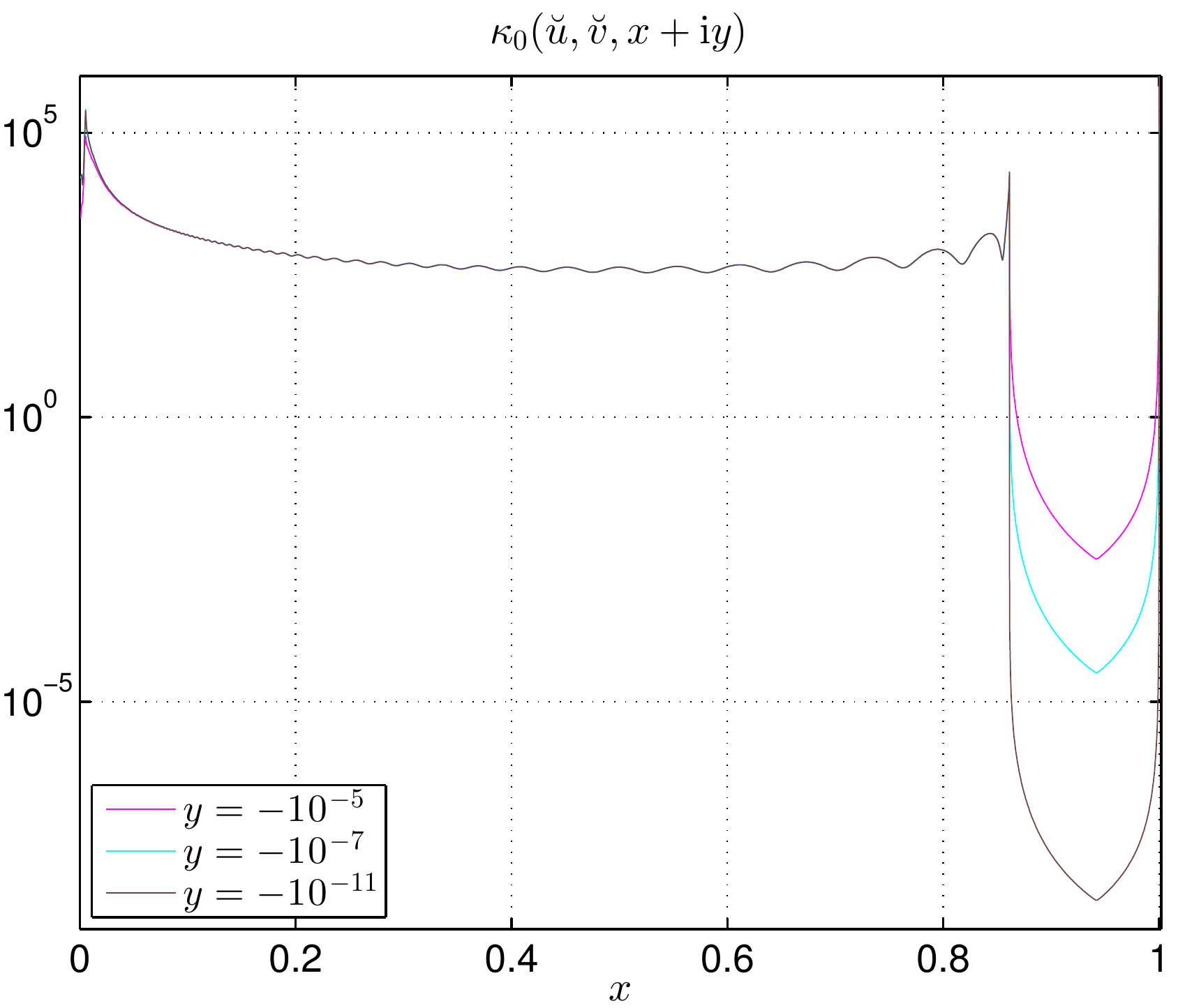}
\caption{(a): The indicator function 
  $\kappa_0^{\mathrm{ind}}(\breve{u},\breve{v},x+iy)$ for
  $2\alpha=5\pi/18$, $y_0=-10^{-70}$, and
  $y\in\{-10^{-5},-10^{-7},-10^{-11}\}$. Special values are marked
  with darker dots. (b): The absolute deviation from the corresponding
  step function. (c): $\kappa_0(\breve{u},\breve{v},x+iy)$. }
\label{fig:indicatorzeromode}
\end{figure}

Figure~\ref{fig:indicatorzeromode}(c) shows
$\kappa_0(\breve{u},\breve{v},x+iy)$, supported by 1156 data points
for each $y$. The appearance suggests that
$\kappa_0(\breve{u},\breve{v},x+iy)$ is uniformly bounded in $x$, $y$,
and $y_0$. If a spectral measure of $K_0^\gamma$ had any singular
parts in $[-1,1)$, either a point mass or a singular continuous part,
there would certainly be a $y_0$, an $x$, and functions $\breve{u}$
and $\breve{v}$ such that $\lim_{y\to
  0^-}\kappa_0(\breve{u},\breve{v},x+iy)=\infty$. There are no signs
of any overlooked singular points.

Finally, we remark that
$\kappa_0^{\mathrm{ind}}(\breve{u},\breve{v},F_0(0)+iy)$ has converged
to $1/4$ with more than 3 digits at $y=-10^{-11}$, see
Figure~\ref{fig:indicatorzeromode}(a,b). This suggests that
$\mu'_{\breve{u},\breve{v}}$ has a singularity in the right endpoint
of its support, a phenomenon that was analytically demonstrated for
certain 2D-domains with corners in \cite[Section~6.2]{HKL16}.

Figure~\ref{fig:indicatoracute}(a,b) shows the indicator function for
mode $n=1$. The interval of continuous spectrum coincides with
$\Sigma_1$. In addition there are 6 bright plasmons.
Figure~\ref{fig:indicatoracute}(c,d) shows mode $n=2$. The interval
coincides with $\Sigma_2$, and in this case there are 10 dark
plasmons. In view of \eqref{eq:polarmode0} and \eqref{eq:polarmode1}, the modes $|n| \geq 2$ never contribute to the spectral measures of the polarizability tensor. Therefore, the spectrum of $K_n^\gamma\colon\mathcal{E}_n\to\mathcal{E}_n$, $|n| \geq 2$, is always dark. Note also that the computed eigenvalues for $n=1$ and $n=2$ are embedded
in the continuous spectrum of $K_0^\gamma$ and therefore in the
continuous spectrum of $K^\Gamma$.
\begin{figure}[t!]
\centering
\includegraphics[width =0.32\linewidth]{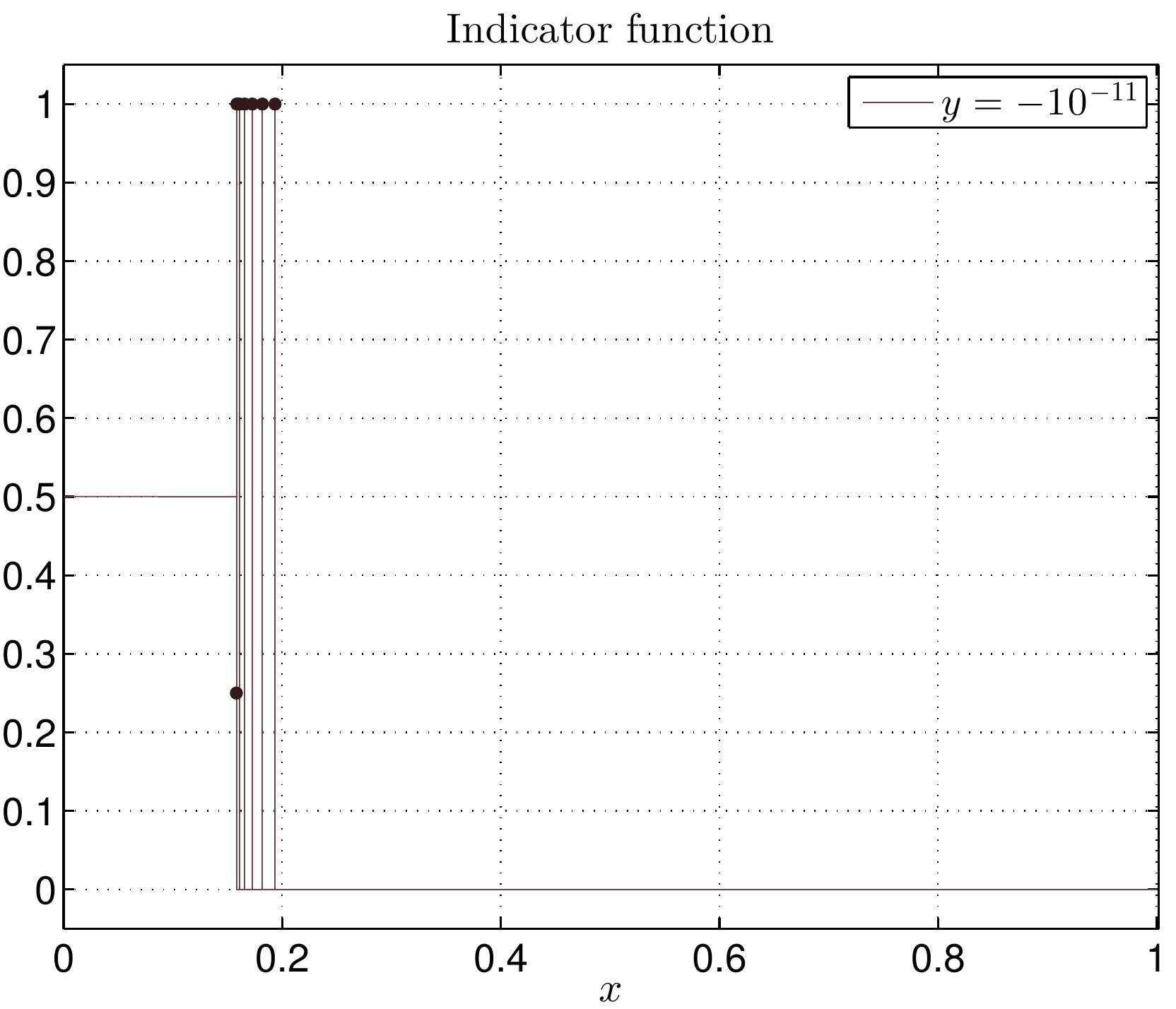}
\quad
\includegraphics[width =0.32\linewidth]{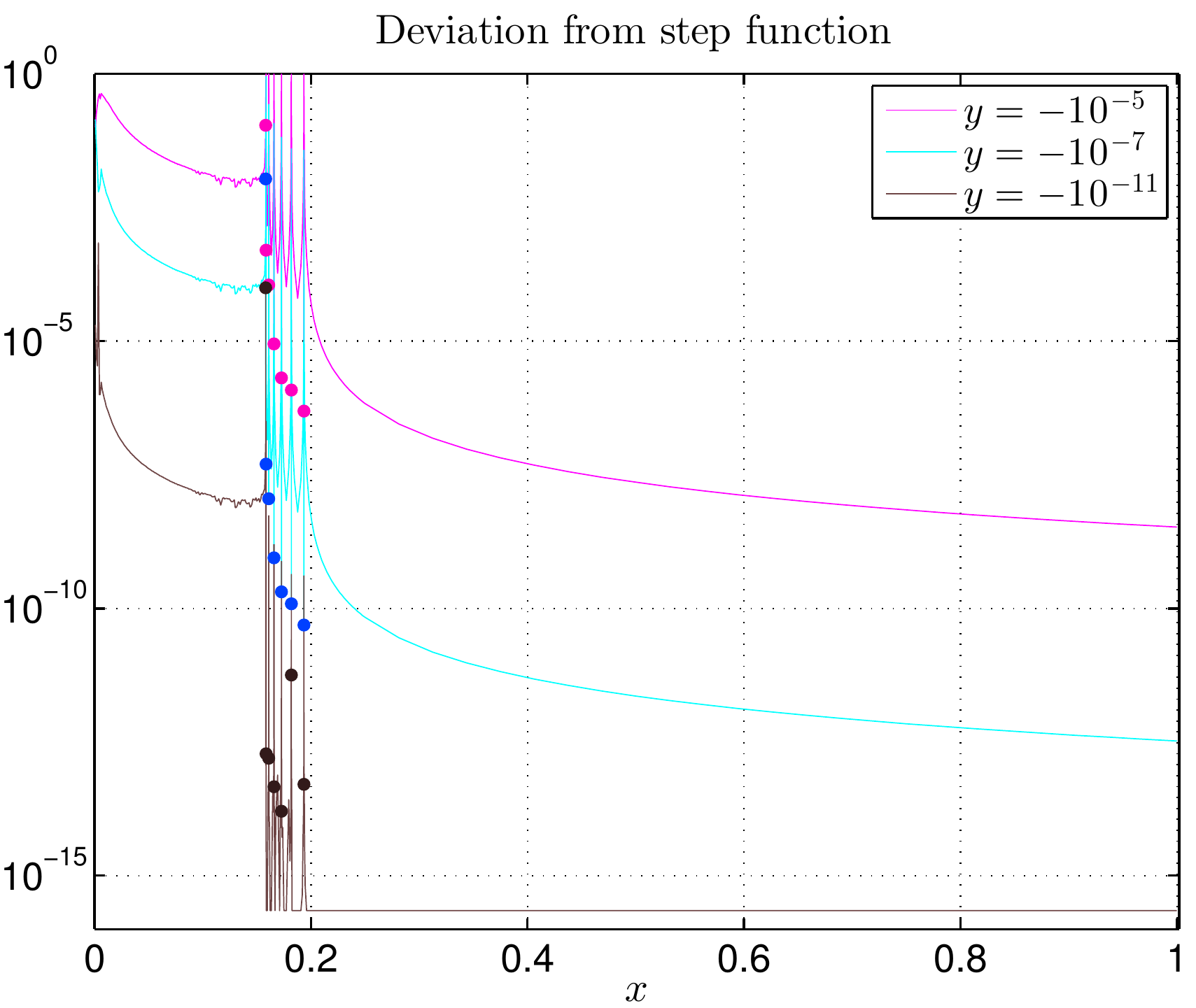}\\
\includegraphics[width =0.32\linewidth]{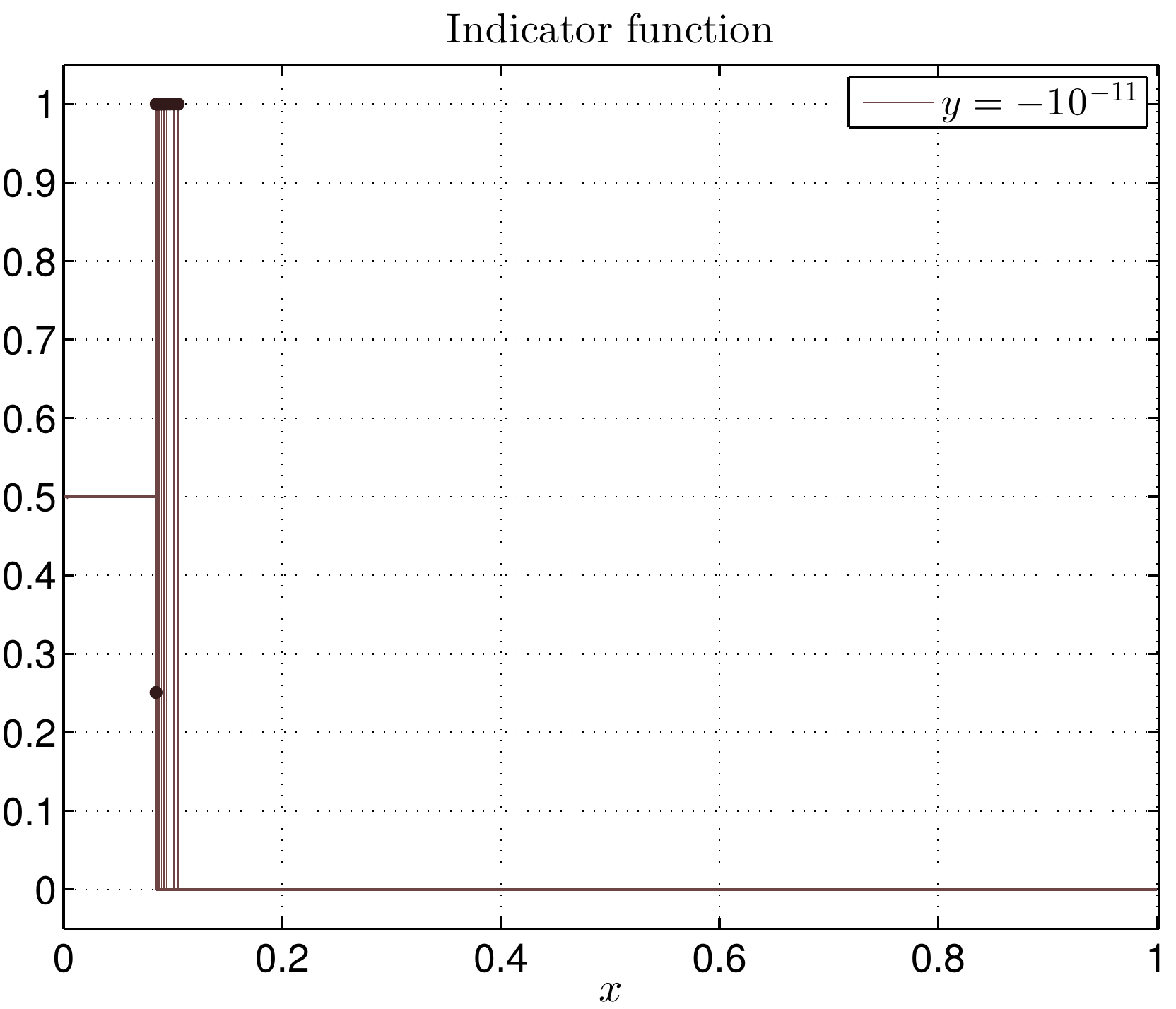}
\quad
\includegraphics[width =0.32\linewidth]{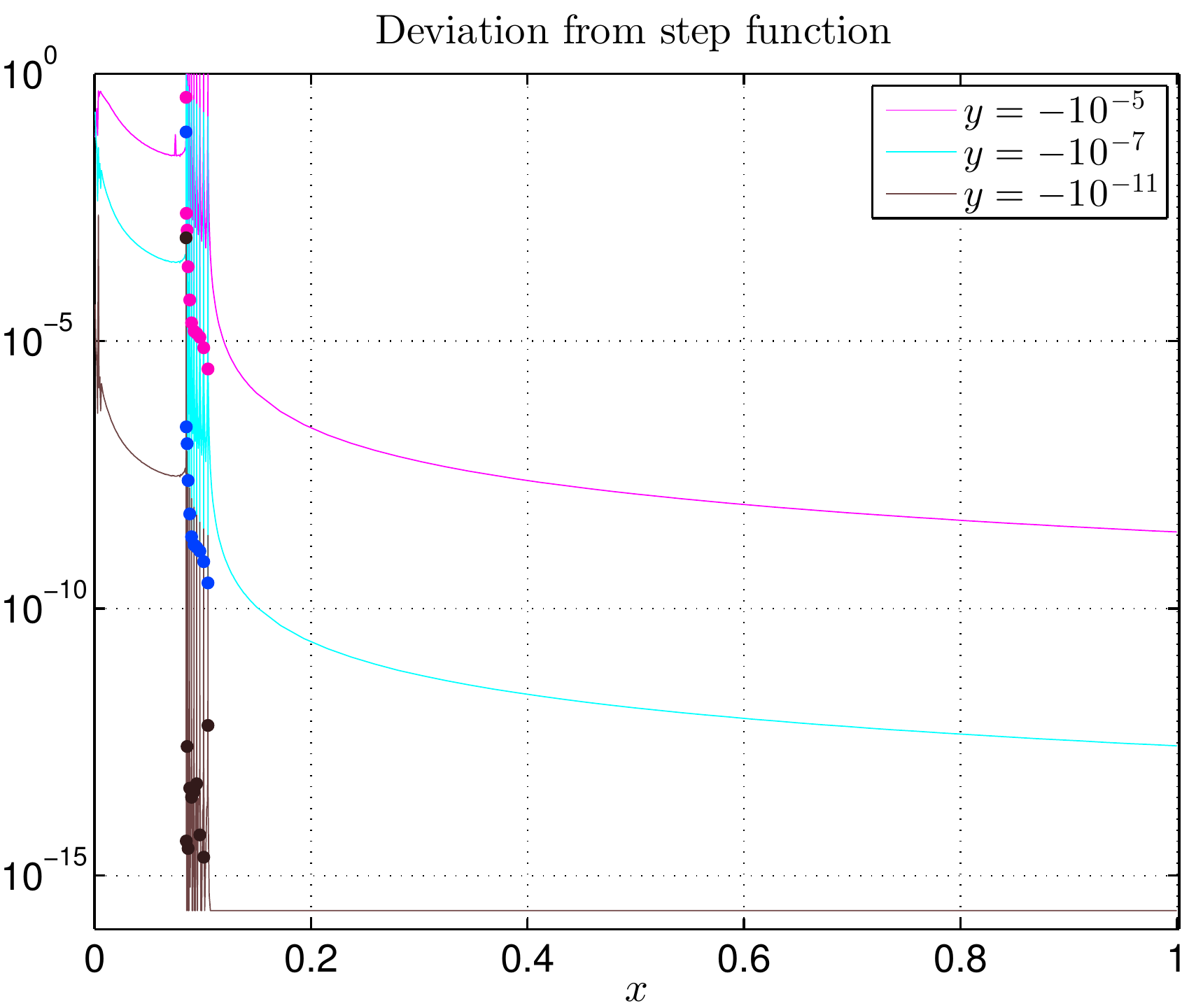}
\caption{(a,b): The indicator function
  $\kappa_1(\breve{u},\breve{v},x+iy)$ for $2\alpha=5\pi/18$. (c,d):
  $\kappa_2(\breve{u},\breve{v},x+iy)$.}
\label{fig:indicatoracute}
\end{figure}

\subsubsection{Results for a reflex angle} 
\label{sec:reflex}

We now carry out experiments for the reflex opening angle $2\alpha =
31\pi/18$ and mode $n=0$. The results are shown in
Figure~\ref{fig:reflex}. The non-discrete spectrum of $K_0^\gamma
\colon\mathcal{E}_0\to\mathcal{E}_0$ is as predicted by
Theorem~\ref{thm:Eessspec}. However, in contrast to the previous
sections, the reflex angle also exhibits a discrete spectrum
consisting of an infinite sequence of eigenvalues converging to $0$.
All of these eigenvalues, except $x=1$, are bright plasmons. Hence this geometry features an infinite number of bright plasmons.
\begin{figure}[h]
\centering
\includegraphics[width =0.32\linewidth]{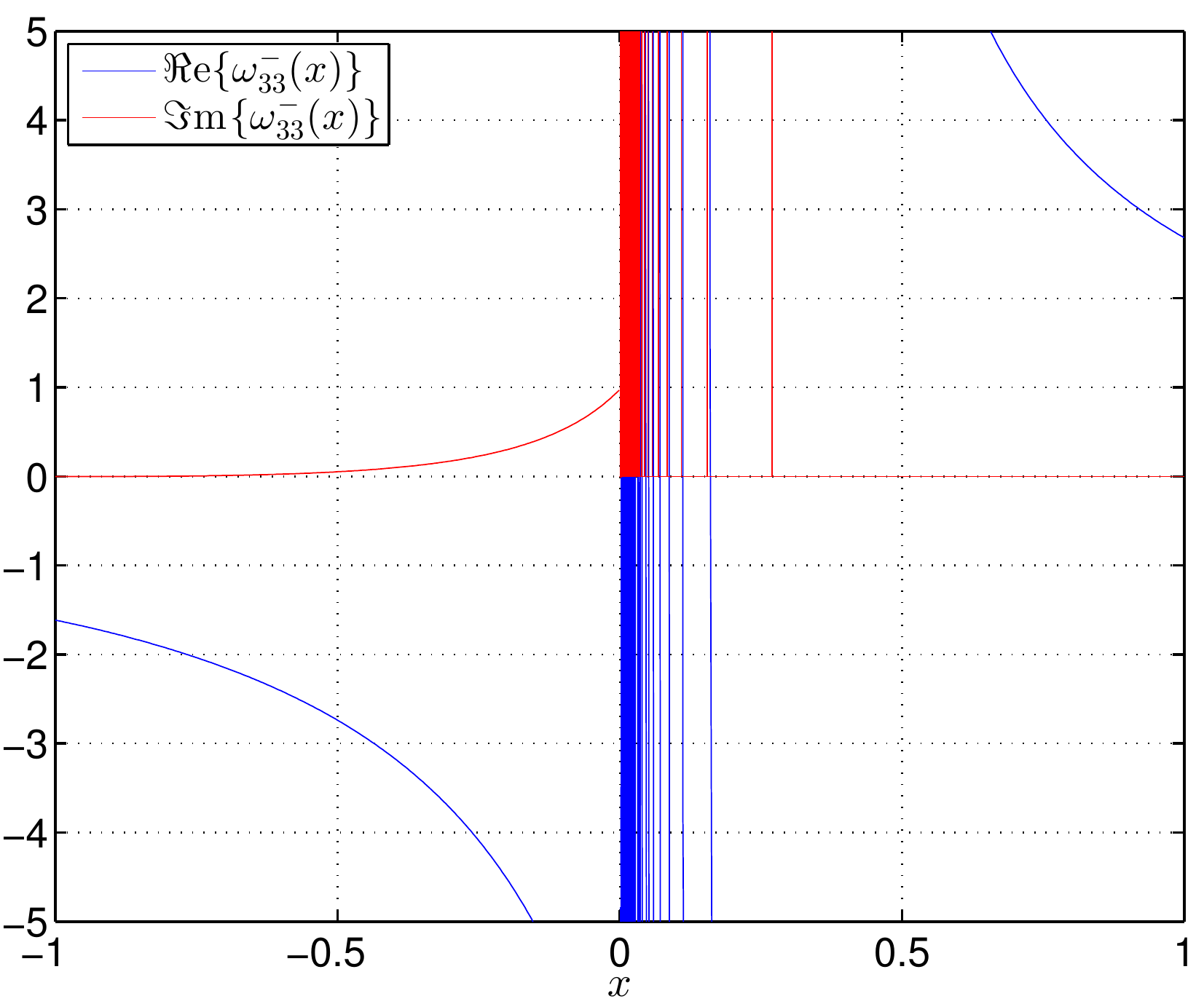}
\quad
\includegraphics[width =0.32\linewidth]{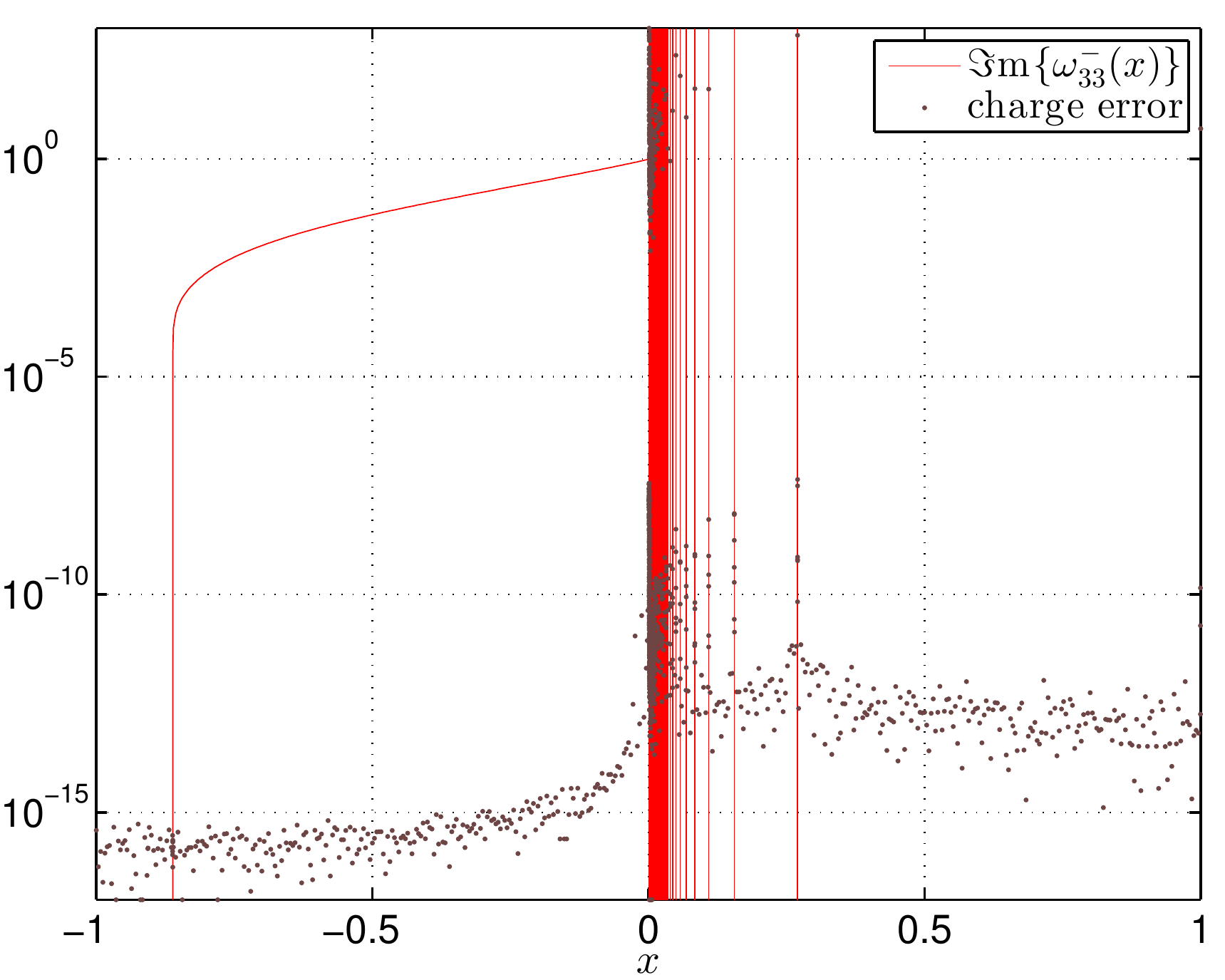} \\
\includegraphics[width =0.32\linewidth]{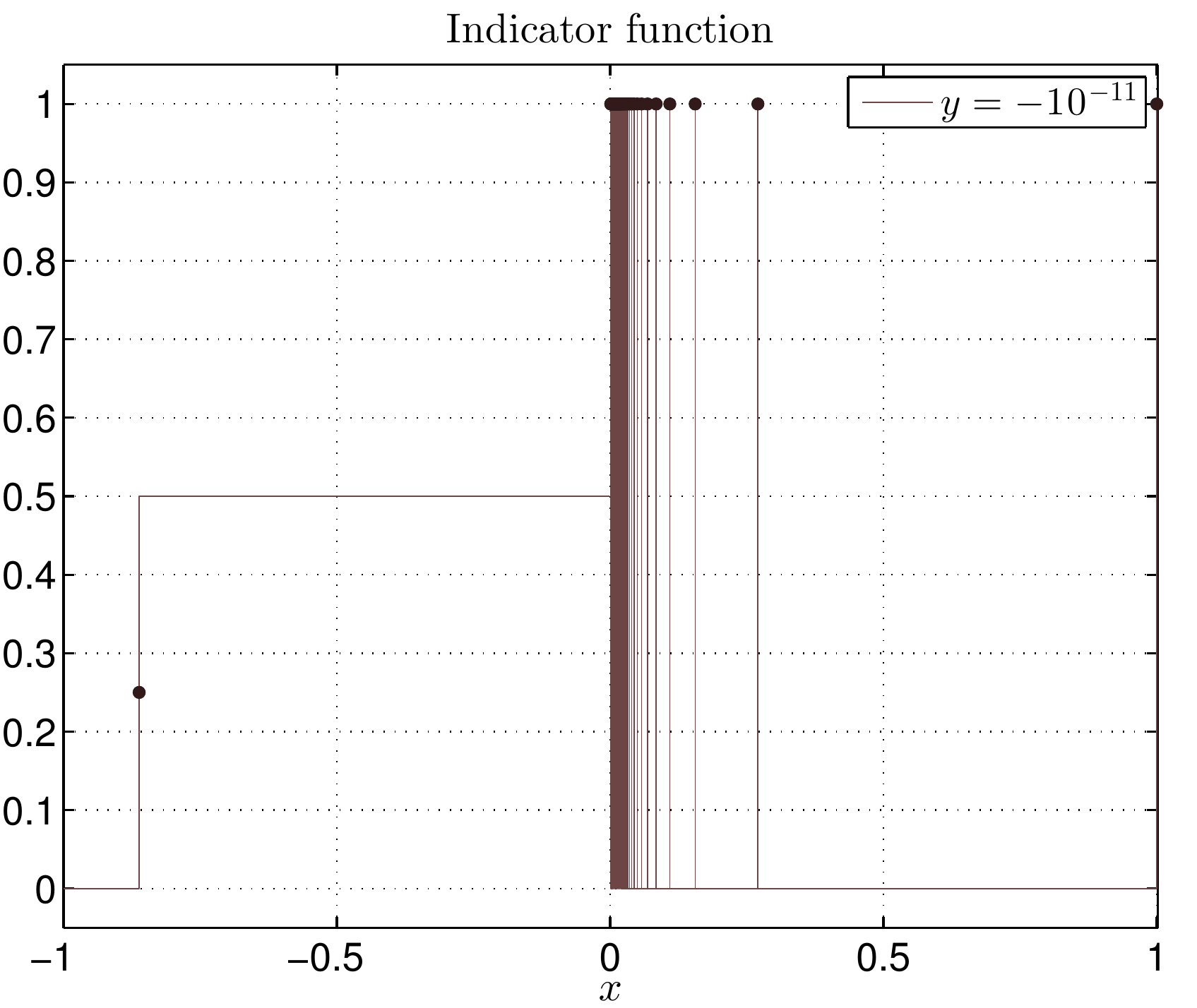}
\quad
\includegraphics[width =0.32\linewidth]{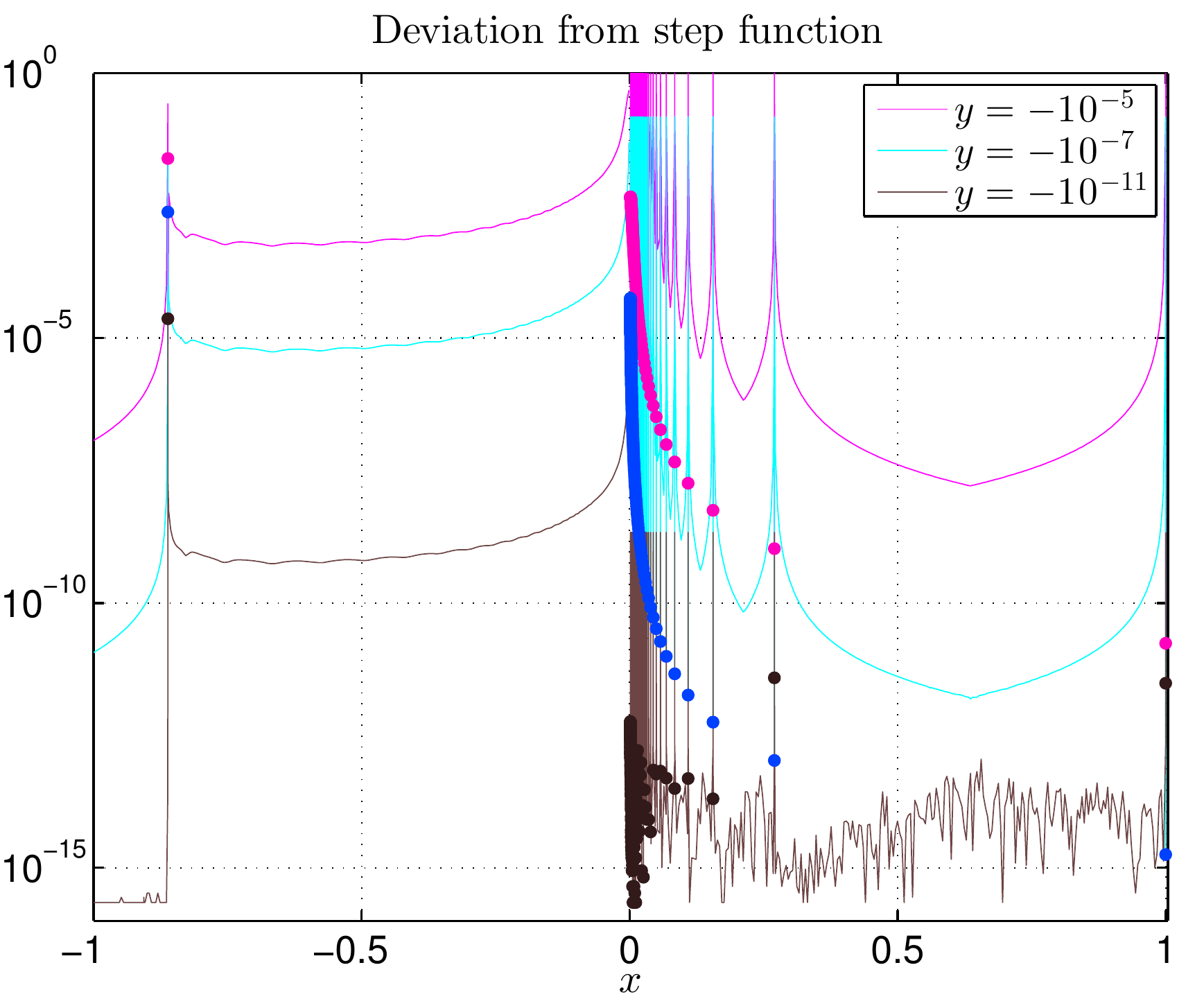}
\caption{Opening angle $2\alpha = 31\pi/18$. (a,b): Limit
  polarizability in the $\bm{r}_3$-direction. (c,d): The indicator
  function $\kappa_0^{\mathrm{ind}}(\breve{u},\breve{v},x+iy)$. 275
  eigenvalues are drawn.}
\label{fig:reflex}
\end{figure}

\appendix
\section{Explicit kernel formulas} \label{sec:appendix}

As in Section~\ref{sec:perturb}, let $\Gamma$ be a closed surface of revolution with a conical point of opening angle $2\alpha$, obtained by revolving a $C^5$-curve $\gamma$. We parametrize $\Gamma$ as before,
$$\bm{r}(t, \theta) = (\gamma_1(t)\cos\theta, \gamma_1(t) \sin\theta, \gamma_2(t)), \qquad \theta \in [0, 2\pi], \; 0 \leq t \leq 1,$$
In this Section we provide explicit formulas for the kernels $K_n^\alpha$, $K_n^\gamma$ and $S_n^\gamma$, defined in Sections~\ref{sec:infcone}, \ref{sec:L2spec}, and \ref{sec:indicator}, respectively. We use the first of these formulas to give the missing proof of Lemma~\ref{lem:specestimate}.

The formulas we are after can be read from \cite[Section~5.3]{YHM12}. We refer also to \cite{HK14}, where several typos of \cite{YHM12} are corrected. We have that
$$S_n^\gamma(t,t') = \frac{1}{\sqrt{2\pi^3\gamma_1(t)\gamma_1(t')}}\mathfrak{Q}_{n-1/2}(\chi),$$
and for $n \geq 0$ that
\begin{equation} 
\label{eq:Knexp}
K_n^\gamma(t, t') = \frac{1}{\sqrt{2\pi^3\gamma_1(t)\gamma_1(t')}} 
\left[\frac{\gamma_2'(t)}{2\gamma_1(t)|\gamma'(t)|}
\left(\mathfrak{Q}_{n-1/2}(\chi)+\mathfrak{R}_n(\chi)\right) 
- |\gamma(t)-\gamma(t')|K^\Gamma(t,0,t',0)\mathfrak{R}_n(\chi)\right], 
\end{equation}
where
$$\chi=1+\frac{|\gamma(t)-\gamma(t')|^2}{2\gamma_1(t)\gamma_1(t')},$$
and
$$\mathfrak{R}_n(\chi) = \frac{2n-1}{\chi+1}\left(\chi\mathfrak{Q}_{n-1/2}(\chi)-\mathfrak{Q}_{n-3/2}(\chi) \right).$$
To evaluate $K_n^\gamma$ for the negative indices $n < 0$, just note that $K^\gamma_n(t,t') = K^\gamma_{-n}(t,t')$.
In these formulas, $\mathfrak{Q}_{n-1/2}$ is an associated Legendre function of the second kind of half-integer degree,
$$\mathfrak{Q}_{n-1/2}(\chi) = \int_{-\pi}^\pi \frac{\cos(n\theta) \, d\theta}{\sqrt{8(\chi-\cos(\theta))}}.$$
By for example \cite[p. 153]{MagnusFormulas}, $\mathfrak{Q}_{n-1/2}$ has for $|\chi| > 1$ the series development
\begin{equation*}
\mathfrak{Q}_{n-1/2}(\chi) = \sqrt{\pi}\frac{\Gamma(n+1/2)}{\Gamma(n+1)}2^{-n-1/2} \chi^{-n-1/2} F(n/2+3/4, n/2 + 1/4, n+1, 1/\chi^2), \quad |\chi| > 1,
\end{equation*}
where $\Gamma$ denotes the usual gamma function and $F$ is the hypergeometric function
$$F(a, b, c, w) = \sum_{k=0}^\infty \frac{(a)_k (b)_k}{(c)_k} \frac{w^k}{k!}.$$
Here $(a)_k$ denotes the Pochhammer symbol \eqref{eq:pochhammer}.
We also note here that the associated Legendre function of the first kind, $P_\lambda^n(z)$, may be defined through the formula
$$P_\lambda^n(z) = \frac{1}{\Gamma(1-n)} \left( \frac{1+z}{1-z} \right)^{n/2} F(-\lambda, \lambda+1, 1-n, \frac{1-z}{2}), \quad |1-z| < 2.$$

We now supply the proof of Lemma~\ref{lem:specestimate}.
\newtheorem*{lem:specestimate}{Lemma~\ref{lem:specestimate}}
\begin{lem:specestimate} \it
For all $t > 0$ it holds that $K^\alpha_n(1,t) = tK^\alpha_n(t,1)$.
There is a constant $C > 0$, depending only on $\alpha$, such that 
\begin{equation}\label{eq:aKdecayinf}
|K_0^\alpha(t,1)| \leq \frac{C}{t^3}, \; t \geq \frac{3}{2}, \quad |K_{n}^\alpha(t,1)| \leq \frac{C}{t^{|n|+2}}, \; t \geq \frac{3}{2}, \; n\neq 0,
\end{equation}
and such that 
\begin{equation}\label{eq:aKdecayzero}
|K_0^\alpha(t,1)| \leq C, \; t \leq \frac{1}{2}, \quad |K_{n}^\alpha(t,1)| \leq Ct^{|n|-1}, \; t \leq \frac{1}{2}, \; n\neq 0.
\end{equation}
At $t=1$, $K_n^\alpha(t,1)$ has a logarithmic singularity: there is an analytic function $G(t)$ on $[1/2, 3/2]$ such that $K_n^\alpha(t,1) - \log |1-t| G(t)$ is analytic on $[1/2, 3/2]$.

Furthermore, for every $\beta$, $-1 < \beta < 2$, the functions $b_n(t) = t^\beta K_n^\alpha(t, 1)$ satisfy
\begin{equation} \label{eq:ahndecay}
 \|b_n\|_{L^1(dt/t)} \lesssim \frac{1}{n}.
 \end{equation}
\end{lem:specestimate}
\begin{proof}
Due to symmetry, we only have to consider the case $n \geq 0$. Equation \eqref{eq:Knexp} is valid also on the infinite cone $\Gamma_{\alpha}$, yielding that
\begin{equation} \label{eq:Kneq}
K_n^\alpha(t,1) = \frac{1}{(2\pi)^{3/2} \tan \alpha \sin \alpha} \frac{(2n \chi + 1) \mathfrak{Q}_{n-1/2}(\chi) - (2n-1)\mathfrak{Q}_{n-3/2}(\chi)}{t^{3/2}(\chi+1)},
\end{equation}
where 
$$\chi = \chi(t) = 1 + \frac{(t-1)^2}{2\sin^2(\alpha) t}.$$
When $t=1$ we instead have that $K_n^{\alpha}(1, t) = tK_n^{\alpha}(t,1)$.
We denote the coefficients of $\mathfrak{Q}_{n-1/2}$ by $q_{n,k}$,
\begin{align*}
\mathfrak{Q}_{n-1/2}(\chi) &= \sqrt{\pi}\frac{\Gamma(n+1/2)}{\Gamma(n+1)}2^{-n-1/2} \chi^{-n-1/2} F(n/2+3/4, n/2 + 1/4, n+1, 1/\chi^2) \\
&=: \chi^{-n-1/2} \sum_{k=0}^\infty q_{n, k} \chi^{-2k}, \qquad |\chi| > 1.
\end{align*}
 By Stirling's formula, they satisfy, for $n, k \geq 1$, that
$$4q_{n,k} = \sqrt{\pi}\frac{\Gamma(n+1/2)}{\Gamma(n+1)}2^{-n+3/2} \frac{(n/2 + 3/4)_k (n/2 + 1/4)_k}{(n+1)_k k!} = \frac{1}{\sqrt{n+k}\sqrt{k}} \frac{(n/2+k)^{n+2k}}{(n+k)^{n+k}k^k}\left(1 + O\left( \frac{1}{n} + \frac{1}{k} \right) \right).$$
We will also consider the coefficients $b_{n,k}$, defined by the equality
$$\mathfrak{Q}_{n-3/2}(\chi)  - \chi \mathfrak{Q}_{n-1/2}(\chi)= \chi^{-n+1/2} \sum_{k=0}^\infty b_{n,k} \chi^{-2k}.$$
From the formula for $q_{n,k}$, we deduce that
$$8 b_{n,k} = \frac{1}{\sqrt{n+k}\sqrt{k}} \frac{(n/2+k)^{n+2k}}{(n+k)^{n+k}k^k}\left(\frac{n}{n+k} + O\left( \frac{1}{n} + \frac{1}{k} \right) \right).$$
Consider the function
$$H(x,y) = \frac{(x/2+y)^{x+2y}}{(x+y)^{x+y}y^y}, \quad x,y > 0.$$
Then
$$ \partial_x H(x,y) = H(x,y)\log \frac{x+2y}{2x+2y} \leq 0, \quad \partial_y H(x,y) = H(x,y)\log \frac{(x+2y)^2}{4y(x+y)} \geq 0,$$
so that $H(n,k)$ is decreasing in $n$ and increasing in $k$. Since $\lim_{y \to \infty}H(x,y) = 1$ for every $x > 0$, it follows in particular that $H(n,k) \leq 1$ for all $n,k \geq 1$.

We consider first the case in which $t \geq 1 + \varepsilon$ or $t \leq 1 - \varepsilon$. The number $\varepsilon > 0$ will be chosen later depending only on $\alpha$. When $k \leq n$ we have, since $H(n, \cdot)$ is increasing, that
$$q_{n,k} \lesssim \frac{H(n,n)}{k} \lesssim \left(\frac{27}{32}\right)^n.$$
When $k \geq n$ we instead note that
$$q_{n,k} \lesssim \frac{1}{n}.$$
In total, we obtain that
$$n \mathfrak{Q}_{n-1/2}(\chi) \lesssim \chi^{-n-1/2}\sum_{k=0}^\infty \chi^{-2k} = \frac{\chi^{-n-1/2}}{1-\chi^{-2}}.$$
Since $\chi \simeq t$ for $t \geq 1+\varepsilon$ and $\chi \simeq t^{-1}$ for $t \leq 1 + \varepsilon$, the estimates \eqref{eq:aKdecayinf} and \eqref{eq:aKdecayzero} now follow from \eqref{eq:Kneq}.

To prove \eqref{eq:ahndecay} we have to work harder. Note first that 
$$\int_{1+ \varepsilon}^\infty  t^\beta\left|K_n^\alpha(t, 1)\right| \, dt \lesssim \frac{1}{n}, \quad \int_{0}^{1-\varepsilon} t^\beta \left|K_n^\alpha(t, 1)\right| \, dt \lesssim \frac{1}{n}$$
by  \eqref{eq:aKdecayinf} and \eqref{eq:aKdecayzero}. Hence we are left to consider $\int_{1 - \varepsilon}^{1+\varepsilon} |K_n^\alpha(t, 1)| \, dt$.
We let $n$ be fixed in our argument, but all implied constants will be independent of $n$. As before, for those $k$ such that $k \leq n$ it holds that
$$\max\{q_{n,k}, b_{n,k}\} \lesssim \frac{H(n,n)}{k}= \left(\frac{27}{32}\right)^n \frac{1}{k}.$$
When $2^{j-1}n \leq k \leq 2^{j}n$ for some $j \geq 1$ we similarly have that 
$$q_{n,k} \lesssim \frac{H(n,2^{j}n)}{k}= \left(\frac{\left(\frac{1}{2} + 2^j \right)^{1+2\cdot2^{j}}}{\left(1+2^j\right)^{1+2^j}\left(2^j\right)^{2^j}}\right)^n \frac{1}{k} \lesssim \left(1 - \frac{1}{2^{j+2}}\right)^n \frac{1}{k},$$
where the last inequality follows from the fact that
$$\frac{(1/2 + x)^{1+2x}}{(1+x)^{1+x} x^x} = 1 - \frac{1}{2x} + O\left(\frac{1}{x^2}\right), \quad x \to \infty.$$
For $b_{n,k}$ we have the better estimate
$$b_{n,k} \lesssim \frac{H(n,2^{j}n)n}{k^2} \lesssim \left(1 - \frac{1}{2^{j+2}}\right)^n \frac{n}{k^2}.$$

Recall that $\chi(t) = 1 + \frac{(t-1)^2}{2\sin^2(\alpha) t}$, so that $1/\chi(t)^2 \leq 1 - \frac{(t-1)^2}{2\sin^2 \alpha}$, when $|t-1| < \varepsilon$ and $\varepsilon$ is sufficiently small (depending on $\alpha$). Suppose that $\varepsilon^2 < 2\sin^2 \alpha$. Then note that
$$
\int_{1-\varepsilon}^{1+\varepsilon}\left(1 - \frac{(t-1)^2}{2\sin^2 \alpha}\right)^k \, dt = \frac{\sin \alpha}{\sqrt{2}} \int_{1-\frac{\varepsilon^2}{2\sin^2 \alpha}}^1 \frac{s^k}{\sqrt{1-s}} \, ds \lesssim \int_{0}^1 \frac{s^k}{\sqrt{1-s}} \, ds = \sqrt{\pi} \frac{\Gamma(k+1)}{\Gamma(k+3/2)} \simeq \frac{1}{\sqrt{k}}.
$$
Hence,
$$
 \int_{1-\varepsilon}^{1+\varepsilon} \chi(t)^{n-1/2} \sum_{k=2^{j-1}n}^{2^j n} q_{n,k} \chi(t)^{-2k} \, dt \lesssim \left(1 - \frac{1}{2^{j+2}}\right)^n \sum_{k=2^{j-1}n}^{2^j n} \frac{1}{k^{3/2}} \lesssim \left(1 - \frac{1}{2^{j+2}}\right)^n \frac{1}{2^{j/2}\sqrt{n}}.
$$
and
$$
 \int_{1-\varepsilon}^{1+\varepsilon} \chi(t)^{n-1/2} \sum_{k=2^{j-1}n}^{2^j n} b_{n,k} \chi(t)^{-2k} \, dt \lesssim \left(1 - \frac{1}{2^{j+2}}\right)^n n \sum_{k=2^{j-1}n}^{2^j n} \frac{1}{k^{5/2}} \lesssim \left(1 - \frac{1}{2^{j+2}}\right)^n \frac{1}{2^{3j/2}\sqrt{n}}.
$$
Therefore,
\begin{align*}
 \int_{1-\varepsilon}^{1+\varepsilon} \mathfrak{Q}_{n-1/2}(\chi(t)) \, dt &\lesssim \frac{1}{\sqrt{n}} \sum_{j=1}^\infty \left(1 - \frac{1}{2^{j+2}}\right)^n \frac{1}{2^{j/2}} \\
 &\lesssim \frac{1}{\sqrt{n}} \int_1^\infty \left(1 - \frac{1}{2^{x+2}}\right)^n 2^{-x/2} \, dx \lesssim \frac{1}{\sqrt{n}} \int_0^1 \frac{y^n}{\sqrt{1-y}} \, dy \simeq \frac{1}{n},
 \end{align*}
 and
 \begin{multline*}
  n\int_{1-\varepsilon}^{1+\varepsilon} \left|\mathfrak{Q}_{n-3/2}(\chi(t))  - \chi(t) \mathfrak{Q}_{n-1/2}(\chi(t))\right| \, dt \lesssim \sqrt{n} \sum_{j=1}^\infty \left(1 - \frac{1}{2^{j+2}}\right)^n \frac{1}{2^{3j/2}} \\
  \lesssim \sqrt{n} \int_1^\infty \left(1 - \frac{1}{2^{x+2}}\right)^n 2^{-3x/2} \, dx \lesssim \sqrt{n} \int_0^1 y^n\sqrt{1-y} \, dy = \frac{\sqrt{\pi}}{2} \frac{\sqrt{n}\Gamma(n+1)}{\Gamma(n+5/2)} \simeq \frac{1}{n}.
  \end{multline*}
In view of \eqref{eq:Kneq}, we have proven \eqref{eq:ahndecay}.

It only remains to show that $K_n^\alpha(t,1)$ has a logarithmic singularity at $t=1$. But this follows from the standard fact that the same is true of $\mathfrak{Q}_{n-1/2}(\chi(t))$. For example, when $\chi \simeq 1$, $\mathfrak{Q}_{n-1/2}(\chi)$ has the following series expansion \cite{mw:legendrelog},
\begin{multline*}
\mathfrak{Q}_{n-1/2}(\chi) = \left( \frac{1}{2}(\log(1+\chi) - \log(1-\chi)) - \psi(n+1/2)\right) F\left(-n+1/2, n+1/2, 1, \frac{1-\chi}{2}\right)  \\+ \sum_{k=0}^\infty \frac{(-n+1/2)_k (n+1/2)_k \psi(k+1)}{k!^2} \left(\frac{1-\chi}{2}\right)^k, \quad \left|\frac{1-\chi}{2}\right| < 1,
\end{multline*}
where $\psi$ denotes the digamma function.
\end{proof}

\section*{References}
\bibliographystyle{amsplain}
\bibliography{npconical} 

\end{document}